\documentclass[10pt]{amsart}

\usepackage{latexsym,exscale,enumerate,amsfonts,amssymb,xparse, mathtools}
\usepackage[normalem]{ulem}
\usepackage{amsmath,amsthm,amsfonts,amssymb,amscd, stmaryrd,textcomp}
\usepackage[bookmarks=false]{hyperref}
\usepackage{ulem}
\usepackage{bbold}
\usepackage[usenames,dvipsnames]{xcolor}
\usepackage{enumitem}
\usepackage{verbatim}

\numberwithin{equation}{section}

\newcommand{\arxiv}[1]{\href{https://arxiv.org/abs/#1}{\small  arXiv:#1}}

\theoremstyle{definition}
\newtheorem{thm}{Theorem}[section]
\newtheorem{cor}[thm]{Corollary}
\newtheorem{por}[thm]{Porism}

\newtheorem{lem}[thm]{Lemma}
\newtheorem{rem}[thm]{Remark}
\newtheorem{conv}[thm]{Convention}

\newtheorem{prop}[thm]{Proposition}
\newtheorem{defn}[thm]{Definition}
\newtheorem{example}[thm]{Example}

\newtheorem*{ack}{Acknowledgements}

\usepackage{tikz}
\usetikzlibrary{cd}
\usetikzlibrary{decorations.markings}
\usetikzlibrary{decorations.pathreplacing}
\usetikzlibrary{arrows,shapes,positioning}
\usetikzlibrary{decorations.pathmorphing}
\tikzset{anchorbase/.style={baseline={([yshift=-0.5ex]current bounding box.center)}},
tinynodes/.style={font=\tiny,text height=0.75ex,text depth=0.15ex},
smallnodes/.style={font=\scriptsize,text height=0.75ex,text depth=0.15ex},
>={Latex[length=1mm, width=1.5mm]},
overcross/.style={line width=4pt,color=white},
}
\tikzstyle directed=[postaction={decorate,decoration={markings,
    mark=at position #1 with {\arrow{>}}}}]
\tikzstyle rdirected=[postaction={decorate,decoration={markings,
    mark=at position #1 with {\arrow{<}}}}]

\colorlet{green}{black!30!green}
\definecolor{ourblue}{RGB}{109, 156, 179}

\newcommand{\Hom}{{\rm Hom}}

\newcommand{\End}{{\rm End}}

\renewcommand{\to}{\rightarrow}

\newcommand{\id}{{\rm id}}

\newcommand{\TL}{{\mathbf{TL}}}

\newcommand{\BMW}{\mathrm{BMW}}

\let\tilde=\widetilde
\def\C{{\mathbb C}}
\def\D{{\mathbb D}}
\def\N{{\mathbb N}}
\def\R{{\mathbb R}}
\def\Z{{\mathbb Z}}

\newcommand{\sln}[1][n]{\mathfrak{sl}_{#1}}
\newcommand{\gln}[1][n]{\mathfrak{gl}_{#1}}
\newcommand{\spn}[1][2n]{\mathfrak{sp}_{#1}}

\newcommand{\Rep}{\mathbf{Rep}}

\newcommand{\FRep}{\mathbf{FundRep}}
\newcommand{\SRep}{\mathbf{StdRep}}
\newcommand{\Web}{\mathbf{Web}}
\newcommand{\SWeb}{\mathbf{StdWeb}}
\newcommand{\FWeb}{\mathbf{FreeWeb}}
\newcommand{\HT}{\mathbf{HT}}
\newcommand\mydots{\makebox[1em][c]{$\cdot$\hfil$\cdot$\hfil$\cdot$}}

\newcommand{\ot}{\otimes}
\newcommand{\HB}{\mathbb{H}}
\DeclareMathOperator{\qdim}{qdim}
\DeclareMathOperator{\Kar}{Kar}

\newcommand{\CS}{\mathcal{C}}
\newcommand{\SG}{\mathfrak{S}}
\newcommand{\ZA}{\mathcal{A}}

\begin{document}
%

\title[Type $C$ Webs]{Type $C$ Webs}

\author{Elijah Bodish}
\address{Department of Mathematics, University of Oregon,
Fenton Hall, Eugene, OR 97403-1222, USA}
\email{\{ebodish,belias\}@uoregon.edu}

\author{Ben Elias}

\author{David E. V. Rose}
\address{Department of Mathematics, University of North Carolina, 
Phillips Hall, CB \#3250, UNC-CH, 
Chapel Hill, NC 27599-3250, USA}
\email{davidrose@unc.edu, ltatham@live.unc.edu}

\author{Logan Tatham}

\begin{abstract}
We define a $\C(q)$-linear pivotal category $\Web(\spn)$ and prove that it 
is equivalent to the full subcategory of finite-dimensional representations of 
$U_q(\spn)$ tensor-generated by the fundamental representations. 
This answers the type $C$ case of the main open problem from 
Kuperberg's 1996 paper \emph{Spiders for rank {$2$} Lie algebras}.
\end{abstract}

\maketitle

%
\section{Introduction}
%

In his seminal 1996 paper, Kuperberg \cite{Kup} gives a diagrammatic presentation for a category that encodes the representation theory of any rank two simple complex Lie algebra,  and its associated quantum group.
More precisely, Kuperberg gives a presentation of the \emph{fundamental subcategory}, the full subcategory whose objects are iterated tensor products of fundamental representations, as a braided pivotal category. 
Further, he poses the following ``main open problem'': to give analogous presentations of the fundamental subcategories for simple complex Lie algebras and quantum groups of rank larger than two 
(see \cite[\S 8.1]{Kup} and also \cite[Problem 12.18]{Probs} for a restatement).
In 2012, Cautis-Kamnitzer-Morrison\footnote{Cautis-Kamnitzer-Morrison \cite{CKM} define a category of webs which is equivalent to the fundamental subcategory for $U_q(\sln)$. 
Implicit in their work is a very similar (and even simpler) web category for $U_q(\gln)$.
For the sake of simplicity in some discussions, we will talk about $\gln$ webs rather than $\mathfrak{sl}_n$ webs.} 
resolved the type $A$ case of this problem \cite{CKM}.
In the present paper, we solve this problem in type $C$.

\subsection{Statement of results}

We begin by introducing the titular diagrammatic category.

\begin{defn}\label{def:webs}
Let $\Web(\spn)$ be the $\C(q)$-linear pivotal category defined by the following presentation. 
The objects are generated monoidally by self-dual objects $\{1,\ldots,n\}$. In addition to the cap/cup (co)unit morphisms implicit in the pivotal structure, the morphisms are generated by
\begin{equation}\label{eq:WebGen}
\begin{tikzpicture}[scale =.5, smallnodes,anchorbase]
	\draw[very thick] (0,0) node[below]{$1$} to [out=90,in=210] (.5,.75);
	\draw[very thick] (1,0) node[below]{$k$} to [out=90,in=330] (.5,.75);
	\draw[very thick] (.5,.75) to (.5,1.5) node[above=-2pt]{$k{+}1$};
\end{tikzpicture}
\qquad , \qquad
\begin{tikzpicture}[scale =.5, smallnodes,anchorbase]
	\draw[very thick] (0,0) node[below]{$k$} to [out=90,in=210] (.5,.75);
	\draw[very thick] (1,0) node[below]{$1$} to [out=90,in=330] (.5,.75);
	\draw[very thick] (.5,.75) to (.5,1.5) node[above=-2pt]{$k{+}1$};
\end{tikzpicture}
\end{equation}
for $k \in \{1,\ldots,n-1\}$. One then takes the quotient by the tensor ideal generated by the following (local) relations:
\begin{equation}\label{eq:spn}
\begin{gathered}
(\text{\ref{eq:spn}a}) \;\;
\begin{tikzpicture}[scale =.75, tinynodes,anchorbase]
	\draw[very thick] (0,0) node[left,xshift=-8pt]{$1$} circle (.5);
\end{tikzpicture}
= -\frac{[n][2n+2]}{[n{+}1]}
\quad , \quad
(\text{\ref{eq:spn}b}) \;\;
\begin{tikzpicture}[scale=.175,tinynodes,anchorbase]
	\draw [very thick] (0,-2.75) to [out=30,in=0] (0,.75);
	\draw [very thick] (0,-2.75) to [out=150,in=180] (0,.75);
	\draw [very thick] (0,-4.5) node[below,yshift=2pt]{$2$} to (0,-2.75);
\end{tikzpicture}
= 0
\quad , \quad
(\text{\ref{eq:spn}c}) \;\;
\begin{tikzpicture}[scale=.175,tinynodes, anchorbase]
	\draw [very thick] (0,.75) to (0,2.5) node[above,yshift=-3pt]{$k$};
	\draw [very thick] (0,-2.75) to [out=30,in=330] node[right,xshift=-2pt]{$k{-}1$} (0,.75);
	\draw [very thick] (0,-2.75) to [out=150,in=210] node[left,xshift=2pt]{$1$} (0,.75);
	\draw [very thick] (0,-4.5) node[below,yshift=2pt]{$k$} to (0,-2.75);
\end{tikzpicture}
= [k]
\begin{tikzpicture}[scale=.175, tinynodes, anchorbase]
	\draw [very thick] (0,-4.5) node[below,yshift=2pt]{$k$} to (0,2.5);
\end{tikzpicture} \\
(\text{\ref{eq:spn}d}) \;\;
\begin{tikzpicture}[scale=.2, xscale=-1,tinynodes, anchorbase]
	\draw [very thick] (-1,-1) node[below,yshift=2pt]{$1$} to [out=90,in=210] (0,.75);
	\draw [very thick] (1,-1) node[below,yshift=2pt]{$k$} to [out=90,in=330] (0,.75);
	\draw [very thick] (3,-1) node[below,yshift=2pt]{$1$} to [out=90,in=330] (1,2.5);
	\draw [very thick] (0,.75) to [out=90,in=210] (1,2.5);
	\draw [very thick] (1,2.5) to (1,4.25) node[above,yshift=-3pt]{$k{+}2$};
\end{tikzpicture}
=
\begin{tikzpicture}[scale=.2, tinynodes, anchorbase]
	\draw [very thick] (-1,-1) node[below,yshift=2pt]{$1$} to [out=90,in=210] (0,.75);
	\draw [very thick] (1,-1) node[below,yshift=2pt]{$k$} to [out=90,in=330] (0,.75);
	\draw [very thick] (3,-1) node[below,yshift=2pt]{$1$} to [out=90,in=330] (1,2.5);
	\draw [very thick] (0,.75) to [out=90,in=210] (1,2.5);
	\draw [very thick] (1,2.5) to (1,4.25) node[above,yshift=-3pt]{$k{+}2$};
\end{tikzpicture}
\quad , \quad
(\text{\ref{eq:spn}e}) \;\;
\begin{tikzpicture}[scale=.4, rotate=90, tinynodes, anchorbase]
	\draw[very thick] (-1,0) node[below,yshift=2pt,xshift=2pt]{$k$} to (0,1);
	\draw[very thick] (1,0) node[above,yshift=-4pt,xshift=2pt]{$1$} to (0,1);
	\draw[very thick] (0,2.5) to (-1,3.5) node[below,yshift=2pt,xshift=-2pt]{$k$};
	\draw[very thick] (0,2.5) to (1,3.5) node[above,yshift=-4pt,xshift=-2pt]{$1$};
	\draw[very thick] (0,1) to node[below,yshift=2pt]{$k{+}1$} (0,2.5);
\end{tikzpicture}
=
\begin{tikzpicture}[scale=.4, tinynodes, anchorbase]
	\draw[very thick] (-1,0) node[below,yshift=2pt,xshift=-2pt]{$k$} to (0,1.5);
	\draw[very thick] (1,0) node[below,yshift=2pt,xshift=2pt]{$k$} to (0,1.5);
	\draw[very thick] (-.7,.5) to node[below,yshift=4pt]{$_{k{-}1}$} (.7,.5);
	\draw[very thick] (0,2.5) to (-1,3.5) node[above,yshift=-4pt,xshift=-2pt]{$1$};
	\draw[very thick] (0,2.5) to (1,3.5) node[above,yshift=-4pt,xshift=2pt]{$1$};
	\draw[very thick] (0,1.5) to node[right, xshift=-2pt]{$2$} (0,2.5);
\end{tikzpicture} \!\!
-\frac{[n{-}k]}{[n{-}k{+}1]}
\begin{tikzpicture}[scale=.4, rotate=90, tinynodes, anchorbase]
	\draw[very thick] (-1,0) node[below,yshift=2pt,xshift=2pt]{$k$} to (0,1);
	\draw[very thick] (1,0) node[above,yshift=-4pt,xshift=2pt]{$1$} to (0,1);
	\draw[very thick] (0,2.5) to (-1,3.5) node[below,yshift=2pt,xshift=-2pt]{$k$};
	\draw[very thick] (0,2.5) to (1,3.5) node[above,yshift=-4pt,xshift=-2pt]{$1$};
	\draw[very thick] (0,1) to node[below,yshift=2pt]{$k{-}1$} (0,2.5);
\end{tikzpicture}
+
\frac{[n{-}k]}{[n]}
\begin{tikzpicture}[scale=.4, tinynodes, anchorbase]
	\draw[very thick] (-1,0) node[below,yshift=2pt]{$k$} to [out=90,in=180] (0,1) 
		to [out=0,in=90] (1,0);
	\draw[very thick] (-1,3) node[above,yshift=-3pt]{$1$} to [out=270,in=180] (0,2)
		to [out=0,in=270] (1,3);
\end{tikzpicture}
\end{gathered}
\end{equation}
\end{defn}

Graphs built from the generators in \eqref{eq:WebGen} (and identity and (co)unit morphisms) are called $\spn$ \emph{webs}. 
In \S\ref{sec:more} we will expand the allowable webs, using these generators to define other kinds of trivalent vertices which we also allow in a web.

\begin{rem} We will allow web edges to be labeled by any integers to streamline some formulae.
By convention, any $0$-labeled edges should be erased, and any web containing labels not in $\{0, 1, \ldots, n\}$ is set equal to zero.
In particular, we impose the $k=n$ case of (\text{\ref{eq:spn}e}) as a defining relation, wherein the left-hand side is zero by convention.
Meanwhile, the $k=0$ case of (\text{\ref{eq:spn}e}) already holds tautologically.
\end{rem}

Next, we introduce the representation-theoretic categories of interest.

\begin{defn} Let $\FRep(U_q(\spn))$ denote the \emph{fundamental subcategory} of representations of the quantum group $U_q(\spn)$ over $\C(q)$. This is the full (monoidal)
subcategory whose objects are iterated tensor products of fundamental representations. We let $V_{\varpi_k}$ denote the $k$th fundamental representation, 
i.e. the irreducible representation of $U_q(\spn)$ with highest weight $\varpi_k$.
For example, $V_{\varpi_1}$ is the standard representation, having dimension $2n$, and $V_{\varpi_k}$ is a direct summand of $V_{\varpi_1}^{\otimes k}$, for $1 \le k \le n$. \end{defn}

The main result of this paper is the following.

\begin{thm} \label{thm:main}
There is an equivalence of $\C(q)$-linear ribbon categories 
\[
\Phi \colon \Web(\spn) \to \FRep(U_q(\spn))
\]
sending $k \mapsto V_{\varpi_k}$.
\end{thm}

Every (type I) finite-dimensional
irreducible representation\footnote{As is typical in the field, we tacitly ignore non-type-I representations throughout this paper.}
of $U_q(\spn)$ is a direct summand of some object in $\FRep(U_q(\spn))$, 
therefore the Karoubi envelope of $\Web(\spn)$ is equivalent to the category of all finite-dimensional representations of $U_q(\spn)$.
Hence, we view $\Web(\spn)$ as providing a diagrammatic description of this latter category.
Implicit in Theorem \ref{thm:main} is
a description of a braiding on $\Web(\spn)$ that is compatible with the ribbon structure on the category $\Rep(U_q(\spn))$.
This allows for the use of $\Web(\spn)$ in the study of the $\spn$ link invariant and its associated TQFT.

\subsection{The goal of this paper}

This paper grew out of an ongoing study of diagrammatic presentations of categories of quantum group representations. 
Along the way, we noticed that it was possible to prove Theorem \ref{thm:main} with surprising efficiency. 
By pairing previous work (of many authors) on quantum Brauer-Schur-Weyl duality with a handful of crucial skein-theoretic arguments, 
we are able to provide a basis for $\Hom_{\Web(\spn)}(1^{\otimes k}, 1^{\otimes l})$ for all $k, l \ge 0$. From this we can deduce that $\Phi$ is fully faithful.
The goal of this paper is to give this efficient proof, and not much more.
In particular, here are some things we shamelessly do not provide in this paper.
\begin{itemize}
	\item an explicit description of where the functor $\Phi$ sends the generating morphisms,
	\item a coherent presentation (i.e. additional useful formulas) that further facilitates the simplification of diagrams,
	\item a basis of morphisms for $\Hom$-spaces in $\Web(\spn)$, aside from the $\Hom$-spaces mentioned above,
	\item cellular bases for $\Hom$-spaces, adapted to the monoidal structure as in \cite{EliasLL},
	\item a proof that $\Web(\spn)$ is equivalent to $\FRep(U_q(\spn))$ in their integral forms, rather than over $\C(q)$.
\end{itemize}
However, all of these useful tools will be provided in the sequel to this paper \cite{BERT2}. 
In particular, we will construct a \emph{double ladders basis} of morphism spaces in $\Web(\spn)$
that is analogous to the basis constructed in type $A$ by the second-named author in \cite{EliasLL}, 
and which extends the $C_2$ double ladders basis constructed by the first-named author \cite{Bodish}. 

Despite our attempts at minimalism, we felt it was worthwhile to comment further on our diagrammatic calculus;
thus, we do so in \S\ref{sec:more}. 
Although we will not use them in the proof of our main theorem, 
we introduce additional trivalent vertices and provide some relations that are useful for computations.
We also introduce some convenient pieces of notation that are used in our proofs.

In \S\ref{sec:background}, we record standard facts about the representation theory of $U_q(\spn)$ and recall various combinatorial results on type $C$ representation theory
that we will use. In \S\ref{sec:outline}, we present an outline of the proof of our main result. The expert reader is invited to skip ahead to this section, which gets right ``to
the point.'' Finally, \S\ref{sec:proof} contains the proof of Theorem \ref{thm:main}.

\begin{ack}
E.B. and B.E. were supported on this project by NSF CAREER grant DMS-1553032, and B.E. was supported by the Institute of Advanced Study via NSF grant DMS-1926686.
D.E.V.R. and L.T. were partially supported by Simons Collaboration Grant 523992: \emph{Research on knot invariants, representation theory, and categorification}. 
The authors would like to thank Jon Brundan, Victor Ostrik, Noah Snyder, Daniel Tubbenhauer, and Geordie Williamson for generally useful conversations.
\end{ack}

%
\section{Background}\label{sec:background}
%

We begin with some background, mostly focusing on quantum group representation theory in type $C$.
In \S\ref{subsec:QGtypeC} we recall standard facts about the quantum group $U_q(\spn)$. 
In \S\ref{subsec:qdim} we recall the essential features of the Snyder-Tingley ribbon element from \cite{ST}, 
and give a simple formula for the quantum dimensions of fundamental representations. 
In \S\ref{subsec:previousC} we discuss Brauer-Schur-Weyl duality.
Finally, in \S\ref{subsec:connections} we discuss the relation between $\Web(\spn)$ and various diagrammatic categories 
that have previously appeared in the literature.

\subsection{Quantum groups (in type $C$)} \label{subsec:QGtypeC}

We review the definition of the quantum group $U_q(\spn)$. Standard references are \cite{CP} and \cite{JantzenQgps}.

The Lie algebra $\spn$ gives rise to a root system $\Phi_{C_n}$ and weight lattice 
$X= \oplus_{i=1}^n \mathbb{Z}\epsilon_i$. 
We let $\alpha_i = \epsilon_i- \epsilon_{i+1}$, for $i= 1, \ldots, n-1$, and $\alpha_n = 2\epsilon_n$ be our simple roots.
The positive roots are then $\Phi_+ = \lbrace \epsilon_i- \epsilon_j, \epsilon_i + \epsilon_j, \text{for} \ i< j, \ 2\epsilon_i\rbrace$. 
We denote by $(-, -)$, the usual symmetric bilinear form given by  $(\epsilon_i, \epsilon_j) = \delta_{ij}$, 
and write $\langle\alpha_i^{\vee}, \lambda\rangle= (2\alpha_i/(\alpha_i, \alpha_i), \lambda)$ for the pairing between co-roots and weights. 
The Cartan matrix is $(a_{ij})$ for $a_{ij} = \langle \alpha_i^{\vee}, \alpha_j\rangle$.
The fundamental weights are $\varpi_k = \epsilon_1 +\cdots + \epsilon_k$, and $\rho$ denotes the sum of the fundamental weights 
(or, equivalently, the half-sum of the positive roots). 

The quantum group of type $C_n$, denoted $U_q(\spn)$, is defined as the $\mathbb{C}(q)$-algebra given by generators $E_i, F_i$ and $K_i^{\pm 1}$ modulo the following relations:
\[  K_iK_i^{-1} = 1= K_i^{-1}K_i, \quad K_iK_j = K_jK_i, \quad  K_iE_j= q^{(\alpha_i, \alpha_j)}E_jK_i, \quad  K_iF_j = q^{-(\alpha_i, \alpha_j)}F_jK_i, \]
\[ E_iF_j - F_jE_i = \delta_{ij}\dfrac{K_i- K_i^{-1}}{q_i - q_i^{-1}}, \]
\[ \sum_{s= 0}^{1- a_{ij}} (-1)^s{1- a_{ij}\brack s}_{q_i} E_i^{1- a_{ij}- s}E_jE_i^s = 0, \quad \sum_{s= 0}^{1- a_{ij}} (-1)^s{1- a_{ij}\brack s}_{q_i} F_i^{1- a_{ij}- s}F_jF_i^s = 0. \]
Here, we use the conventions
\[ q_i= q^{(\alpha_i, \alpha_i)/2}, \quad \text{so } q_i = q \text{ for } 1 \le i < n \text{ and } q_n = q^2, \]
\[ [n]_{q_i}= \frac{q_i^n- q_i^{-n}}{q_i - q_i^{-1}}, \quad [n]_{q_i}! = [n]_{q_i}[n-1]_{q_i}\cdots[2]_{q_i}[1]_{q_i},  \quad {n \brack k}_{q_i}= \frac{[n]_{q_i}!}{[n-k]_{q_i}![k]_{q_i}!}. \]
Undecorated quantum integers $[n]$ should be taken to be quantum integers in $q$. 

The quantum group $U_q(\spn)$ is a Hopf algebra with coproduct, counit, and antipode defined on generators by
\[ \Delta(F_i) = F_i\ot K_i^{-1} + 1+ F_i,\quad \Delta(E_i) = E_i \ot 1+ K_i\ot E_i,\quad \Delta(K_i) = K_i\ot K_i ,\]
\[ \epsilon(F_i) = 0,\quad \epsilon(E_i) = 0,\quad \epsilon(K_i) = 1, \]
\[ S(F_i) = - F_iK_i,\quad S(E_i) = - K_i^{-1}E_i,\quad S(K_i) = K_i^{-1}.\]

Let $\Rep(U_q(\spn))$ denote the category of finite dimensional, type I representations of $U_q(\spn)$ over $\C(q)$, and let $\Rep(\spn)$ denote the category of finite-dimensional
representations of the lie algebra $\spn$ over $\C$. All representations in either category are completely reducible, and the irreducible objects (up to isomorphism) in either
category are classified by the set $X_+$ of dominant integral weights. For $\lambda \in X_+$ we write $L_{\lambda}$ for the corresponding irreducible representation of $\spn$, and
$V_{\lambda}$ for the irreducible representation of $U_q(\spn)$. The representations $V_{\lambda}$ and $L_{\lambda}$ have the same formal characters; 
i.e. if $V[\mu]$ denotes the $\mu$ weight space of $V$, then $\dim(V_{\lambda}[\mu]) = \dim(L_{\lambda}[\mu])$.  
The categories $\Rep(U_q(\spn))$ and $\Rep(\spn)$ also share the same rules for decomposing tensor products.
For the remainder of this paper we write $L_k$ for $L_{\varpi_k}$ and $V_k$ for $V_{\varpi_k}$. 
We also write $L_0$ (resp. $V_0$) for the trivial representation.
  
\begin{rem} 
For the Lie algebra $\gln$, the $k$-th fundamental representation is isomorphic to the exterior product $\Lambda^k \C^n$. 
For $\spn$, the exterior products $\Lambda^k L_1$ are not irreducible, 
e.g. the symplectic form embeds the trivial representation $L_0$ as a summand of $\Lambda^2 L_1$. 
Instead, $L_k$ is the unique direct summand of $\Lambda^k L_1$ that does not appear in a smaller exterior product. 
This construction of fundamental representations is less explicit for the quantum group $U_q(\spn)$; 
its category of representations is braided but not symmetric, so more care must be taken in defining the exterior 
product\footnote{In \cite{BerZwick} a symmetric product and exterior product is defined for quantum group representations. 
It appears that the exterior product $\Lambda^k_q V_1$ for $U_q(\spn)$ does have the same character as $\Lambda^k L_1$ for $\spn$, 
so that the same description of fundamental representations should work in the quantum case.}.
\end{rem}

\subsection{On quantum dimensions and pivotal structures} \label{subsec:qdim}

The following discussion is based on \cite{ST}.
It is common in the study of quantum groups to define braidings and related structures using an $R$-matrix. To discuss the $R$-matrix one must first take some sort of completion of the quantum group; 
let $\tilde{U_q(\spn)}$ (resp. $\tilde{U_q(\spn)\ot U_q(\spn)}$) denote the completions from \cite[Definition 3.1]{ST}.
It is shown in \cite[Theorem 5.2]{ST} that there is an element $X\in \tilde{U_q(\spn)}$ so that $R= (X^{-1} \ot X^{-1})\Delta(X)$ is the universal $R$-matrix, and $\nu = X^{-2}$ is a (non-standard) ribbon
element. This means that $\Rep(U_q(\spn)$ is a ribbon category, so in particular is a braided pivotal category. There is also an element $u$, satisfying $\nu^2= uS(u)$, so that if $g=
\nu^{-1}u$ then the corresponding pivotal structure $V\xrightarrow{\cong}V^{**}$ is given by $v\mapsto (f\mapsto (f(gv)))$. A consequence is that the \emph{quantum dimension} of $V$ is 
\[ \qdim(V) = \text{tr}(g|_{V}). \] The element $g$ acts on $V_{\lambda}$ by $(-1)^{\langle \rho^{\vee}, 2\lambda \rangle} K^{2\rho}$, and $K^{2\rho}$ acts on the $\mu$ weight space of a
representation as $q^{(2\rho, \mu)}$. From this we deduce the \emph{quantum Weyl dimension formula}
\begin{equation}
\qdim(V_{\lambda}) = (-1)^{\langle \rho^{\vee}, 2\lambda \rangle} \sum_{\mu} \dim (V_{\lambda})[\mu] q^{(2\rho, \mu)} 
= (-1)^{\langle \rho^{\vee}, 2\lambda \rangle} \prod_{\alpha\in \Phi_+}\dfrac{[\langle \alpha^{\vee}, \lambda + \rho\rangle]_{q_{\alpha}}}{[\langle \alpha^{\vee}, \rho\rangle]_{q_{\alpha}}}
\end{equation}
where we write $q_{\alpha}= q^{(\alpha, \alpha)/2}$.

There is a simpler formula for the quantum dimension of the fundamental representations in type $C$, which appears in the PhD thesis of the fourth-named author.
\begin{prop}
We have
\begin{equation}\label{dimVk}
\qdim (V_k) = (-1)^k \dfrac{[n+ 1- k]}{[n+1]}{2n+2 \brack k}
\end{equation}
\end{prop}

\begin{proof}
Since $2\rho^{\vee} = \sum_{\alpha\in \Phi_+}\alpha^{\vee} = (2n-1)\epsilon_1 + (2n-3)\epsilon_2+ \cdots+ 3 \epsilon_{n-1} + \epsilon_n$, 
we find that $\langle \rho^{\vee}, 2\varpi_{k} \rangle = 2nk- k^2$, so $(-1)^{\langle \rho^{\vee}, 2\varpi_{k} \rangle} = (-1)^{k^2} = (-1)^k$.

Expanding the quantum Weyl dimension formula results in
\begin{equation}\label{dimV1a}
\qdim(V_{1}) = -\prod_{2\epsilon_i}\dfrac{[(\epsilon_i, \epsilon_1 + \rho)]_{q^2}}{[( \epsilon_i, \rho)]_{q^2}}
	\prod_{\substack{ \epsilon_i- \epsilon_j \\ i< j}}\dfrac{[( \epsilon_i - \epsilon_j, \epsilon_1 + \rho)]_{q}}{[(\epsilon_i- \epsilon_j,  \rho)]_{q}}
	\prod_{\substack{ \epsilon_i+\epsilon_j \\ i< j}}\dfrac{[(\epsilon_i + \epsilon_j, \epsilon_1 + \rho)]_{q}}{[(\epsilon_i + \epsilon_j, \rho)]_{q}}.
\end{equation}
In each term for which $i \ne 1$, $\epsilon_i$ and $\epsilon_1$ are orthogonal (as are $\epsilon_j$ and $\epsilon_1$ since $i < j$), and the numerator and denominator cancel. Thus \eqref{dimV1a} becomes
\begin{equation}\label{dimV1b}
\qdim(V_{1}) = -\dfrac{[(\epsilon_1, \epsilon_1 + \rho)]_{q^2}}{[( \epsilon_1, \rho)]_{q^2}}
	\prod_{ \epsilon_1- \epsilon_j}\dfrac{[( \epsilon_1 - \epsilon_j, \epsilon_1 + \rho)]_{q}}{[(\epsilon_1- \epsilon_j,  \rho)]_{q}}
	\prod_{\epsilon_1 + \epsilon_j}\dfrac{[(\epsilon_1 + \epsilon_j, \epsilon_1 + \rho)]_{q}}{[(\epsilon_1+ \epsilon_j, \rho)]_{q}}.
\end{equation}
To continue we note that
\[ \rho = n\epsilon_1 + (n-1) \epsilon_2 + \cdots + \epsilon_n \]
and
\[ [n]_{q^{2}}= [2n]_q/[2]_q, \quad \text{ so } \quad \dfrac{[a]_{q^2}}{[b]_{q^2}} = \dfrac{[2a]}{[2b]}. \] Then \eqref{dimV1b} evaluates to be
\begin{equation}\label{dimV1c}
\qdim(V_{1}) = - \dfrac{[2n+2]}{[2n]} \cdot \dfrac{[2][3]\cdots[n]}{[1][2]\cdots[n-1]} \cdot \dfrac{[2n][2n-1]\cdots [n+2]}{[2n-1][2n-2]\cdots [n+1]} = - \dfrac{[2n+2][n]}{[n+1]},
\end{equation}
proving the $k=1$ case of \eqref{dimVk}.

We now consider the general case. Instead of directly computing $\qdim V_k$, we compute the ratio of two successive dimensions. Using the quantum Weyl dimension formula we
see that \begin{equation} \qdim(V_{k})/\qdim(V_{{k-1}}) = (-1)^{\langle \rho^{\vee}, 2(\varpi_k - \varpi_{k-1}) \rangle} \prod_{\alpha\in \Phi_+}\dfrac{[\langle
\alpha^{\vee}, \varpi_k + \rho\rangle]_{q_{\alpha}}}{[\langle \alpha^{\vee}, \varpi_{k-1} + \rho\rangle]_{q_{\alpha}}}. \end{equation} Since $\varpi_k - \varpi_{k-1} = \epsilon_k$,
the only differences between the numerator and denominator occur when pairing against positive roots involving $\epsilon_k$, just as before we only considered positive roots
involving $\epsilon_1$. Computing, one obtains \begin{equation} \label{dimratio} \qdim(V_{k})/\qdim(V_{{k-1}}) = - \dfrac{[2(n+2-k)]}{[2(n+1-k)]} \cdot
\dfrac{[n-k+1]}{[k]} \cdot \dfrac{[2n+3-k][2n+2-2k]}{[2n+4-2k][n+2-k]}. \end{equation} Meanwhile, the ratio of two successive terms from \eqref{dimVk} is
\begin{equation} - \dfrac{[n-k+1][2n+3-k]}{[n-k+2][k]} \end{equation} which matches \eqref{dimratio}. Thus the result follows by induction.
\end{proof}

The main result of this paper is a diagrammatic, generators and relations, presentation of the category $\Rep(U_q(\spn))$, as a pivotal category. 
The choice of pivotal structure for our presentation is the one coming from Snyder and Tingley's ribbon element $\nu = X^{-2}$. Let us motivate this choice.

Recall that a pivotal structure is a coherent identification of $V$ and $(V^{*})^{*}$ for every object $V$.
Since every pivotal category is equivalent (as a pivotal category) to a strict pivotal category wherein $V = (V^{*})^{*}$,
in general, a monoidal presentation of a pivotal category with generating object $V$ will not require a generating object $V^{**}$.
However, even when $V \cong V^*$, it is not always possible to make this latter identification sufficiently canonical that one can remove the objects $V^*$ from the presentation 
(using unoriented strands rather than oriented strands in the diagrammatics). 
A paper of Selinger \cite{Sel2} (in Sections 1.2, 1.3, and 5) gives a condition where duals can be removed from the graphical description: when all objects have Frobenius-Schur indicator equal to one. 
We will refer to such objects as \emph{naturally self-dual}.
In \cite[Lemma 5.7]{ST}, Snyder and Tingley prove that this condition is satisfied for the pivotal structure associated with the ribbon element $\nu = X^{-2}$.
(See also the earlier discussion on \cite[p. 123]{Kup} where a related $\Z/2$-grading is introduced, to the same end.)

\subsection{Brauer-Schur-Weyl duality} \label{subsec:previousC}

The representation theory of $\spn$ has been studied somewhat extensively in the context of (quantum) Brauer-Schur-Weyl duality. The latter is the type $C$ analogue of quantum
Schur-Weyl duality, and can be interpreted as the study of a full subcategory of $\FRep(\spn)$. 
These results will play an important supporting role in our proof of \ref{thm:main}. 
We now recall the pertinent facts, beginning with a recollection about quantum Schur-Weyl duality in type $A$.

\begin{rem}
In this section, we will use the notation $V_k$ to denote the $k$th fundamental representation of either $U_q(\gln)$ or $U_q(\spn)$.
The quantum group under discussion will be clear from the context.
For $\mathfrak{g}=\spn$ or $\gln$, 
we will use the notation $\SRep(U_q(\mathfrak{g}))$ to denote the full monoidal subcategory of $\Rep(U_q(\mathfrak{g}))$
whose objects are iterated tensor products of the standard representation $V_1$. 
\end{rem}

In type $A$, Schur-Weyl duality states that $\End_{U_q(\gln)}(V_1^{\otimes k})$ is a quotient of the Hecke algebra $\HB(\mathfrak{S}_k)$ of the symmetric group. 
This description is compatible with the monoidal structure for the usual inclusion of $\HB(\mathfrak{S}_k) \ot \HB(\mathfrak{S}_l) \hookrightarrow \HB(\mathfrak{S}_{k+l})$. 
The monoidal Hecke algebra
\[
\HB := \bigoplus_{k \ge 0} \HB(\mathfrak{S}_k)
\]
is a deformation of the corresponding direct sum of group algebras $\C[\mathfrak{S}_k]$ and has a familiar diagrammatic interpretation. 
Namely, $\HB$ can be presented as a monoidal category diagrammatically, 
with one generating object $1$ and one generating morphism, the braiding map 
$
\begin{tikzpicture} [scale=.35,anchorbase]
	\draw[very thick] (1,0) to [out=90,in=270] (0,1.5);
	\draw[overcross] (0,0) to [out=90,in=270] (1,1.5);
	\draw[very thick] (0,0) to [out=90,in=270] (1,1.5);
\end{tikzpicture}
\colon 1 \ot 1 \to 1 \ot 1.
$
The kernel of the map $\HB(\mathfrak{S}_k) \to \End(V_1^{\otimes k})$ depends on $n$, but it has a reasonably explicit description: 
it is generated monoidally by a single morphism, the (smooth) Kazhdan-Lusztig basis element corresponding to the longest element of $\mathfrak{S}_{n+1}$ 
(in particular, it is trivial when $k \leq n$).
It is easy to write this element as a linear combination of crossing diagrams\footnote{Explicitly, one takes the sum over all $w \in \mathfrak{S}_{n+1}$ of a positive braid lift of $w$, times $(-q)^{\binom{n+1}{2} - \ell(w)}$.},
thus one obtains an explicit presentation of a diagrammatic category, which is generated by the braiding, and is equivalent to $\SRep(U_q(\gln))$.

In type $C$, $V_1$ has a canonical isomorphism to its dual (classically, this isomorphism comes from the symplectic structure), 
and most of the extra features of type $C$ representation theory come from this self-duality of $V_1$. 
In this case, the algebra $\End_{U_q(\spn)}(V_1^{\otimes k})$ is instead a quotient of the $k$-strand BMW algebra $\BMW_k$, named after Birman-Murakami-Wenzl \cite{BW,Murakami}.
The latter admits a diagrammatic description akin to $\HB(\mathfrak{S}_k)$, but with additional (non-invertible) generators
that are typically depicted as cup-caps:
$
\begin{tikzpicture}[scale =.3, smallnodes,anchorbase]
	\draw[very thick] (0,0) to [out=90,in=180] (.5,.75) to [out=0,in=90] (1,0);
	\draw[very thick] (0,2.25) to [out=270,in=180] (.5,1.5) to [out=0,in=270] (1,2.25);
\end{tikzpicture}.
$

We now give the precise definition. Below, $g_i$ represents the $i$-th braiding map (overcrossing), and $e_i$ represents the $i$-th cup-cap.

\begin{defn}\label{def:BMW}
The \emph{$k$-strand BMW algebra} $\BMW_k(r,z)$ is the unital associative $\Z[r^{\pm},z^{\pm}]$-algebra 
generated by $e_i,g_i,g_i^{-1}$ for $1\leq i \leq k-1$, with relations:
\begin{enumerate}
\item $g_i-g_i^{-1}=z(1-e_i), \quad g_i g_i^{-1} = 1 = g_i^{-1} g_i,$
\item $e_i^2=\left(1+\frac{r-r^{-1}}{z}\right)e_i,$
\item $g_ig_{i+1}g_i=g_{i+1}g_ig_{i+1} \text{ for } 1 \leq i \leq k-2,$
\item $g_ig_j=g_jg_i \text{ for } |i-j|>1,$
\item $e_ie_{i+1}e_i=e_i, \quad e_{i+1}e_ie_{i+1}=e_{i+1} \quad \text{ for } 1 \leq i \leq k-2,$
\item $g_ig_{i+1}e_i=e_{i+1}e_i, \quad  g_{i+1}g_ie_{i+1}=e_ie_{i+1} \quad \text{ for }1 \leq i \leq k-2,$
\item $e_ig_i=g_ie_i=r^{-1} e_i,$
\item $e_ig_{i+1}e_i=re_i, \quad e_{i+1}g_ie_{i+1}=re_{i+1} \quad \text{ for } 1 \leq i \leq k-2.$
\end{enumerate}
\end{defn}

Taking the sum over all $k$, one again obtains a monoidal category (the BMW category) with one generating object $1$, 
and generating morphisms $g_1,e_1 \in \BMW_2$. 
The BMW category admits a full functor to $\SRep(U_q(\spn))$, as we recall below in \S \ref{sec:full}.

The kernel of the map
\begin{equation}\label{eq:BMW}
\BMW_k \to \End_{U_q(\spn)}(V_{1}^{\otimes k}) 
\end{equation}
is generated monoidally by a single morphism in $\BMW(n+1)$, 
a fact proved in \cite{HuXiaoTensorBMW} and further explored in \cite{RubWes,LZbrauer} 
(see also \cite{LZorthogonal} for the analogous result for the orthogonal group).
This morphism has a less explicit description than in type $A$; 
classically (i.e. at $q=1$) it is a \emph{Pfaffian} \cite{RubWes}, 
but, to our knowledge, an explicit formula for the analogous \emph{quantum Pfaffian} in terms of BMW generators
does not appear in the literature (however, see \cite[Theorem 8.2]{LZbrauer} for a characterization).
Given this, an approach to Theorem \ref{thm:main} using these results seems more laborious than the one taken in \S \ref{sec:proof}.
Indeed, beyond explicitly identifying the \emph{quantum Pfaffian}, it would require analogues of many of our spanning results obtained in \S \ref{sec:RGS}--\ref{sec:PA}
in the setting of the \emph{quantized Brauer category}.
The latter is the category obtained from the BMW category by adjoining cap and cup morphisms 
(as defined above, the BMW category only contains the composition of cup with cap).
We instead skirt these difficulties via a change of generators; 
see Remarks \ref{rem:LZ1} and \ref{rem:LZ2} for further discussion.

One useful takeaway from the present work is that the kernel of \eqref{eq:BMW} is easier to describe in the language of webs 
than in the language of tangled planar matchings. 
The generator of the kernel is the composition of natural webs $1^{\ot n+1} \to (n+1) \to 1^{\ot n+1}$.
In fact, this result is paralleled in type $A$, where the skew Howe duality approach of Cautis-Kamnitzer-Morrison 
shows that the kernel of the map
$\HB(\mathfrak{S}_k) \to \End(V_1^{\otimes k})$ is generated by a single web of the same form.

\subsection{Other diagrammatic categories}\label{subsec:connections}

We now review a number of existing diagrammatic descriptions of categories of quantum group representations appearing in the literature, 
with an eye towards comparing to (and contrasting with) our category $\Web(\spn)$.

The primordial example of a diagrammatic description of a representation category is the equivalence 
between the Temperley-Lieb category $\TL_q$ and the category $\FRep(U_q(\sln[2]))$ \cite{RTW,TL,Kup}.
Recall that $\TL_q$ is the $\C(q)$-linear pivotal category freely generated by a single self-dual object, 
modulo the ideal generated by the relation
\begin{equation}\label{eq:TLcircle}
\begin{tikzpicture}[scale =.75, tinynodes,anchorbase]
	\draw[very thick] (0,0) circle (.5);
\end{tikzpicture}
= -[2] = -(q^{-1} + q).
\end{equation}
In other words, morphisms in $\TL_q$ are $\C(q)$-linear combinations of planar tangles, 
modulo isotopy (rel boundary) and the relation \eqref{eq:TLcircle}.
As above, $\FRep(U_q(\sln[2]))$ is the full subcategory of $\Rep(U_q(\sln[2]))$ monoidally generated 
by the fundamental representation, which in this case is simply the (quantum analogue of the) vector representation $V$. 
The equivalence is given by sending the generating object of $\TL_q$ to $V$, 
and uses the non-standard pivotal structure on $\FRep(U_q(\sln[2]))$ from \cite{ST} 
(the $\sln[2]$ case of the pivotal structure discussed in \S\ref{subsec:qdim}).
See also \cite{Tingley}.

Since there is an isomorphism of algebras $U_q(\mathfrak{sp}_2) \cong U_{q^2}(\mathfrak{sl}_2)$, 
the following is not surprising.

\begin{prop} When $n=1$, the category $\Web(\spn)$ is equivalent to $\TL_{q^2}$. \end{prop}
	
\begin{proof}
The relations (\text{\ref{eq:spn}b}), (\text{\ref{eq:spn}c}), (\text{\ref{eq:spn}d}), and (\text{\ref{eq:spn}e}) all become trivial, since the object $k$ is zero whenever $k> 1$. The relation (\text{\ref{eq:spn}a}) says that 
\[  
\begin{tikzpicture}[scale =.75, tinynodes,anchorbase]
	\draw[very thick] (0,0) circle (.5);
\end{tikzpicture}
= -\frac{[4]}{[2]} = -(q^{-2}+q^2). \qedhere
\] 
\end{proof}

Note that $q^2$ appears both in the isomorphism of quantum groups, and in the corresponding equivalence of categories, 
since the unique positive root in type $C_1$ is a long root, i.e. $\alpha_1 = 2\epsilon_1$ so $(\alpha_1, \alpha_1)= 4$. 

Given the diagrammatic description of $\FRep(U_q(\sln[2]))$, 
the next natural question is to find a presentation via generators and relations for $\FRep(\mathfrak{g})$ for other simple Lie algebras.
This problem was solved in rank $2$ by Kuperberg in \cite{Kup}, where he defined diagrammatic categories that we denote by $\Web_q(A_2)$, $\Web_q(B_2)$, and $\Web_q(G_2)$,
and proved that they are equivalent to $\FRep(U_{\sqrt{q}}(\mathfrak{sl}_3))$, $\FRep(U_{\sqrt{q}}(\mathfrak{so}_5))$, and $\FRep(U_{\sqrt{q}}(\mathfrak{g}_2))$, respectively.
The low rank coincidence $\spn[4] \cong \mathfrak{so}_5$ leads one to ask how Kuperberg's category $\Web_q(B_2)$ 
(as defined on \cite[p. 126]{Kup}) relates to the $n=2$ version of $\Web(\spn)$.

\begin{prop}\label{prop:n=2} 
When $n=2$, the category $\Web(\spn)$ is equivalent (as a $\C(q)$-linear ribbon category) to $\Web_{q^2}(B_2)$.
\end{prop}

\begin{proof}[Proof (sketch).] 
The equivalence identifies Kuperberg's single and doubled edges with (the identity morphisms of) the objects $1$ and $2$ from $\Web(\spn)$, respectively.
The simplest way to define the functor on morphisms is to extend scalars to $\C(q,\sqrt{-1/[2]})$,
and then send Kuperberg's generating morphism $1 \ot 1 \to 2$ to $\sqrt{-1/[2]}$ times our generating trivalent vertex. 
It is a pleasant exercise\footnote{For assistance in proving Kuperberg's penultimate relation, the reader should consult \eqref{eq:triangle} and its proof.} 
to show that the relations in \cite[Equation (3)]{Kup} imply our relations in \eqref{eq:spn}, and vice versa. 

Alternatively, this equivalence can be defined over $\C(q)$. 
In this case, the functor rescales both the generating trivalent vertex and the $2$-labeled cap/cup morphisms.
(Compare to the automorphisms of $\Web(\spn)$ as described in Porism \ref{por:auto}.)
\end{proof}

The higher-rank analogue of $\Web_q(A_2)$ was studied by several authors \cite{Kim,Sikora,Mor} 
and culminated in landmark work of Cautis-Kamnitzer-Morrison \cite{CKM}.
They use a quantized version of skew Howe duality, i.e. the duality arising from the 
commuting actions of $\gln[m]$ and $\gln$ on $\Lambda^\bullet(\C^m \otimes \C^n)$, 
to construct a category of $\sln$ webs and show that it is equivalent to $\FRep(U_q(\sln))$.
Unfortunately, work of Sartori-Tubbenhauer \cite{SarTub} suggests that this elegant approach is limited to type $A$. 
Although Howe dualities exist in other classical types, \cite{SarTub} shows that
they do not quantize to give dualities between pairs of quantum groups, 
but rather between a quantum group and an associated coideal subalgebra of $U_q(\gln)$. 
To obtain web categories for quantum groups using this approach would require new ideas.

Thus, far less is known outside of type $A$ . 
In \cite{WesSpin}, Westbury gives a diagrammatic presentation for the subcategory of $\Rep(\mathfrak{so}_7)$ 
generated by the vector and the spin representation (note that these are not all the fundamental representations, so there is still work to do here). 
More recently, the third- and fourth-named authors defined a diagrammatic category that surjects onto $\FRep(U_q(\mathfrak{sp}_6))$ \cite{RoseTatham}. 
In fact, the $\spn[6]$ web category in \cite{RoseTatham} was the starting point for the definition of $\Web(\spn)$, and some of the ideas in this article have antecedents in that work. 
The following result is therefore not surprising.

\begin{prop}\label{prop:n=3} 
When $n=3$, the category $\Web(\spn)$ is equivalent (as a $\C(q)$-linear ribbon category) to the $\spn[6]$ web category of \cite{RoseTatham}.
\end{prop}

\begin{proof}[Proof (sketch).]
As in the proof of Proposition \ref{prop:n=2}, the equivalence is given by a rescaling.
Again, this is most-easily accomplished by extending scalars to $\C(q,\sqrt{1/[2]},\sqrt{1/[3]})$ and rescaling the  $1\otimes 1 \to 2$ 
trivalent vertex in \cite{RoseTatham} by $\sqrt{1/[3]}$ and the $1\otimes 2 \to 3$ trivalent vertex by $\sqrt{1/[2]}$.
However, as above, it is also possible to define the rescaling over $\C(q)$ by rescaling both trivalent vertices and cap/cup morphisms
(see the related discussion on \cite[pp. 11-12]{RoseTatham} concerning the parameters $\delta$ and $\delta'$).

Note that the second line of the relations in \cite[Definition 2.1]{RoseTatham} contains two relations that do not have immediate 
analogues in $\Web(\spn)$ when $n=2$.
However, it is straightforward to see that both follow from (\text{\ref{eq:spn}e}) and further, 
that both have natural generalizations for general $n$.
The first of these relations asserts that a ``triangle web'' in $\Hom(3 \otimes 3,2)$ is zero; this follows from the $k=n$, $n=3$ case of (\text{\ref{eq:spn}e}).
The same argument shows that the corresponding web in $\Hom(n \otimes n,2)$ in $\Web(\spn)$ is always zero.
The second of these relations is is a consequence of the $k=2$ case of (\text{\ref{eq:spn}e}), and is philosophically an alternative form of this relation.
One can check that after the appropriate rescaling, the coefficient of the identity on the right hand side is now $1$.
This should be compared with the ```elementary neutral ladders'' from \cite{EliasLL, Bodish}. 
Finally, the second equation on the third line of the relations in \cite[Definition 2.1]{RoseTatham} is the analogue of our (\text{\ref{eq:spnOther}e})
(and is thus redundant).
\end{proof}

Lastly, we comment on some diagrammatic categories in the literature that describe (classical) type $C$ representation theory in higher rank.

\begin{rem} The previously-mentioned work of Sartori-Tubbenhauer \cite{SarTub} defines web categories that describe maps between Lie algebra representations in types $B$, $C$, and $D$, 
and between representations of associated coideal subalgebras. 
Since the latter are alternative quantizations that \textbf{are not} equal to the corresponding quantum groups, 
we do not expect a relation to $\Web(\spn)$, except when $q=1$. 
Under this specialization, it is an interesting problem to relate their work to ours.
One should note that the representations considered in \cite{SarTub} are tensor products of exterior and symmetric powers of the standard representation $L_1$. 
Hence, the comparison would involve interpreting exterior and symmetric powers as objects in the Karoubi envelope $\Kar(\Web(\spn)|_{q=1})$
and explicitly computing the corresponding idempotents.
\end{rem}

\begin{rem} \label{rem:LZ1} 
The aforementioned work of Lehrer-Zhang \cite{LZbrauer} proves that a quotient of the Brauer category is equivalent to $\SRep(\spn)$, 
the full subcategory of the category of finite-dimensional representations of the Lie algebra $\spn$ monoidally generated by the standard representation $L_1=\C^{2n}$.
By contrast, we show that $\Web(\spn)$ is equivalent to the entire fundamental subcategory $\FRep(U_q(\spn))$ of the category of finite-dimensional representations of the quantum group $U_q(\spn)$.
Nevertheless, it is natural to ask whether the results in \cite{LZbrauer} could be adapted to the quantum setting to give an alternative approach to our main result.
In fact, such an approach would require analogues of many of the combinatorial steps in our proof, and would further require an explicit description of the quantum Pfaffian.
We will comment further on these issues in Remark \ref{rem:LZ2}, after outlining the main steps in our proof.
In the meantime, we warn the reader of a potential source of confusion. 
Recall that morphisms in the Brauer category are described by generically immersed planar curves, i.e. quadrivalent graphs.
Our proof similarly makes use of quadrivalent graphs (see Definition \ref{def:quad}), 
but our quadrivalent vertex \textbf{does not} match the corresponding vertices in the Brauer category, even at $q=1$.
Rather, the vertices in the Brauer category are equal to the $q=1$ specialization of the braiding \eqref{eq:braiding},
which the reader should observe is distinct from the $q=1$ specialization of our quadrivalent vertices.
\end{rem}

%
\section{Expanded diagrammatics}\label{sec:more}
%

In this section, we define several other useful morphisms in $\Web(\spn)$ and present some relations that they satisfy.
Along the way we establish some related results.
We hope that the additional relations increase the usability of our diagrammatic calculus, 
and serve as a preview for further relations that will be found in the sequel to this paper \cite{BERT2}.

\subsection{Flow vertices}\label{subsec:flow}

For $\gln$, the direct sum of the fundamental representations (together with the trivial representation, denoted $L_{\varpi_0}$)
\[ \Lambda := \bigoplus_{k=0}^n L_{\varpi_k} \]
can be equipped with the structure of a (graded) algebra object in the category of $\gln$ representations. 
Since $L_{\varpi_k} \cong \Lambda^k L_{\varpi_1}$, this is just the exterior algebra of $L_{\varpi_1}$. 
In particular, there is a wedge-product multiplication map $L_{\varpi_k} \ot L_{\varpi_\ell} \to L_{\varpi_{k+\ell}}$ which is associative.
Moreover, the results in \cite{CKM} imply that the analogous result holds in the quantum setting.

Meanwhile, for $\spn$ and $U_q(\spn)$, the fundamental representations are not exterior products. 
Nonetheless, $V_{k+\ell}$ appears with multiplicity one as a direct summand of $V_k \ot V_\ell$, so there is a nonzero morphism (unique up to rescaling) which could serve as a multiplication map. 
In fact, in Lemma \ref{lem:extrarels} below we show that it is possible to choose these maps so that they satisfy the necessary associativity relation. 
This immediately implies the following.

\begin{thm} \label{thm:algebra} 
The direct sum of the fundamental $U_q(\spn)$ representations $\bigoplus_{k=0}^n V_k$ is a (graded) algebra object in $\Rep(U_q(\spn))$.
\end{thm}

More generally, suppose $\CS$ is a $\C(q)$-linear pivotal category containing naturally self-dual objects $\{1,\ldots,n\}$ 
and a collection of distinguished morphisms in $\Hom_{\CS}(1\ot k, k+1)$ and $\Hom_{\CS}(k\ot 1, k+1)$, 
that we depict as trivalent vertex morphisms as in \eqref{eq:WebGen}.
Further, suppose that they satisfy (\text{\ref{eq:spn}c}) and (\text{\ref{eq:spn}d}).
Theorem \ref{thm:algebra} will follow by using these morphisms to build more-general multiplication morphisms:
\begin{equation} \label{eq:flowvertfirst}
\begin{tikzpicture}[scale =.5, smallnodes,anchorbase]
	\draw[very thick] (0,0) node[below]{$k$} to [out=90,in=210] (.5,.75);
	\draw[very thick] (1,0) node[below]{$\ell$} to [out=90,in=330] (.5,.75);
	\draw[very thick] (.5,.75) to (.5,1.5) node[above=-2pt]{$k{+}\ell$};
\end{tikzpicture}
\end{equation}
for all $0 \le k, \ell \le n$ that satisfy the general associativity relation:
\begin{equation} \label{eq:generalassoc} 
\begin{tikzpicture}[scale=.2, xscale=-1,tinynodes, anchorbase]
	\draw [very thick] (-1,-1) node[below,yshift=2pt]{$m$} to [out=90,in=210] (0,.75);
	\draw [very thick] (1,-1) node[below,yshift=2pt]{$\ell$} to [out=90,in=330] (0,.75);
	\draw [very thick] (3,-1) node[below,yshift=2pt]{$k$} to [out=90,in=330] (1,2.5);
	\draw [very thick] (0,.75) to [out=90,in=210] node[right=-2pt]{$\ell{+}m$} (1,2.5);
	\draw [very thick] (1,2.5) to (1,4.25) node[above,yshift=-3pt]{$k{+}\ell{+}m$};
\end{tikzpicture}
=
\begin{tikzpicture}[scale=.2, tinynodes, anchorbase]
	\draw [very thick] (-1,-1) node[below,yshift=2pt]{$k$} to [out=90,in=210] (0,.75);
	\draw [very thick] (1,-1) node[below,yshift=2pt]{$\ell$} to [out=90,in=330] (0,.75);
	\draw [very thick] (3,-1) node[below,yshift=2pt]{$m$} to [out=90,in=330] (1,2.5);
	\draw [very thick] (0,.75) to [out=90,in=210] node[left=-2pt]{$k{+}\ell$} (1,2.5);
	\draw [very thick] (1,2.5) to (1,4.25) node[above,yshift=-3pt]{$k{+}\ell{+}m$};
\end{tikzpicture}.
\end{equation}
As usual, any diagram with an index $> n$ is considered to be the zero morphism.
Further, when $k=0$ or $\ell=0$, we let \eqref{eq:flowvertfirst} be the relevant identity morphism.

We now motivate the definition of the morphisms in \eqref{eq:flowvertfirst}.
Suppose we wish to construct the multiplication map $k \ot 2 \to (k+2)$ using the generating vertices 
(and the (co)unit morphisms in $\CS$). 
A naive attempt is as follows:
\[
\begin{tikzpicture}[scale=.3,smallnodes,anchorbase]
	\draw[very thick] (-1,0) to node[below=-1pt]{$1$} (1,0);
	\draw[very thick] (-1,0) to node[left,xshift=2pt]{$k{+}1$} (0,1.732);
	\draw[very thick] (1,0) to node[right,xshift=-2pt]{$1$} (0,1.732);
	\draw[very thick] (0,1.732) to (0,3.232) node[above=-2pt]{$k{+}2$};
	\draw[very thick] (-2.3,-.75) node[below=-2pt,xshift=-2pt]{$k$} to (-1,0);
	\draw[very thick] (2.3,-.75) node[below=-2pt,xshift=2pt]{$2$} to (1,0);
\end{tikzpicture}.
\]
This indeed gives a morphism in $\Hom_{\CS}(k \ot 2, k+2)$;
however, if \eqref{eq:generalassoc} is to hold, then we would compute:
\begin{equation} \label{eq:naiveattemptresolves}
\begin{tikzpicture}[scale=.3,smallnodes,anchorbase]
	\draw[very thick] (-1,0) to node[below=-1pt]{$1$} (1,0);
	\draw[very thick] (-1,0) to node[left,xshift=2pt]{$k{+}1$} (0,1.732);
	\draw[very thick] (1,0) to node[right,xshift=-2pt]{$1$} (0,1.732);
	\draw[very thick] (0,1.732) to (0,3.232) node[above=-2pt]{$k{+}2$};
	\draw[very thick] (-2.3,-.75) node[below=-2pt,xshift=-2pt]{$k$} to (-1,0);
	\draw[very thick] (2.3,-.75) node[below=-2pt,xshift=2pt]{$2$} to (1,0);
\end{tikzpicture}
= 
\begin{tikzpicture}[scale=.2,anchorbase,smallnodes,xscale=-1]
	\draw [very thick] (0,-2) node[below=-1pt]{$2$} to (0,-1);
	\draw [very thick] (0,-1) to [out=150,in=210] node[right=-2pt]{$1$} (0,.75);
	\draw [very thick] (0,-1) to [out=30,in=330] node[left=-2pt]{$1$} (0,.75);
	\draw [very thick] (3,-2) node[below=-1pt]{$k$} to [out=90,in=330] (1,2.5);
	\draw [very thick] (0,.75) to [out=90,in=210] (1,2.5);
	\draw [very thick] (1,2.5) to (1,4.25) node[above=-2pt]{$k{+}2$};
\end{tikzpicture}
\stackrel{(\text{\ref{eq:spn}c})}{=}
[2]
\begin{tikzpicture}[scale =.5, smallnodes,anchorbase]
	\draw[very thick] (0,-.25) node[below=-1pt]{$k$} to [out=90,in=210] (.5,.75);
	\draw[very thick] (1,-.25) node[below=-1pt]{$2$} to [out=90,in=330] (.5,.75);
	\draw[very thick] (.5,.75) to (.5,1.5) node[above=-2pt]{$k{+}2$};
\end{tikzpicture}
\end{equation} 
Dividing both sides of the equation by $[2]$ suggests the correct definition of the trivalent vertex $k \ot 2 \to (k+2)$. 
We now make this precise.

\begin{defn} \label{defn:flowvertex}
Given $k,\ell$ satisfying $k+\ell \leq n$, 
let the \emph{flow vertices} be recursively defined by
\begin{equation}\label{eq:flowvertex}
\begin{tikzpicture}[scale =.5, smallnodes,anchorbase]
	\draw[very thick] (0,0) node[below=-1pt]{$k$} to [out=90,in=210] (.5,.75);
	\draw[very thick] (1,0) node[below=-1pt]{$\ell$} to [out=90,in=330] (.5,.75);
	\draw[very thick] (.5,.75) to (.5,1.5) node[above=-2pt]{$k{+}\ell$};
\end{tikzpicture}
:=
\frac{1}{[k]}
\begin{tikzpicture}[scale=.3,smallnodes,anchorbase]
	\draw[very thick] (-1,0) to node[below=-1pt]{$k{-}1$} (1,0);
	\draw[very thick] (-1,0) to node[left,xshift=2pt]{$1$} (0,1.732);
	\draw[very thick] (1,0) to node[right,xshift=-2pt]{$k{+}l{-}1$} (0,1.732);
	\draw[very thick] (0,1.732) to (0,3.232) node[above=-2pt]{$k{+}l$};
	\draw[very thick] (-2.3,-.75) node[below=-2pt,xshift=-2pt]{$k$} to (-1,0);
	\draw[very thick] (2.3,-.75) node[below=-2pt,xshift=2pt]{$l$} to (1,0);
\end{tikzpicture}
=
\frac{1}{[l]}
\begin{tikzpicture}[scale=.3,smallnodes,anchorbase]
	\draw[very thick] (-1,0) to node[below=-1pt]{$l{-}1$} (1,0);
	\draw[very thick] (-1,0) to node[left,xshift=2pt]{$k{+}l{-}1$} (0,1.732);
	\draw[very thick] (1,0) to node[right,xshift=-2pt]{$1$} (0,1.732);
	\draw[very thick] (0,1.732) to (0,3.232) node[above=-2pt]{$k{+}l$};
	\draw[very thick] (-2.3,-.75) node[below=-2pt,xshift=-2pt]{$k$} to (-1,0);
	\draw[very thick] (2.3,-.75) node[below=-2pt,xshift=2pt]{$l$} to (1,0);
\end{tikzpicture}.
\end{equation}
\end{defn}

\begin{rem} Equation \eqref{eq:flowvertex} is really two recursive definitions at once. 
If $k=1$ then the first equality is trivial and if $\ell = 1$ then the second equality is trivial, but at least one recursion can be used to define the flow vertex for any pair $(k,\ell)$ with $\max(k,\ell)>1$. 
We will prove the two recursions agree below. 
\end{rem}

\begin{conv}[Local orientation on flow vertices]\label{conv:flow}
We will occasionally orient the edge carrying the largest label in a flow vertex away from the vertex,
which will allow us to only specify labels on two of the three edges, e.g.
\[
\begin{tikzpicture}[scale =.5, smallnodes,anchorbase]
	\draw[very thick] (0,0) node[below]{$k$} to [out=90,in=210] (.5,.75);
	\draw[very thick] (1,0) node[below]{$l$} to [out=90,in=330] (.5,.75);
	\draw[very thick,->] (.5,.75) to (.5,1.5);
\end{tikzpicture}
=
\begin{tikzpicture}[scale =.5, smallnodes,anchorbase]
	\draw[very thick] (0,0) to [out=90,in=210] (.5,.75);
	\draw[very thick] (1,0) node[below]{$l$} to [out=90,in=330] (.5,.75);
	\draw[very thick,->] (.5,.75) to (.5,1.5) node[above]{$k{+}l$};
\end{tikzpicture}
=
\begin{tikzpicture}[scale =.5, smallnodes,anchorbase]
	\draw[very thick] (0,0) node[below]{$k$} to [out=90,in=210] (.5,.75);
	\draw[very thick] (1,0) to [out=90,in=330] (.5,.75);
	\draw[very thick,->] (.5,.75) to (.5,1.5) node[above]{$k{+}l$};
\end{tikzpicture}
=
\begin{tikzpicture}[scale =.5, smallnodes,anchorbase]
	\draw[very thick] (0,0) node[below]{$k$} to [out=90,in=210] (.5,.75);
	\draw[very thick] (1,0) node[below]{$l$} to [out=90,in=330] (.5,.75);
	\draw[very thick] (.5,.75) to (.5,1.5) node[above]{$k{+}l$};
\end{tikzpicture}
\]
Note that this orientation is only local, 
i.e. it only relates to the vertex away from which it points.
For example, using this convention the last relation in \eqref{eq:spn} takes the form
\begin{equation}\label{eq:flowHI}
\begin{tikzpicture}[scale=.4, rotate=90, tinynodes, anchorbase]
	\draw[very thick] (-1,0) node[below,yshift=2pt,xshift=2pt]{$k$} to (0,1);
	\draw[very thick] (1,0) node[above,yshift=-4pt,xshift=2pt]{$1$} to (0,1);
	\draw[very thick] (0,2.5) to (-1,3.5) node[below,yshift=2pt,xshift=-2pt]{$k$};
	\draw[very thick] (0,2.5) to (1,3.5) node[above,yshift=-4pt,xshift=-2pt]{$1$};
	\draw[very thick,directed=.3,rdirected=.85] (0,1) to (0,2.5);
\end{tikzpicture}
=
\begin{tikzpicture}[scale=.4, tinynodes, anchorbase]
	\draw[very thick,rdirected=.05] (-1,0) node[below,yshift=2pt,xshift=-2pt]{$k$} 
		to node[pos=.7,left,xshift=2pt]{$1$} (0,1.5);
	\draw[very thick,rdirected=.05] (1,0) node[below,yshift=2pt,xshift=2pt]{$k$} 
		to node[pos=.7,right,xshift=-2pt]{$1$} (0,1.5);
	\draw[very thick] (-.7,.5) to (.7,.5);
	\draw[very thick] (0,2.5) to (-1,3.5) node[above,yshift=-4pt,xshift=-2pt]{$1$};
	\draw[very thick] (0,2.5) to (1,3.5) node[above,yshift=-4pt,xshift=2pt]{$1$};
	\draw[very thick, directed=.4,rdirected=.8] (0,1.5) to (0,2.5);
\end{tikzpicture}
-\frac{[n{-}k]}{[n{-}k{+}1]}
\begin{tikzpicture}[scale=.4, rotate=90, tinynodes, anchorbase]
	\draw[very thick,rdirected=.05] (-1,0) node[below,yshift=2pt,xshift=2pt]{$k$} to (0,1);
	\draw[very thick] (1,0) node[above,yshift=-4pt,xshift=2pt]{$1$} to (0,1);
	\draw[very thick,rdirected=.05] (-1,3.5) node[below,yshift=2pt,xshift=-2pt]{$k$} to (0,2.5);
	\draw[very thick] (0,2.5) to (1,3.5) node[above,yshift=-4pt,xshift=-2pt]{$1$};
	\draw[very thick] (0,1) to (0,2.5);
\end{tikzpicture}
+
\frac{[n{-}k]}{[n]}
\begin{tikzpicture}[scale=.4, tinynodes, anchorbase]
	\draw[very thick] (-1,0) node[below,yshift=2pt]{$k$} to [out=90,in=180] (0,1) 
		to [out=0,in=90] (1,0);
	\draw[very thick] (-1,3) node[above,yshift=-3pt]{$1$} to [out=270,in=180] (0,2)
		to [out=0,in=270] (1,3);
\end{tikzpicture}
\end{equation}

This orientation of flow vertices makes it easy to determine whether \eqref{eq:generalassoc} can be applied within a diagram. Regardless of the labels, \eqref{eq:generalassoc} implies that
\[
\begin{tikzpicture}[scale=.2, tinynodes, anchorbase,xscale=-1]
	\draw [very thick] (-1,-1) to [out=90,in=210] (0,.75);
	\draw [very thick] (1,-1) to [out=90,in=330] (0,.75);
	\draw [very thick] (3,-1) to [out=90,in=330] (1,2.5);
	\draw[very thick,rdirected=.7] (1,2.5) to [out=210,in=90] (0,.75);
	\draw [very thick,->] (1,2.5) to (1,4.25);
\end{tikzpicture}
=
\begin{tikzpicture}[scale=.2, tinynodes, anchorbase]
	\draw [very thick] (-1,-1) to [out=90,in=210] (0,.75);
	\draw [very thick] (1,-1) to [out=90,in=330] (0,.75);
	\draw [very thick] (3,-1) to [out=90,in=330] (1,2.5);
	\draw[very thick,rdirected=.7] (1,2.5) to [out=210,in=90] (0,.75);
	\draw [very thick,->] (1,2.5) to (1,4.25);
\end{tikzpicture}
\]
However, one cannot apply \eqref{eq:generalassoc} to a diagram if the orientations are not convergent;
e.g. it cannot be applied to the left-hand side of \eqref{eq:flowHI}.
\end{conv}

We now have the following.

\begin{lem}\label{lem:extrarels}
Let $\CS$ be a pivotal category containing naturally self-dual objects $\{1,\ldots,n\}$ and distinguished morphisms depicted as in \eqref{eq:WebGen} 
that satisfy (\text{\ref{eq:spn}c}) and (\text{\ref{eq:spn}d}).
Then the formulae \eqref{eq:flowvertex} are well-defined, and these morphisms satisfy the general associativity relation in \eqref{eq:generalassoc}.
Further, they also satisfy the general bigon relation given below in (\text{\ref{eq:spnOther}c}).
\end{lem}

\begin{proof}
Let $N > 0$. We begin by showing that the flow vertices are well-defined and satisfy the general associativity relation.
Suppose that we have defined all trivalent vertices \eqref{eq:flowvertfirst} for $k+\ell < N$, and that we have verified \eqref{eq:generalassoc} whenever $k+\ell+m < N$.
For example, when $N = 3$, the only nontrivial case of \eqref{eq:generalassoc} is (\text{\ref{eq:spn}d}) -- this is our base case. 
We now define the trivalent vertices when $k + \ell = N$ using \eqref{eq:flowvertex}, checking that the two definitions agree. 
We then check \eqref{eq:generalassoc} for $k+\ell+m = N$ and (\text{\ref{eq:spnOther}c}), and the result will follow by induction.
	
First we prove the equality of the two diagrams in \eqref{eq:flowvertex}. We compute:
\begin{equation}
\frac{1}{[k]}
\begin{tikzpicture}[scale=.35,smallnodes,anchorbase]
	\draw[very thick] (-1,0) to node[below]{$k{-}1$} (1,0);
	\draw[very thick] (-1,0) to node[left,xshift=2pt]{$1$} (0,1.732);
	\draw[very thick,directed=.4] (1,0) to node[right,xshift=-1pt]{$k{+}\ell{-}1$} (0,1.732);
	\draw[very thick,->] (0,1.732) to (0,3.232) node[above=-2pt]{$k{+}\ell$};
	\draw[very thick,<-] (-2.3,-.75) node[below=-2pt,xshift=-2pt]{$k$} to (-1,0);
	\draw[very thick] (2.3,-.75) node[below=-2pt,xshift=2pt]{$\ell$} to (1,0);
\end{tikzpicture}
\stackrel{\eqref{eq:flowvertex}}{=} 
\frac{1}{[k][\ell]} 
\begin{tikzpicture}[scale=.25,tinynodes,anchorbase,xscale=-1]
	\draw[very thick,<-] (-2,-2) node[below=-2pt]{$\ell$} to [out=90,in=210] (-1,0);
	\draw[very thick,->] (1,0) to (2,-.5) to (2,-2) node[below=-2pt]{$k$};
	\draw[very thick] (-1,0) to node[below=-1pt]{$\ell{-}1$} (1,0);
	\draw[very thick] (-1,0) to node[pos=.7,right=-2pt]{$1$} (0,1.732);
	\draw[very thick,directed=.6] (1,0) to (0,1.732);
	\draw[very thick,directed=.4] (0,1.732) to [out=90,in=210] node[right=-1pt]{$k{+}\ell{-}1$} (1.25,3.5);
	\draw[very thick] (2,-.5) to [out=30,in=330] node[pos=.7,left=-1pt]{$1$} (1.25,3.5);
	\draw[very thick,->] (1.25,3.5) to (1.25,5) node[above=-2pt]{$k{+}\ell$};
\end{tikzpicture}
\stackrel{(\text{\ref{eq:spn}d})}{=} 
\frac{1}{[k][\ell]} 
\begin{tikzpicture}[scale=.25,tinynodes,anchorbase]
	\draw[very thick,<-] (-2,-2) node[below=-2pt]{$k$} to [out=90,in=210] (-1,0);
	\draw[very thick,->] (1,0) to (2,-.5) to (2,-2) node[below=-2pt]{$\ell$};
	\draw[very thick] (-1,0) to node[below=-1pt]{$k{-}1$} (1,0);
	\draw[very thick] (-1,0) to node[pos=.7,left=-2pt]{$1$} (0,1.732);
	\draw[very thick,directed=.6] (1,0) to (0,1.732);
	\draw[very thick,directed=.4] (0,1.732) to [out=90,in=210] node[left=-1pt]{$k{+}\ell{-}1$} (1.25,3.5);
	\draw[very thick] (2,-.5) to [out=30,in=330] node[pos=.7,right=-1pt]{$1$} (1.25,3.5);
	\draw[very thick,->] (1.25,3.5) to (1.25,5) node[above=-2pt]{$k{+}\ell$};
\end{tikzpicture}
\stackrel{\eqref{eq:flowvertex}}{=} 
\frac{1}{[l]}
\begin{tikzpicture}[scale=.35,smallnodes,anchorbase]
	\draw[very thick] (-1,0) to node[below]{$\ell{-}1$} (1,0);
	\draw[very thick,directed=.4] (-1,0) to node[left,xshift=2pt]{$k{+}\ell{-}1$} (0,1.732);
	\draw[very thick] (1,0) to node[right,xshift=-2pt]{$1$} (0,1.732);
	\draw[very thick,->] (0,1.732) to (0,3.232) node[above=-2pt]{$k{+}\ell$};
	\draw[very thick] (-2.3,-.75) node[below=-2pt,xshift=-2pt]{$k$} to (-1,0);
	\draw[very thick,<-] (2.3,-.75) node[below=-2pt,xshift=2pt]{$\ell$} to (1,0);
\end{tikzpicture} .
\end{equation}
(Note that the local orientations determine the labels on all unlabeled edges.)

Next, we prove the $m=1$ case of \eqref{eq:generalassoc}. We compute:
\begin{equation}\label{eq:genassocm=1}
\begin{tikzpicture}[scale=.2, tinynodes, anchorbase]
	\draw [very thick] (-1,-1) node[below,yshift=2pt]{$k$} to [out=90,in=210] (0,.75);
	\draw [very thick] (1,-1) node[below,yshift=2pt]{$\ell$} to [out=90,in=330] (0,.75);
	\draw [very thick] (3,-1) node[below,yshift=2pt]{$1$} to [out=90,in=330] (1,2.5);
	\draw [very thick] (0,.75) to [out=90,in=210] node[left=-1pt]{$k{+}\ell$} (1,2.5);
	\draw [very thick] (1,2.5) to (1,4.25) node[above,yshift=-3pt]{$k{+}\ell{+}1$};
\end{tikzpicture}
\stackrel{\eqref{eq:flowvertex}}{=} 
\frac{1}{[k]}
\begin{tikzpicture}[scale=.25,tinynodes,anchorbase]
	\draw[very thick,<-] (-2,-2) node[below=-2pt]{$k$} to [out=90,in=210] (-1,0);
	\draw[very thick] (1,0) to [out=330,in=90] (2,-2) node[below=-2pt]{$\ell$};
	\draw[very thick] (-1,0) to node[below=-1pt]{$k{-}1$} (1,0);
	\draw[very thick] (-1,0) to node[pos=.7,left=-1pt]{$1$} (0,1.732);
	\draw[very thick,directed=.6] (1,0) to (0,1.732);
	\draw[very thick,directed=.4] (0,1.732) to [out=90,in=210] node[left=-1pt]{$k{+}\ell$} (1.25,3.5);
	\draw[very thick] (4,-2) node[below=-2pt]{$1$} to [out=90,in=330] (1.25,3.5);
	\draw[very thick,->] (1.25,3.5) to (1.25,5) node[above=-2pt]{$k{+}\ell{+}1$};
\end{tikzpicture}
\stackrel{(\text{\ref{eq:spn}d})}{=} 
\frac{1}{[k]}
\begin{tikzpicture}[scale=.25,tinynodes,anchorbase]
	\draw[very thick,<-] (-2,-2) node[below=-2pt]{$k$} to [out=90,in=210] (-1,0);
	\draw[very thick] (1,.5) to [out=330,in=90] (2,-2) node[below=-2pt]{$\ell$};
	\draw[very thick] (-1,0) to node[below=-1pt,xshift=1pt]{$k{-}1$} (1,.5);
	\draw[very thick,directed=.7] (2.5,2) to [out=90,in=330] (1.25,3.5);
	\draw[very thick,directed=.7] (1,.5) to [out=90,in=210] (2.5,2);
	\draw[very thick] (-1,0) to [out=90,in=210] node[left=-1pt]{$1$} (1.25,3.5);
	\draw[very thick] (4,-2) node[below=-2pt]{$1$} to [out=90,in=330] (2.5,2);
	\draw[very thick,->] (1.25,3.5) to (1.25,5) node[above=-2pt]{$k{+}\ell{+}1$};
\end{tikzpicture}
\stackrel{\eqref{eq:generalassoc}}{=} 
\frac{1}{[k]}
\begin{tikzpicture}[scale=.25,tinynodes,anchorbase]
	\draw[very thick,<-] (-1,-2) node[below=-2pt]{$k$} to [out=90,in=210] (.25,1.768);
	\draw[very thick] (3,0) to [out=210,in=90] (2,-2) node[below=-2pt]{$\ell$};
	\draw[very thick,directed=.5] (3,0) to [out=90,in=330] node[right=-1pt]{$\ell{+}1$} (2.25,1.768);
	\draw[very thick] (4,-2) node[below=-2pt]{$1$} to [out=90,in=330] (3,0);
	\draw[very thick] (0.25,1.768) to node[below=-1pt]{$k{-}1$} (2.25,1.768);
	\draw[very thick] (0.25,1.768) to node[pos=.7,left=-2pt]{$1$} (1.25,3.5);
	\draw[very thick,directed=.6] (2.25,1.768) to (1.25,3.5);
	\draw[very thick,->] (1.25,3.5) to (1.25,5) node[above=-2pt]{$k{+}\ell{+}1$};
\end{tikzpicture}
\stackrel{\eqref{eq:flowvertex}}{=} 
\begin{tikzpicture}[scale=.2, xscale=-1,tinynodes, anchorbase]
	\draw [very thick] (-1,-1) node[below,yshift=2pt]{$1$} to [out=90,in=210] (0,.75);
	\draw [very thick] (1,-1) node[below,yshift=2pt]{$\ell$} to [out=90,in=330] (0,.75);
	\draw [very thick] (3,-1) node[below,yshift=2pt]{$k$} to [out=90,in=330] (1,2.5);
	\draw [very thick] (0,.75) to [out=90,in=210] node[right=-1pt]{$\ell{+}1$} (1,2.5);
	\draw [very thick] (1,2.5) to (1,4.25) node[above,yshift=-3pt]{$k{+}\ell{+}1$};
\end{tikzpicture}.
\end{equation}
The general case of \eqref{eq:generalassoc} follows from a similar computation:
\begin{equation}
\begin{tikzpicture}[scale=.2, tinynodes, anchorbase]
	\draw [very thick] (-1,-1) node[below,yshift=2pt]{$k$} to [out=90,in=210] (0,.75);
	\draw [very thick] (1,-1) node[below,yshift=2pt]{$\ell$} to [out=90,in=330] (0,.75);
	\draw [very thick] (3,-1) node[below,yshift=2pt]{$m$} to [out=90,in=330] (1,2.5);
	\draw [very thick] (0,.75) to [out=90,in=210] node[left=-1pt]{$k{+}\ell$} (1,2.5);
	\draw [very thick] (1,2.5) to (1,4.25) node[above,yshift=-3pt]{$k{+}\ell{+}m$};
\end{tikzpicture}
\stackrel{\eqref{eq:flowvertex}}{=} 
\frac{1}{[m]}
\begin{tikzpicture}[scale=.25,tinynodes,anchorbase,xscale=-1]
	\draw[very thick,<-] (-1,-2) node[below=-2pt]{$m$} to [out=90,in=210] (.25,1.768);
	\draw[very thick] (3,0) to [out=210,in=90] (2,-2) node[below=-2pt]{$\ell$};
	\draw[very thick,directed=.5] (3,0) to [out=90,in=330] node[left=-1pt]{$k{+}\ell$} (2.25,1.768);
	\draw[very thick] (4,-2) node[below=-2pt]{$k$} to [out=90,in=330] (3,0);
	\draw[very thick] (0.25,1.768) to node[below=-1pt]{$m{-}1$} (2.25,1.768);
	\draw[very thick] (0.25,1.768) to node[pos=.7,right=-2pt]{$1$} (1.25,3.5);
	\draw[very thick,directed=.6] (2.25,1.768) to (1.25,3.5);
	\draw[very thick,->] (1.25,3.5) to (1.25,5) node[above=-2pt]{$k{+}\ell{+}m$};
\end{tikzpicture}
\stackrel{\eqref{eq:generalassoc}}{=} 
\frac{1}{[m]}
\begin{tikzpicture}[scale=.25,tinynodes,anchorbase,xscale=-1]
	\draw[very thick,<-] (-2,-2) node[below=-2pt]{$m$} to [out=90,in=210] (-1,0);
	\draw[very thick] (1,.5) to [out=330,in=90] (2,-2) node[below=-2pt]{$\ell$};
	\draw[very thick] (-1,0) to node[below=-1pt,xshift=-1pt]{$m{-}1$} (1,.5);
	\draw[very thick,directed=.7] (2.5,2) to [out=90,in=330] (1.25,3.5);
	\draw[very thick,directed=.7] (1,.5) to [out=90,in=210] (2.5,2);
	\draw[very thick] (-1,0) to [out=90,in=210] node[right=-1pt]{$1$} (1.25,3.5);
	\draw[very thick] (4,-2) node[below=-2pt]{$k$} to [out=90,in=330] (2.5,2);
	\draw[very thick,->] (1.25,3.5) to (1.25,5) node[above=-2pt]{$k{+}\ell{+}m$};
\end{tikzpicture}
\stackrel{\eqref{eq:genassocm=1}}{=}
\frac{1}{[m]}
\begin{tikzpicture}[scale=.25,tinynodes,anchorbase,xscale=-1]
	\draw[very thick,<-] (-2,-2) node[below=-2pt]{$m$} to [out=90,in=210] (-1,0);
	\draw[very thick] (1,0) to [out=330,in=90] (2,-2) node[below=-2pt]{$\ell$};
	\draw[very thick] (-1,0) to node[below=-1pt]{$m{-}1$} (1,0);
	\draw[very thick] (-1,0) to node[pos=.7,right=-1pt]{$1$} (0,1.732);
	\draw[very thick,directed=.6] (1,0) to (0,1.732);
	\draw[very thick,directed=.4] (0,1.732) to [out=90,in=210] node[right=-1pt]{$\ell{+}m$} (1.25,3.5);
	\draw[very thick] (4,-2) node[below=-2pt]{$k$} to [out=90,in=330] (1.25,3.5);
	\draw[very thick,->] (1.25,3.5) to (1.25,5) node[above=-2pt]{$k{+}\ell{+}1$};
\end{tikzpicture}
\stackrel{\eqref{eq:flowvertex}}{=} 
\begin{tikzpicture}[scale=.2, xscale=-1,tinynodes, anchorbase]
	\draw [very thick] (-1,-1) node[below,yshift=2pt]{$m$} to [out=90,in=210] (0,.75);
	\draw [very thick] (1,-1) node[below,yshift=2pt]{$\ell$} to [out=90,in=330] (0,.75);
	\draw [very thick] (3,-1) node[below,yshift=2pt]{$k$} to [out=90,in=330] (1,2.5);
	\draw [very thick] (0,.75) to [out=90,in=210] node[right=-1pt]{$\ell{+}m$} (1,2.5);
	\draw [very thick] (1,2.5) to (1,4.25) node[above,yshift=-3pt]{$k{+}\ell{+}m$};
\end{tikzpicture}.
\end{equation}

Finally, suppose that we have verified (\text{\ref{eq:spnOther}c}) when $k+\ell \leq N$ 
(the non-trivial $N=2$ base case is an instance of (\text{\ref{eq:spn}c})). 
We then compute
\begin{equation}
\begin{tikzpicture}[scale=.25,tinynodes, anchorbase]
	\draw[very thick] (0,.75) to (0,2.5) node[above,yshift=-3pt]{$k{+}\ell{+}1$};
	\draw[very thick] (0,-2.75) to [out=30,in=330] node[right,xshift=-2pt]{$\ell$} (0,.75);
	\draw[very thick] (0,-2.75) to [out=150,in=210] node[left,xshift=2pt]{$k{+}1$} (0,.75);
	\draw[very thick] (0,-4.5) node[below,yshift=2pt]{$k{+}\ell{+}1$} to (0,-2.75);
\end{tikzpicture}
\stackrel{(\text{\ref{eq:spn}c})}{=}
\frac{1}{[k+1]}
\begin{tikzpicture}[scale=.25,tinynodes, anchorbase,xscale=1.25]
	\draw[very thick] (0,.75) to (0,2.5) node[above,yshift=-3pt]{$k{+}\ell{+}1$};
	\draw[very thick] (0,-2.75) to [out=30,in=330] node[right,xshift=-2pt]{$\ell$} (0,.75);
	\draw[very thick] (0,-2.75) to [out=150,in=210] node[left,xshift=2pt]{$1$} (0,.75);
	\draw[very thick] (0,-4.5) node[below,yshift=2pt]{$k{+}\ell{+}1$} to (0,-2.75);
	\draw[very thick] (-.375,-2.5) to [out=30,in=330] node[left=-3pt]{$k$} (-.375,.5);
\end{tikzpicture}
\stackrel{\eqref{eq:generalassoc}}{=}
\frac{1}{[k+1]}
\begin{tikzpicture}[scale=.25,tinynodes, anchorbase,xscale=-1.25]
	\draw[very thick] (0,.75) to (0,2.5) node[above,yshift=-3pt]{$k{+}\ell{+}1$};
	\draw[very thick] (0,-2.75) to [out=30,in=330] node[left=-2pt]{$1$} (0,.75);
	\draw[very thick] (0,-2.75) to [out=150,in=210] node[right=-2pt]{$\ell$} (0,.75);
	\draw[very thick] (0,-4.5) node[below,yshift=2pt]{$k{+}\ell{+}1$} to (0,-2.75);
	\draw[very thick] (-.375,-2.5) to [out=30,in=330] node[right=-3pt]{$k$} (-.375,.5);
\end{tikzpicture}
\stackrel{(\text{\ref{eq:spnOther}c})}{=}
\frac{1}{[k+1]}
{k+\ell \brack k}
\begin{tikzpicture}[scale=.25,tinynodes, anchorbase]
	\draw[very thick] (0,.75) to (0,2.5) node[above,yshift=-3pt]{$k{+}\ell{+}1$};
	\draw[very thick] (0,-2.75) to [out=30,in=330] node[right,xshift=-2pt]{$k{+}\ell$} (0,.75);
	\draw[very thick] (0,-2.75) to [out=150,in=210] node[left,xshift=2pt]{$1$} (0,.75);
	\draw[very thick] (0,-4.5) node[below,yshift=2pt]{$k{+}\ell{+}1$} to (0,-2.75);
\end{tikzpicture}
\stackrel{(\text{\ref{eq:spn}c})}{=}
{k+\ell+1 \brack k+1}
\begin{tikzpicture}[scale=.175, tinynodes, anchorbase]
	\draw [very thick] (0,-4.5) node[below,yshift=2pt]{$k{+}\ell$} to (0,2.5);
\end{tikzpicture}.
\end{equation}
\end{proof}

\subsection{Another presentation}\label{subsec:morerelns}

We now record some important additional relations between $\spn$ webs. Each is given in analogy with one of the relations in \eqref{eq:spn}.

\begin{thm}\label{thm:spnOther}
The following relations hold in $\Web(\spn)$.
\begin{equation}\label{eq:spnOther}
\begin{gathered}
(\text{\ref{eq:spnOther}a}) \;\;
\begin{tikzpicture}[scale =.75, tinynodes,anchorbase]
	\draw[very thick] (0,0) node[left,xshift=-8pt]{$k$} circle (.5);
\end{tikzpicture}
= (-1)^k\frac{[n{-}k{+}1]}{[n{+}1]}{2n{+}2 \brack k}
\quad , \quad
(\text{\ref{eq:spnOther}b}) \;\;
\begin{tikzpicture}[scale=.175,tinynodes, anchorbase]
	\draw [very thick] (0,.75) to (0,2.5) node[above,yshift=-3pt]{$k{+}2$};
	\draw [very thick] (0,-2.75) to [out=30,in=330] node[right,xshift=-2pt]{$k{+}1$} (0,.75);
	\draw [very thick] (0,-2.75) to [out=150,in=210] node[left,xshift=2pt]{$1$} (0,.75);
	\draw [very thick] (0,-4.5) node[below,yshift=2pt]{$k$} to (0,-2.75);
\end{tikzpicture}
= 0
\quad , \quad
(\text{\ref{eq:spnOther}c}) \;\;
\begin{tikzpicture}[scale=.175,tinynodes, anchorbase]
	\draw [very thick] (0,.75) to (0,2.5) node[above,yshift=-3pt]{$k{+}\ell$};
	\draw [very thick] (0,-2.75) to [out=30,in=330] node[right,xshift=-2pt]{$\ell$} (0,.75);
	\draw [very thick] (0,-2.75) to [out=150,in=210] node[left,xshift=2pt]{$k$} (0,.75);
	\draw [very thick] (0,-4.5) node[below,yshift=2pt]{$k{+}\ell$} to (0,-2.75);
\end{tikzpicture}
= {k+\ell \brack k}
\begin{tikzpicture}[scale=.175, tinynodes, anchorbase]
	\draw [very thick] (0,-4.5) node[below,yshift=2pt]{$k{+}\ell$} to (0,2.5);
\end{tikzpicture} \\
(\text{\ref{eq:spnOther}d}) \;\;
\begin{tikzpicture}[scale=.2, xscale=-1,tinynodes, anchorbase]
	\draw [very thick] (-1,-1) node[below,yshift=2pt]{$m$} to [out=90,in=210] (0,.75);
	\draw [very thick] (1,-1) node[below,yshift=2pt]{$\ell$} to [out=90,in=330] (0,.75);
	\draw [very thick] (3,-1) node[below,yshift=2pt]{$k$} to [out=90,in=330] (1,2.5);
	\draw [very thick] (0,.75) to [out=90,in=210] node[right=-2pt]{$\ell{+}m$} (1,2.5);
	\draw [very thick] (1,2.5) to (1,4.25) node[above,yshift=-3pt]{$k{+}\ell{+}m$};
\end{tikzpicture}
=
\begin{tikzpicture}[scale=.2, tinynodes, anchorbase]
	\draw [very thick] (-1,-1) node[below,yshift=2pt]{$k$} to [out=90,in=210] (0,.75);
	\draw [very thick] (1,-1) node[below,yshift=2pt]{$\ell$} to [out=90,in=330] (0,.75);
	\draw [very thick] (3,-1) node[below,yshift=2pt]{$m$} to [out=90,in=330] (1,2.5);
	\draw [very thick] (0,.75) to [out=90,in=210] node[left=-2pt]{$k{+}\ell$} (1,2.5);
	\draw [very thick] (1,2.5) to (1,4.25) node[above,yshift=-3pt]{$k{+}\ell{+}m$};
\end{tikzpicture}
\quad , \quad
(\text{\ref{eq:spnOther}e}) \;\;
\begin{tikzpicture}[scale=.35, rotate=90, tinynodes, anchorbase]
	\draw[very thick] (-1,0) node[below,yshift=2pt,xshift=2pt]{$k$} to (0,1);
	\draw[very thick] (1,0) node[above,yshift=-4pt,xshift=2pt]{$1$} to (0,1);
	\draw[very thick] (0,2.5) to (-1,3.5) node[below,yshift=2pt,xshift=-2pt]{$1$};
	\draw[very thick] (0,2.5) to (1,3.5) node[above,yshift=-4pt,xshift=-2pt]{$k$};
	\draw[very thick] (0,1) to node[below,yshift=2pt]{$k{+}1$} (0,2.5);
\end{tikzpicture}
=
\begin{tikzpicture}[scale=.3, tinynodes, anchorbase]
	\draw[very thick] (-1,0) node[below,yshift=2pt,xshift=-2pt]{$1$} to (0,1);
	\draw[very thick] (1,0) node[below,yshift=2pt,xshift=2pt]{$k$} to (0,1);
	\draw[very thick] (0,2.5) to (-1,3.5) node[above,yshift=-4pt,xshift=-2pt]{$k$};
	\draw[very thick] (0,2.5) to (1,3.5) node[above,yshift=-4pt,xshift=2pt]{$1$};
	\draw[very thick] (0,1) to node[left,xshift=2pt]{$k{+}1$} (0,2.5);
\end{tikzpicture}
+
\frac{[n{-}k]}{[n{-}k{+}1]}
\left(
\begin{tikzpicture}[scale=.3, tinynodes, anchorbase]
	\draw[very thick] (-1,0) node[below,yshift=2pt,xshift=-2pt]{$1$} to (0,1);
	\draw[very thick] (1,0) node[below,yshift=2pt,xshift=2pt]{$k$} to (0,1);
	\draw[very thick] (0,2.5) to (-1,3.5) node[above,yshift=-4pt,xshift=-2pt]{$k$};
	\draw[very thick] (0,2.5) to (1,3.5) node[above,yshift=-4pt,xshift=2pt]{$1$};
	\draw[very thick] (0,1) to node[left,xshift=2pt]{$k{-}1$} (0,2.5);
\end{tikzpicture}
-
\begin{tikzpicture}[scale=.35, rotate=90, tinynodes, anchorbase]
	\draw[very thick] (-1,0) node[below,yshift=2pt,xshift=2pt]{$k$} to (0,1);
	\draw[very thick] (1,0) node[above,yshift=-4pt,xshift=2pt]{$1$} to (0,1);
	\draw[very thick] (0,2.5) to (-1,3.5) node[below,yshift=2pt,xshift=-2pt]{$1$};
	\draw[very thick] (0,2.5) to (1,3.5) node[above,yshift=-4pt,xshift=-2pt]{$k$};
	\draw[very thick] (0,1) to node[below,yshift=2pt]{$k{-}1$} (0,2.5);
\end{tikzpicture}
\right)
\end{gathered}
\end{equation}
Here, the general flow vertices are defined as in \eqref{eq:flowvertex}.
Moreover, these equations could instead be taken as the defining relations in $\Web(\spn)$. 
\end{thm}

\begin{proof}[Proof (sketch).]
Relations (\text{\ref{eq:spnOther}d}) and (\text{\ref{eq:spnOther}c}) are established in Lemma \ref{lem:extrarels}.
Since we do not use the other relations in this paper, we leave the task of deducing them from \eqref{eq:spn} as exercises (of varying difficulty!).

However, we also note that Theorem \ref{thm:main} provides a roundabout means to deduce these relations.
Indeed, it immediately implies that (\text{\ref{eq:spnOther}a}) and (\text{\ref{eq:spnOther}b}) hold.
Further, it pairs with the techniques used in the proof of Theorem \ref{thm:functor} to give (\text{\ref{eq:spnOther}e}).
(Note: this argument is not circular, since only the relations established in Lemma \ref{lem:extrarels} are used in the proof of Theorem \ref{thm:main}.)

Finally, the equivalence to the presentation given in Definition \ref{def:webs} follows since the first four relations in \eqref{eq:spn} are special cases of the first four relations in \eqref{eq:spnOther}.
Deducing (\text{\ref{eq:spn}e}) from (\text{\ref{eq:spnOther}e}) (and the other relations) is again left as a (difficult) exercise.
\end{proof}

\begin{rem} We have chosen to start this paper with the relations of \eqref{eq:spn} because they appear to be a minimal set of relations. 
We conjecture that none of the relations in \eqref{eq:spn} can be removed,
aside from trivial special cases (e.g. the $k=1$ case of (\text{\ref{eq:spn}c}) or the $k=0$ case of (\text{\ref{eq:spn}e})).
\end{rem}

\begin{rem}\label{rem:Z1}
A presentation of a category equips the diagrammatic category with an integral form, defined over the smallest ring containing all the coefficients appearing in the relations. 
In the case of $\Web(\spn)$ (for either the presentation given in Definition \ref{def:webs} or Theorem \ref{thm:spnOther}),
that ring $\Bbbk$ is the extension of $\Z[q,q^{-1}]$ where the quantum number $[k]$ is inverted for all $1 \le k \le n$. 
(The circle relation (\text{\ref{eq:spn}a}) appears to have $[n+1]$ in the denominator, but this is cancelled by $[2n+2]$ in the numerator.) 
Meanwhile, the condition that $[k]$ is invertible for all $1 \le k \le n$ is also necessary and sufficient for the fundamental representations $V_k$ to all be self-dual and tilting. 
This provides an integral form of $\FRep(U_q(\spn))$ over $\Bbbk$, whose Karoubi envelope agrees with the category of all tilting modules, after base change to a field. 

Thus, we are in the situation of having two categories, both of which are defined over $\Bbbk$: 
our diagrammatic category $\Web(\spn)$ and a combinatorial subcategory of the category of representations. 
A priori, it is not at all clear whether these two categories are equivalent; however, we do expect this to be the case. 
The only obstruction to proving this stronger result comes from the indirect methods we use to build the functor $\Phi \colon \Web(\spn) \to \FRep(U_q(\spn))$. 
The argument that $\Phi$ exists requires the ability to invert various unknown scalars $c\in \C(q) \smallsetminus \{0\}$ (which we may not be able to invert in $\Bbbk$), 
and the complete reducibility of tensor products in $\Rep(U_q(\spn))$ (which fails over arbitrary fields). 
In our follow-up work \cite{BERT2}, we will construct this evaluation functor over $\Bbbk$, and prove that it is an equivalence. 
In particular, this will extend our results to give a diagrammatic description of tilting modules in type $C_n$ when the characteristic (or quantum characteristic) is larger than $n$. 
\end{rem}

\subsection{More trivalent vertices}\label{subsec:evenmore}

We've defined flow vertices, which will eventually be proven to span the one-dimensional morphism space $V_k \ot V_\ell \to V_{k+\ell}$. 
However, tensor product decomposition rules in type $C$ imply that $V_{m}$ is a direct summand of $V_k \ot V_\ell$ (with multiplicity one) only if
\begin{equation} \label{eq:inequalitiesfortri} m \le k + \ell, \quad k \le \ell + m, \quad \ell \le k + m, \quad k + \ell + m \text{ is even}. \end{equation}
We can use flow vertices to define trivalent vertices that span the morphism space $V_k \ot V_\ell \to V_m$ as well.

\begin{defn} Let $k,\ell,m \in \N$ satisfy \eqref{eq:inequalitiesfortri}. 
The \emph{general trivalent vertices} are defined using the flow vertices via:
\begin{equation} \label{eq:generaltri}
\begin{tikzpicture}[scale =.5, smallnodes,anchorbase]
	\draw[very thick] (0,0) node[below=-1pt]{$k$} to [out=90,in=210] (.5,.75);
	\draw[very thick] (1,0) node[below=-1pt]{$\ell$} to [out=90,in=330] (.5,.75);
	\draw[very thick] (.5,.75) to (.5,1.5) node[above=-2pt]{$m$};
\end{tikzpicture}
:=
\begin{tikzpicture}[scale=.35,smallnodes,anchorbase]
	\draw[very thick] (-1,0) to node[below,yshift=-2pt]{$\frac{k{+}\ell{-}m}{2}$} (1,0);
	\draw[very thick] (-1,0) to node[left,xshift=2pt]{$\frac{k{+}m{-}\ell}{2}$} (0,1.732);
	\draw[very thick] (1,0) to node[right,xshift=-2pt]{$\frac{\ell{+}m{-}k}{2}$} (0,1.732);
	\draw[very thick,->] (0,1.732) to (0,3.232) node[above=-2pt]{$m$};
	\draw[very thick,<-] (-2.3,-.75) node[below=-2pt,xshift=-2pt]{$k$} to (-1,0);
	\draw[very thick,<-] (2.3,-.75) node[below=-2pt,xshift=2pt]{$\ell$} to (1,0);
\end{tikzpicture}
\end{equation}
\end{defn}

Conversely, easy linear algebra on the indices implies that any triangle of the form:
\begin{equation} \label{eq:pointingout}
\begin{tikzpicture}[scale=.35,smallnodes,anchorbase]
	\draw[very thick] (-1,0) to node[below=-2pt]{$a$} (1,0);
	\draw[very thick] (-1,0) to node[left,xshift=2pt]{$b$} (0,1.732);
	\draw[very thick] (1,0) to node[right,xshift=-2pt]{$c$} (0,1.732);
	\draw[very thick,->] (0,1.732) to (0,3.232) node[above=-2pt]{$m$};
	\draw[very thick,<-] (-2.3,-.75) node[below=-2pt,xshift=-2pt]{$k$} to (-1,0);
	\draw[very thick,<-] (2.3,-.75) node[below=-2pt,xshift=2pt]{$\ell$} to (1,0);
\end{tikzpicture}
\end{equation}
must be a general trivalent vertex as in \eqref{eq:generaltri}, i.e. $a = \frac{k+\ell - m}{2}$, etcetera.

We permit ourselves to use generalized trivalent vertices in an $\spn$ web, though in this paper we typically use only flow vertices. 
For example, we can rewrite (\text{\ref{eq:spn}e}) because it contains a general trivalent vertex with boundary $(k,k,2)$:
\begin{equation} \label{eq:HIredux}
\begin{tikzpicture}[scale=.4, rotate=90, tinynodes, anchorbase]
	\draw[very thick] (-1,0) node[below,yshift=2pt,xshift=2pt]{$k$} to (0,1);
	\draw[very thick] (1,0) node[above,yshift=-4pt,xshift=2pt]{$1$} to (0,1);
	\draw[very thick] (0,2.5) to (-1,3.5) node[below,yshift=2pt,xshift=-2pt]{$k$};
	\draw[very thick] (0,2.5) to (1,3.5) node[above,yshift=-4pt,xshift=-2pt]{$1$};
	\draw[very thick,directed=.3,rdirected=.85] (0,1) to node[below=0pt]{$k{+}1$} (0,2.5);
\end{tikzpicture}
=
\begin{tikzpicture}[scale=.4, tinynodes, anchorbase]
	\draw[very thick] (-1,0) node[below,yshift=2pt,xshift=-2pt]{$k$} to (0,1);
	\draw[very thick] (1,0) node[below=-2pt,xshift=2pt]{$k$} to (0,1);
	\draw[very thick] (0,2.5) to (-1,3.5) node[above=-3pt,xshift=-2pt]{$1$};
	\draw[very thick] (0,2.5) to (1,3.5) node[above=-3pt,xshift=2pt]{$1$};
	\draw[very thick,rdirected=.85] (0,1) to node[right=-2pt]{$2$} (0,2.5);
\end{tikzpicture}
-\frac{[n{-}k]}{[n{-}k{+}1]}
\begin{tikzpicture}[scale=.4, rotate=90, tinynodes, anchorbase]
	\draw[very thick,<-] (-1,0) node[below,yshift=2pt,xshift=2pt]{$k$} to (0,1);
	\draw[very thick] (1,0) node[above,yshift=-4pt,xshift=2pt]{$1$} to (0,1);
	\draw[very thick,->] (0,2.5) to (-1,3.5) node[below,yshift=2pt,xshift=-2pt]{$k$};
	\draw[very thick] (0,2.5) to (1,3.5) node[above,yshift=-4pt,xshift=-2pt]{$1$};
	\draw[very thick] (0,1) to node[below,yshift=2pt]{$k{-}1$} (0,2.5);
\end{tikzpicture}
+
\frac{[n{-}k]}{[n]}
\begin{tikzpicture}[scale=.4, tinynodes, anchorbase]
	\draw[very thick] (-1,0) node[below,yshift=2pt]{$k$} to [out=90,in=180] (0,1) 
		to [out=0,in=90] (1,0);
	\draw[very thick] (-1,3) node[above,yshift=-3pt]{$1$} to [out=270,in=180] (0,2)
		to [out=0,in=270] (1,3);
\end{tikzpicture}
\;\; . \;\;
\end{equation}

There are many interesting relations involving general trivalent vertices, which we will explore in the sequel. It may be helpful to consider some of these here.

We said that \eqref{eq:inequalitiesfortri} is a necessary condition for $V_m$ to be a direct summand of $V_k \ot V_\ell$, but it is not a sufficient condition. Decomposition
patterns change when the values of $k$ and $\ell$ are large relative to $n$. For example, the tensor product $V_n \ot V_n$ does not have $V_k$ as a summand for any $k$, and $V_n \ot
V_{n-1}$ only has $V_1$ as a summand. Consequently, the trivalent vertices with boundary $(n,n,2)$ or $(n,n-1,3)$ are zero, among others. One can deduce that the trivalent vertex
with boundary $(n,n,2)$ is zero directly from the $k=n$ case of \eqref{eq:HIredux}.

Although the morphism space $V_k \ot V_\ell \to V_m$ is at most one dimensional, there are many ways to make diagrams of the form 
\[
\begin{tikzpicture}[scale=.25,smallnodes,anchorbase]
	\draw[very thick] (-1,0) to node[below=-2pt]{$a$} (1,0);
	\draw[very thick] (-1,0) to node[left,xshift=2pt]{$b$} (0,1.732);
	\draw[very thick] (1,0) to node[right,xshift=-2pt]{$c$} (0,1.732);
	\draw[very thick] (0,1.732) to (0,3.232) node[above=-2pt]{$m$};
	\draw[very thick] (-2.3,-.75) node[below=-2pt,xshift=-2pt]{$k$} to (-1,0);
	\draw[very thick] (2.3,-.75) node[below=-2pt,xshift=2pt]{$\ell$} to (1,0);
\end{tikzpicture}.
\]
The trivalent vertices might not all be flow vertices, or they might be flow vertices which do not point out of the triangle, unlike the vertices in \eqref{eq:pointingout}. 
Any such triangle must be a scalar multiple of the general trivalent vertex, and the formula for this scalar in terms of $(a,b,c,k,\ell,m)$ is surprisingly messy (the proof is even worse). 
We postpone a comprehensive discussion until \cite{BERT2}, teasing the appetite with a few examples.

For example, by placing a $(1,1,2)$ trivalent vertex on top of \eqref{eq:HIredux} we can prove that
\begin{equation}
\begin{tikzpicture}[scale=.25,smallnodes,anchorbase]
	\draw[very thick,directed=.4,rdirected=.75] (-1,0) to node[below=1pt]{$k{+}1$} (1,0);
	\draw[very thick] (-1,0) to (0,1.732);
	\draw[very thick] (1,0) to (0,1.732);
	\draw[very thick,->] (0,1.732) to (0,3.232) node[above=-2pt]{$2$};
	\draw[very thick] (-2.3,-.75) node[below=-2pt, xshift=-2pt]{$k$} to (-1,0);
	\draw[very thick] (2.3,-.75) node[below=-2pt,xshift=2pt]{$k$} to (1,0);
\end{tikzpicture}
= \frac{[n+2-k]}{[n+1-k]}
\begin{tikzpicture}[scale =.5, smallnodes,anchorbase]
	\draw[very thick] (0,0) node[below=-1pt]{$k$} to [out=90,in=210] (.5,.75);
	\draw[very thick] (1,0) node[below=-1pt]{$k$} to [out=90,in=330] (.5,.75);
	\draw[very thick] (.5,.75) to (.5,1.5) node[above=-2pt]{$2$};
\end{tikzpicture}.
\end{equation}
As another example, in the sequel we will prove:
\begin{equation}
\begin{tikzpicture}[scale=.25,smallnodes,anchorbase]
	\draw[very thick] (-1,0) to node[below=-2pt]{$a$} (1,0);
	\draw[very thick] (-1,0) to node[left,xshift=2pt]{$b$} (0,1.732);
	\draw[very thick] (1,0) to node[right,xshift=-2pt]{$c$} (0,1.732);
	\draw[very thick] (0,1.732) to (0,3.232) node[above=-2pt]{$m$};
	\draw[very thick,<-] (-2.3,-.75) node[below=-2pt,xshift=-2pt]{$k$} to (-1,0);
	\draw[very thick,<-] (2.3,-.75) node[below=-2pt,xshift=2pt]{$\ell$} to (1,0);
\end{tikzpicture}
=
{\frac{k+\ell-m}{2} \brack a}
\begin{tikzpicture}[scale =.5, smallnodes,anchorbase]
	\draw[very thick] (0,0) node[below=-1pt]{$k$} to [out=90,in=210] (.5,.75);
	\draw[very thick] (1,0) node[below=-1pt]{$\ell$} to [out=90,in=330] (.5,.75);
	\draw[very thick] (.5,.75) to (.5,1.5) node[above=-2pt]{$m$};
\end{tikzpicture} .
\end{equation}

\begin{rem} For our presentation of $\FRep(U_q(\spn))$ to be correct, it must be possible to simplify all triangles directly from the relations \eqref{eq:spn}, which is a
surprisingly arduous task. We were pleased to find a proof that our presentation is correct without needing to do any of this work! \end{rem}

\section{Outline of the proof}\label{sec:outline}

The details of the proof of Theorem \ref{thm:main} will occupy the entirety of \S \ref{sec:proof}.
For the reader's benefit, we give a condensed and informal account of the pertinent steps.

\subsection{Step 1: Produce the functor.}
First, we show the existence of a pivotal, $\C(q)$-linear, essentially surjective functor $\Phi \colon \Web(\spn) \to \FRep(U_q(\spn))$. 
We do this in a roundabout fashion in \S \ref{sec:thefunctor}, rather than constructing it directly. 
One can compute the dimensions of certain morphism spaces between representations with ease, 
e.g. the morphism spaces $\Hom_{U_q(\spn)}(V_{1} \ot V_{k}, V_{k+1})$ and $\Hom_{U_q(\spn)}(V_{k} \ot V_{1}, V_{k+1})$ are one-dimensional.
Thus, the choice of where $\Phi$ sends the generating trivalent vertices \eqref{eq:WebGen} is unique up to scalar. 
Choosing these morphisms gives an essentially surjective monoidal functor from the pivotal category freely generated by our trivalent vertices. 
We then use properties of $\FRep(U_q(\spn))$ to show that, after appropriate rescalings, our relations must be satisfied in the image; see Theorem \ref{thm:functor}. 
This construction of $\Phi$ is analogous to the argument used in \cite{RoseTatham} for type $C_3$.

For example, the bigon $k \to 1 \ot (k-1) \to k$ (the first diagram in (\text{\ref{eq:spn}c})) must be sent to some morphism in the one-dimensional space $\End_{U_q(\spn)}(V_{k})$, 
so its image in $\Rep (U_q(\spn))$ is a scalar multiple of the identity. Initially, this scalar is unknown.
Similarly, other relations must hold (with unknown coefficients) using $\Hom$-space dimension computations. 
Some of these coefficients (e.g. the value of the circle) are determined from properties inherent to the pivotal/ribbon structure on $\FRep(U_q(\spn))$. 
Formulas relating the coefficients can further be determined using coherence requirements, since resolving a diagram using two different relations must produce the same answer. 
In the end, we show that our relations must be satisfied in the image of the functor, and further are the only ones that work (up to equivalences that rescale the generators).

\subsection{Step 2: Reduce to a full subcategory and use the BMW algebra to show fullness}

Next, in Lemma \ref{lem:reduction} we provide a quick argument that shows that it suffices to show that $\Phi$ gives 
an isomorphism on $\Hom$-spaces of the form $\Hom_{\Web(\spn)}(1^{\ot k}, 1^{\ot l})$, for various $k$ and $l$. 
The main point is that there exist webs $k \xrightarrow{\iota} 1^{\ot k}$ and $1^{\ot k} \xrightarrow{\pi} k$ so that $\pi \circ \iota = \id$.
This parallels the fact that, over $\C(q)$, $V_k$ is a summand of $V_1^{\otimes k}$. 
This allows us to use earlier work on $\SRep(U_q(\spn))$ to assist in the proof (recall that $\SRep(U_q(\spn))$ is the full subcategory whose objects are tensor products of $V_1$). 
In particular, we show that the surjective map from the BMW algebra to endomorphism spaces in $\SRep(U_q(\spn))$ factors through $\Web(\spn)$, 
and deduce that the functor $\Phi \colon \Web(\spn) \to \FRep(U_q(\spn))$ is full using the pivotal structure.
This fullness places a lower bound on the dimension of the $\Hom$-space $\Hom_{\Web(\spn)}(1^{\ot k}, 1^{\ot l})$.
The dimension of the corresponding $\Hom$ space in $\SRep(U_q(\spn))$ was computed by Sundaram \cite{SundaramThesis,SundaramTableaux}; 
see \S \ref{sec:faithful} for more details. 

\subsection{Step 3: change basis to show faithfulness}

Our remaining and most-arduous task is to prove that $\Phi$ is faithful on $\Hom_{\Web(\spn)}(1^{\ot k},1^{\ot l})$. 
By our fullness result, it suffices to place an upper bound on the dimension of this Hom space in $\Web(\spn)$ that matches the known dimension of $\Hom_{U_q(\spn)}(V_{1}^{\ot k}, V_{1}^{\ot l})$.
Our approach to studying $\SRep(U_q(\spn))$ differs from previous works in that we choose a different generating morphism. 
The BMW algebra uses the braiding map $1 \ot 1 \to 1 \ot 1$ as one of its generators, but this braiding map is not rotationally invariant: rotating the braiding gives the inverse braiding. 
We instead choose a rotationally-invariant planar crossing map, defined in terms of type $C$ webs as follows.
\[
\begin{tikzpicture}[scale=.35, tinynodes, anchorbase]
	\draw[very thick] (-1,-1) to (1,1);
	\draw[very thick] (-1,1) to (1,-1);
\end{tikzpicture}
\; :=\; 
\begin{tikzpicture}[scale=.35, tinynodes, anchorbase]
	\draw[very thick] (-1,0) node[below,yshift=2pt,xshift=-2pt]{$1$} to (0,1);
	\draw[very thick] (1,0) node[below,yshift=2pt,xshift=2pt]{$1$} to (0,1);
	\draw[very thick] (0,2.5) to (-1,3.5) node[above,yshift=-4pt,xshift=-2pt]{$1$};
	\draw[very thick] (0,2.5) to (1,3.5) node[above,yshift=-4pt,xshift=2pt]{$1$};
	\draw[very thick] (0,1) to node[right,xshift=-2pt]{$2$} (0,2.5);
\end{tikzpicture}
\; + \;
\frac{[n-1]}{[n]} \;
\begin{tikzpicture}[scale=.3, tinynodes, anchorbase]
	\draw[very thick] (-1,0) node[below,yshift=2pt]{$1$} to [out=90,in=180] (0,1) 
		to [out=0,in=90] (1,0);
	\draw[very thick] (-1,3) node[above,yshift=-3pt]{$1$} to [out=270,in=180] (0,2)
		to [out=0,in=270] (1,3);
\end{tikzpicture}
\; \stackrel{(\text{\ref{eq:spn}e})}{=} \; 
\begin{tikzpicture}[scale=.35, rotate=90, tinynodes, anchorbase]
	\draw[very thick] (-1,0) node[below,yshift=2pt,xshift=2pt]{$1$} to (0,1);
	\draw[very thick] (1,0) node[above,yshift=-4pt,xshift=2pt]{$1$} to (0,1);
	\draw[very thick] (0,2.5) to (-1,3.5) node[below,yshift=2pt,xshift=-2pt]{$1$};
	\draw[very thick] (0,2.5) to (1,3.5) node[above,yshift=-4pt,xshift=-2pt]{$1$};
	\draw[very thick] (0,1) to node[below,yshift=2pt]{$2$} (0,2.5);
\end{tikzpicture}
\; + \;
\frac{[n-1]}{[n]} \;
\begin{tikzpicture}[scale=.35, tinynodes, anchorbase]
	\draw[very thick] (-1,-1) node[below,yshift=2pt]{$1$} to (-1,1) node[above,yshift=-2pt]{$1$};
	\draw[very thick] (1,-1) node[below,yshift=2pt]{$1$} to (1,1) node[above,yshift=-2pt]{$1$};
\end{tikzpicture}
\]
This generalizes a morphism defined in \cite[p. 127]{Kup}.
Diagrams built from this crossing, together with cups and caps, will be called \emph{quadrivalent graphs}. 
Let $\Web^\times(\spn)$ denote the subcategory of $\Web(\spn)$ whose objects are $1^{\ot k}$ for various $k\ge 0$, and whose morphisms are spanned by quadrivalent graphs. 
This is our new variant on (the $\spn$ quotient of) the quantized Brauer category. 
Note that $\Web^\times(\spn)$ is not explicitly defined by generators and relations; rather, the relations between quadrivalent graphs are implicitly determined from those in $\Web(\spn)$.

The first step in using $\Web^\times(\spn)$ to study the dimensions of $\Hom$-spaces in $\Web(\spn)$ is to confirm that $\Web^\times(\spn)$ is 
equal to $\SWeb(\spn)$, the full subcategory of $\Web(\spn)$ with objects $1^{\ot k}$.
This is accomplished using a diagrammatic argument in Lemma \ref{lem:HT}, that further proves useful in simplifying the set of quadrivalent graphs we need consider.
Next, we prove that $\Web^\times(\spn)$ possesses analogues of the Reidemeister moves from knot theory, which agree with the familiar topological relations modulo diagrams with fewer crossings. 
We emphasize that we work with quadrivalent graphs, which do not distinguish between over- and under-crossings.
For example, the analogue of Reidemeister III in $\Web^\times(\spn)$ is the relation:
\begin{equation} \label{eq:RIIItrue}
\begin{tikzpicture}[scale=.225, tinynodes, anchorbase]
	\draw[very thick] (-2,-3) to [out=90,in=270] (2,3);
	\draw[very thick] (0,-3) to [out=90,in=270] (-2,0) to [out=90,in=270] (0,3);
	\draw[very thick] (2,-3) to [out=90,in=270] (-2,3);	
\end{tikzpicture}
=
\begin{tikzpicture}[scale=.225, tinynodes, anchorbase]
	\draw[very thick] (-2,-3) to [out=90,in=270] (2,3);
	\draw[very thick] (0,-3) to [out=90,in=270] (2,0) to [out=90,in=270] (0,3);
	\draw[very thick] (2,-3) to [out=90,in=270] (-2,3);	
\end{tikzpicture}
+
\begin{tikzpicture}[scale=.15, tinynodes, anchorbase,yscale=-1]
	\draw[very thick] (-2,0) to [out=90,in=180] (-1,2) to [out=0,in=90] (0,0);
	\draw[very thick] (-2,9) to [out=270,in=180] (0,6) to [out=0,in=270] (2,9);
	\draw[very thick] (2,0) to [out=90,in=270] (0,9);
\end{tikzpicture}
-
\begin{tikzpicture}[scale=.15, tinynodes, anchorbase,yscale=-1,rotate=180]
	\draw[very thick] (-2,0) to [out=90,in=180] (-1,2) to [out=0,in=90] (0,0);
	\draw[very thick] (-2,9) to [out=270,in=180] (0,6) to [out=0,in=270] (2,9);
	\draw[very thick] (2,0) to [out=90,in=270] (0,9);
\end{tikzpicture}
+
\begin{tikzpicture}[scale=.15, tinynodes, anchorbase]
	\draw[very thick] (-2,0) to [out=90,in=180] (-1,2) to [out=0,in=90] (0,0);
	\draw[very thick] (-2,9) to [out=270,in=180] (0,6) to [out=0,in=270] (2,9);
	\draw[very thick] (2,0) to [out=90,in=270] (0,9);
\end{tikzpicture}
-
\begin{tikzpicture}[scale=.15, tinynodes, anchorbase,rotate=180]
	\draw[very thick] (-2,0) to [out=90,in=180] (-1,2) to [out=0,in=90] (0,0);
	\draw[very thick] (-2,9) to [out=270,in=180] (0,6) to [out=0,in=270] (2,9);
	\draw[very thick] (2,0) to [out=90,in=270] (0,9);
\end{tikzpicture}
+
\begin{tikzpicture}[scale=.15, tinynodes, anchorbase]
	\draw[very thick] (-2,0) to [out=90,in=270] (0,9);
	\draw[very thick] (0,0) to [out=90,in=270] (-2,9);
	\draw[very thick] (2,0) to [out=90,in=270] (2,9);
\end{tikzpicture}
-
\begin{tikzpicture}[scale=.15, tinynodes, anchorbase,rotate=180]
	\draw[very thick] (-2,0) to [out=90,in=270] (0,9);
	\draw[very thick] (0,0) to [out=90,in=270] (-2,9);
	\draw[very thick] (2,0) to [out=90,in=270] (2,9);
\end{tikzpicture}
-
[2n-2]
\left( \
\begin{tikzpicture}[scale=.15, tinynodes, anchorbase]
	\draw[very thick] (-2,0) to [out=90,in=180] (-1,2) to [out=0,in=90] (0,0);
	\draw[very thick] (-2,9) to [out=270,in=180] (-1,7) to [out=0,in=270] (0,9);
	\draw[very thick] (2,0) to [out=90,in=270] (2,9);
\end{tikzpicture}
-
\begin{tikzpicture}[scale=.15, tinynodes, anchorbase,rotate=180]
	\draw[very thick] (-2,0) to [out=90,in=180] (-1,2) to [out=0,in=90] (0,0);
	\draw[very thick] (-2,9) to [out=270,in=180] (-1,7) to [out=0,in=270] (0,9);
	\draw[very thick] (2,0) to [out=90,in=270] (2,9);
\end{tikzpicture}
\ \right)
\end{equation}

These relations allow us to reduce to a more-manageable spanning set for $\Hom$-spaces in $\Web^\times(\spn)$.
To give an analogy, a standard argument in knot theory pairs with Reidemeister's theorem to imply that any tangle diagram simplifies to a reduced diagram 
(no extraneous crossings or closed components) using the following transformations:
\begin{itemize} 
	\item Reidemeister moves RI, RII, and RIII,
	\item switching over- and under-crossings, and
	\item removing circles. 
\end{itemize}
Our results in \S \ref{sec:RGS} imply a strengthening of this result: the Reidemeister moves RI and RII need only be applied in the direction that reduces the number of crossings. 
See Porism \ref{por:tangle} below.
Rather than quote this result as ``folklore,'' we have provided detailed proofs. 
Using this, we prove that reduced graphs (one for each planar matching) span $\Web^\times(\spn)$.

Finally, we use pattern avoidance techniques to trim down this spanning set even further, to a set of elements of cardinality equal to the corresponding dimension in $\SRep(U_q(\spn))$.
In the aforementioned Lemma \ref{lem:HT}, we define the \emph{half twist} $\HT_k$ to be the crossing diagram associated to a particular reduced expression of the longest element of $\mathfrak{S}_k$,
and prove that the merge-split web $1^{\ot k} \xrightarrow{\pi} k \xrightarrow{\iota} 1^{\ot k}$ is equal to the half twist $\HT_{k}$ modulo quadrivalent graphs with fewer crossings, for all $k > 0$. 
Since the object $n+1$ is the zero object, this shows that $\HT_{n+1}$ can be rewritten using diagrams with fewer crossings in $\Web^\times(\spn)$. 
In Proposition \ref{prop:patterngivesHT}, we make an additional diagrammatic/combinatorial argument to deduce that reduced diagrams for $(n+1)$-avoiding planar matchings will span $\Web^\times(\spn)$. 
Sundaram's dimension formula is exactly the number of $(n+1)$-avoiding planar matchings (we justify this known claim in \S \ref{sec:faithful}), which concludes the proof of the faithfulness of $\Phi$.

\subsection{Additional remarks}

First, let us discuss the extent to which the same overall proof outline might work \emph{mutatis mutandis} to give a reproof of
Cautis-Kamnitzer-Morrison's result in type $A$: that their web category is equivalent to $\FRep(U_q(\gln))$ over $\C(q)$. 
There are several differences between type $C$ and type $A$ that should be mentioned. 

\begin{rem} \label{rmk:itischeatingthough} In our proof outline, we define a category $\Web^\times(\spn)$ spanned by quadrivalent graphs 
(wherein we can control the size of morphism spaces), and we have to prove that it is a full subcategory of $\Web(\spn)$, see Corollary \ref{cor:QuadFull}. 
In the context of type $A$, we might use the braiding maps coming from the action of the Hecke algebra as a replacement for quadrivalent vertices. 
However, there are also cup and cap morphisms to account for. 
Thus, one can envision two different subcategories of $\Web(\gln)$:
the subcategory spanned by braid diagrams (where one can control the size of morphism spaces), and the subcategory spanned by oriented tangle diagrams. 
(In both, objects take the form $1^{\otimes k}$, and the only $\Hom$-spaces are endomorphism spaces.)
An analogue of our half twist Lemma could be adapted to show that tangle diagrams span, but one would still need additional arguments 
(e.g. using Porism \ref{por:tangle} and analogues of Propositions \ref{prop:Matsumoto} and \ref{lem:webReidemeister}) to deduce that braid-like diagrams span. 
Cautis-Kamnitzer-Morrison skirt these issues via the ladderization technology afforded by skew Howe duality.
We anticipate that a careful adaptation of our techniques to an appropriate subcategory of $\Web(\sln)$ would give yet another reproof of 
the main result of \cite{CKM}, over $\C(q)$ (but we do not cross all the t's and dot the lower-case j's here). \end{rem}

\begin{rem} \label{rmk:reductiontoonesisntcheating} 
While Cautis-Kamnitzer-Morrison (CKM) \cite{CKM} only proved that their category of webs was equivalent to $\FRep(U_q(\gln))$ over $\C(q)$, 
they intended to provide a presentation which might be correct over $\Z[q,q^{-1}]$ (see the Remark immediately preceding \cite[Lemma 2.2.1]{CKM}). 
This was proved by the second-named author in \cite{EliasLL}. 
The CKM presentation has many ``redundant'' relations\footnote{For example, consider the ``square flop'' relation \cite[Equation (2.10)]{CKM}. 
Only the special case where the rung labels $r$ and $s$ both equal $1$ is a required relation over $\C(q)$; the remaining cases are redundant.}, 
which can be derived from other relations over $\C(q)$, though not necessarily over $\Z[q,q^{-1}]$.
Our presentation \eqref{eq:spn} looks more like their presentation with the ``redundant'' relations removed, 
and this version of the category in type $A$ would be incorrect over $\Z[q,q^{-1}]$, having torsion webs in the kernel of $\Phi$. 
These redundant relations are an important part of the legacy of \cite{CKM}, and are seemingly missing from this paper (but will appear in the sequel).
When working over $\Z[q,q^{-1}]$ in type $A$, the analogue of our proof that $\Phi$ is fully faithful will break down at Step 2: 
we cannot reduce to $\Hom(1^{\ot k},1^{\ot l})$. 
This is because $k$ is not a direct summand of $1^{\otimes k}$ (in the Karoubi envelope): 
the idempotent $\iota \circ \pi \colon 1^{\otimes k} \to 1^{\otimes k}$ can only be defined when $[k]!$ is invertible. 
Proving that type $A$ webs are correct over $\C(q)$ by reduction to $\SRep(U_q(\gln))$ abuses a loophole; 
it works, but the proof is sure to fail for the integral form.

However, the situation in type $C$ is different. 
Recall that $\Bbbk$ is the extension of $\Z[q,q^{-1}]$ obtained by inverting $[n]!$. 
The fundamental representations $V_k$ are not even tilting when $[n]!$ is not invertible, 
and one simply does not expect an integral form for $\FRep(U_q(\spn))$ with a smaller base ring than $\Bbbk$.
Even our minimal set of relations involves denominators, and our category is only defined over $\Bbbk$. Since we are forced into the loophole, we may as well abuse it! 
Indeed, with a bit more effort (defining the functor $\Phi$ explicitly and checking integrality) we could extend our results from $\C(q)$ to $\Bbbk$.
However, in the interest of space and time we save this result for our follow-up work, where we take the different route provided by the double ladder basis.
\end{rem}

Next, we revisit the topic of integral forms from Remark \ref{rem:Z1}.

\begin{rem} 
To simplify notation, set $\ZA= \Z[q, q^{-1}]$. 
Let $U^{\ZA}_q(\spn)$ denote Lusztig's divided powers form of $U_q(\spn)$ and let $V^{\ZA}$ be the Weyl module with highest weight $\varpi_1$. 
If $\Bbbk$ is any field and $q\in \Bbbk^{\times}$, then $\Bbbk\ot V^{\ZA}$ is irreducible and tilting. 
The full subcategory of $U^{\ZA}_q(\spn)$ modules monoidally generated by $V^{\ZA}$ is an 
integral form of $\SRep(U_q(\spn))$ which we will call $\SRep^{\ZA}$. 
One might expect there to be a diagrammatic presentation for $\SRep^{\ZA}$ built from quadrivalent vertices,
akin to the $\C(q)$-linear category $\SWeb(\spn) = \Web^\times(\spn)$.

As evidence, note that the coefficients in the relations (\text{\ref{eq:spn}a}), \eqref{eq:R1}, \eqref{eq:preciseR2}, and \eqref{eq:RIIItrue} in $\Web^\times(\spn)$
have no poles (fractions like $\frac{[2n]}{[n]}$ are actually Laurent polynomials, as the denominator divides the numerator).
What is currently missing in a diagrammatic description of $\SWeb^{\ZA}$ is an explicit description of the ``kernel,'' as discussed following equation \eqref{eq:BMW}.
In other words, one needs an explicit version of Lemma \ref{lem:HT} below that describes the 
split-merge map $1^{\ot n+1} \to (n+1) \to 1^{\ot n+1}$ as a linear combination of quadrivalent graphs.
When $n=2$ this is given by \eqref{eq:preciseR3}, which (in the special case of $n=2$) has no poles.
After setting the $3$-labeled strand equal to zero, \eqref{eq:preciseR3} gives a way to rewrite the half twist on three strands as an $\ZA$-linear combination of diagrams with fewer than $3$ crossings. 
When $n > 2$, the equation \eqref{eq:preciseR3} has a coefficient $[n-2]/[n-1]$ with poles, though this is no longer of relevance. 
Instead, the kernel is some relation analogous to \eqref{eq:preciseR3}, which describes the half twist on $n+1$ strands as a linear combination of diagrams with fewer crossings. 
Potentially, such an equation may exist over $\ZA$. The problem of finding an explicit form of this kernel over $\ZA$ is closely related to Conjecture $5.8$ in \cite{HuXiaoTensorBMW}. 

All summands of tensor products of $\Bbbk\ot V^{\ZA}$ will be tilting modules. 
However, we expect that the idempotent completion of the category monoidally generated by $\Bbbk\ot V^{\ZA}$ will not be the entire category of tilting modules (unless the characteristic is larger than $n$). 
For example, the group $Sp_4(\overline{\mathbb{F}}_2)$ has an indecomposable tilting module $T(\varpi_2)\cong \Lambda^2(\overline{\mathbb{F}}_2 \ot V^{\ZA})$ 
that we conjecture will not appear as a direct summand of $(\overline{\mathbb{F}}_2 \ot V^{\ZA})^{\ot k}$, for any $k$. 
\end{rem}

Finally, we return to the discussion from Remark \ref{rem:LZ1}.

\begin{rem} \label{rem:LZ2} Recall that Lehrer and Zhang study $\SRep(\spn)$ using the Brauer category. 
The quantum analogue of the latter is the quantized Brauer category, which we reminder the reader is the category obtained from the BMW category 
by adjoining the cap and cup morphisms as generators. Note that, as defined above, the BMW category only contains the cup-cap composition as a morphism. 
A priori, this implies that endomorphism spaces in the quantized Brauer category might be larger than the BMW algebra.
This parallels the distinction in type $A$ between braids and tangles that is mentioned in Remark \ref{rmk:itischeatingthough} above.

Hence, an adaptation of the approach from \cite{LZbrauer} using the quantized Brauer category as a replacement for $\Web^\times(\spn)$
would require a proof similar in outline to the one presented above (plus more!), and thus is far from trivial. 
Indeed, the complications arise since the moves needed to simplify diagrams in the quantized Brauer category to a tractable spanning set 
(e.g. to deduce that the BMW algebra spans endomorphism spaces) again hold only modulo lower order terms. 
Thus, our arguments showing that certain Reidemeister-like moves\footnote{For example, while the tangle Reidemeister II move holds in the quantizer Brauer category, 
we also need a version of this relation that expresses the full twist braid on two strands in terms of tangles with fewer crossings.}
need only be applied in the direction that reduces the number of crossings are still required.
This contrasts the combinatorially-simpler (non-quantized) Brauer category, where the relevant ``graph Reidemeister moves'' hold on the nose; see \cite[Figures 2 and 3]{LZbrauer}.
Further, the quantized Brauer approach also requires the resolution of the previously mentioned open problem of finding an explicit formula for the quantum Pfaffian.

Nevertheless, we believe it would be fruitful to further explore the relation between $\Web(\spn)$ and the quantized Brauer category.
In particular, a version of Lemma \ref{lem:HT} in the BMW context could produce a formula for the quantum Pfaffian.
\end{rem}

%
\section{The proofs}\label{sec:proof}
%

\subsection{The essentially surjective functor}\label{sec:thefunctor}

We begin by proving the following.

\begin{thm}\label{thm:functor}
There is a pivotal, essentially surjective functor $\Phi \colon \Web(\spn) \to \FRep(U_q(\spn))$.
\end{thm}

\begin{proof} First, we wish to argue that $\FRep(U_q(\spn))$ can be described by some diagrammatic calculus, where strands are unoriented and
isotopic diagrams represent equal morphisms. For this purpose, we temporarily consider $\FRep^{\pm}(U_q(\spn))$, the full subcategory of $\Rep(U_q(\spn))$ monoidally generated by
the fundamental representations \emph{and} their duals. By the coherence theorem for pivotal categories (see e.g. \cite{NS} or \cite{BarWes}), $\FRep^{\pm}(U_q(\spn))$ is
equivalent, as a pivotal category, to a strict pivotal category. Strict pivotal categories have isotopy-invariant diagrammatic descriptions with \emph{oriented} strands; 
if a strand labeled $k$ oriented upwards represents the $k^{th}$ fundamental representation, then a strand labeled $k$ oriented downwards represents its dual. 
Since fundamental representations in type $C$ are isomorphic to their duals, we might hope that by fixing such isomorphisms one could remove the orientations to obtain a presentation for $\FRep(U_q(\spn))$. 
Selinger's notion of ``coherent self-duality'' \cite{Sel2} gives a precise criterion for when this can be done: when all non-zero Frobenius-Schur indicators are $+1$. 
As recalled in \S \ref{subsec:qdim} above, 
Snyder-Tingley \cite{ST} prove that there is a (non-standard) pivotal structure on $\Rep(U_q(\spn))$ for which these Frobenius-Schur indicators are $+1$, as desired.

Thus, morphisms in $\FRep(U_q(\spn))$ are afforded their own (unoriented) diagrammatic calculus, 
which we depict in \textcolor{gray}{\textbf{gray}} to distinguish them from morphisms in our ``abstract'' web categories. 
As just discussed, morphisms in this gray calculus are invariant under planar isotopy relative to the boundary. 
We will now fix a choice of certain morphisms in $\FRep(U_q(\spn))$, and define our functor relative to these chosen morphisms.

We first fix the cap and cup morphisms:
\[
\begin{tikzpicture}[scale =.5, smallnodes, anchorbase]
	\draw[very thick, gray] (1,2.25) node[above=-2pt]{$k$} to [out=270,in=0] (.5,1.5) to [out=180,in=270] (0,2.25) node[above=-2pt]{$k$};
\end{tikzpicture}
\in \Hom_{U_q(\spn)}(V_0, V_k \otimes V_k), \qquad \begin{tikzpicture}[scale =.5, smallnodes, anchorbase, yscale=-1]
	\draw[very thick, gray] (1,2.25) node[below=-1pt]{$k$} to [out=270,in=0] (.5,1.5) to [out=180,in=270] (0,2.25) node[below=-1pt]{$k$};
\end{tikzpicture} \in \Hom_{U_q(\spn)}(V_k \otimes V_k, V_0).
\]

\begin{rem}
The existence of cap/cup (co)unit morphisms is implicit in a pivotal category, 
since the latter is an autonomous category equipped with a coherent isomorphism between objects and their double duals.
An autonomous category requires the existence of (co)unit morphisms, but does not require choices of such morphisms as 
a structure on the category; see e.g. \cite[Remark 4.1]{Sel1}.
In a linear autonomous category, rescaling an object's unit (cup) by an invertible scalar $c$ gives another valid choice of unit,
provided one also rescales the corresponding counit (cap) by $c^{-1}$ for isotopy relations to hold. 
Since the representations $V_k$ are irreducible, their (co)unit morphisms live in one-dimensional morphism spaces, and rescaling is the only freedom in the choice of (co)unit morphisms.
Such rescaling will not change the value of a circle (though changing the pivotal structure itself can change the value of a circle, see \cite{ST}).
\end{rem}

Next, standard tensor product decomposition rules imply that
\begin{equation}\label{eq:TensorDecomp}
V_1 \otimes V_k \cong V_{\varpi_1+\varpi_k} \oplus V_{k+1} \oplus V_{k-1}
\end{equation} 
(here, we interpret $V_{-1} = 0 = V_{n+1}$).
Fix an arbitrary choice of non-zero morphisms
\[
\begin{tikzpicture}[scale =.5, smallnodes,anchorbase]
	\draw[very thick, gray] (0,0) node[below]{$1$} to [out=90,in=210] (.5,.75);
	\draw[very thick, gray] (1,0) node[below]{$k$} to [out=90,in=330] (.5,.75);
	\draw[very thick, gray] (.5,.75) to (.5,1.5) node[above]{$k{+}1$};
\end{tikzpicture}
\in \Hom_{U_q(\spn)}(V_1 \otimes V_k, V_{k+1})
\quad \text{and}
\begin{tikzpicture}[scale =.5, smallnodes,anchorbase]
	\draw[very thick, gray] (0,0) node[below]{$k$} to [out=90,in=210] (.5,.75);
	\draw[very thick, gray] (1,0) node[below]{$1$} to [out=90,in=330] (.5,.75);
	\draw[very thick, gray] (.5,.75) to (.5,1.5) node[above]{$k{+}1$};
\end{tikzpicture}
\in \Hom_{U_q(\spn)}(V_k \otimes V_1, V_{k+1})
\]
for $1 \leq k \leq n-1$, which therefore span these $\Hom$-spaces. We note that these maps are surjective, since $V_{k+1}$ is irreducible.

Let $\FWeb(n)$ be the strict $\C(q)$-linear pivotal category freely generated by self-dual objects $\{1,\ldots,n\}$ and the morphisms
\[
\begin{tikzpicture}[scale =.5, smallnodes,anchorbase]
	\draw[very thick] (0,0) node[below]{$1$} to [out=90,in=210] (.5,.75);
	\draw[very thick] (1,0) node[below]{$k$} to [out=90,in=330] (.5,.75);
	\draw[very thick] (.5,.75) to (.5,1.5) node[above]{$k{+}1$};
\end{tikzpicture}
\quad \text{and} \quad
\begin{tikzpicture}[scale =.5, smallnodes,anchorbase]
	\draw[very thick] (0,0) node[below]{$k$} to [out=90,in=210] (.5,.75);
	\draw[very thick] (1,0) node[below]{$1$} to [out=90,in=330] (.5,.75);
	\draw[very thick] (.5,.75) to (.5,1.5) node[above]{$k{+}1$};
\end{tikzpicture}
\]
for $1 \leq k \leq n-1$ that generate $\Web(\spn)$.
It follows (e.g. from the results in \cite{Sel1}) that the assignments $k \mapsto V_k$ and
\begin{equation}\label{eq:images}
\begin{tikzpicture}[scale =.5, smallnodes,anchorbase]
	\draw[very thick] (0,0) node[below]{$1$} to [out=90,in=210] (.5,.75);
	\draw[very thick] (1,0) node[below]{$k$} to [out=90,in=330] (.5,.75);
	\draw[very thick] (.5,.75) to (.5,1.5) node[above]{$k{+}1$};
\end{tikzpicture}
\mapsto
a_k
\begin{tikzpicture}[scale =.5, smallnodes,anchorbase]
	\draw[very thick, gray] (0,0) node[below]{$1$} to [out=90,in=210] (.5,.75);
	\draw[very thick, gray] (1,0) node[below]{$k$} to [out=90,in=330] (.5,.75);
	\draw[very thick, gray] (.5,.75) to (.5,1.5) node[above]{$k{+}1$};
\end{tikzpicture}
\quad , \quad
\begin{tikzpicture}[scale =.5, smallnodes,anchorbase]
	\draw[very thick] (0,0) node[below]{$k$} to [out=90,in=210] (.5,.75);
	\draw[very thick] (1,0) node[below]{$1$} to [out=90,in=330] (.5,.75);
	\draw[very thick] (.5,.75) to (.5,1.5) node[above]{$k{+}1$};
\end{tikzpicture}
\mapsto
b_k
\begin{tikzpicture}[scale =.5, smallnodes,anchorbase]
	\draw[very thick, gray] (0,0) node[below]{$k$} to [out=90,in=210] (.5,.75);
	\draw[very thick, gray] (1,0) node[below]{$1$} to [out=90,in=330] (.5,.75);
	\draw[very thick, gray] (.5,.75) to (.5,1.5) node[above]{$k{+}1$};
\end{tikzpicture}
\quad , \quad
\begin{tikzpicture}[scale =.5, smallnodes, anchorbase]
	\draw[very thick] (1,2.25) node[above=-2pt]{$k$} to [out=270,in=0] (.5,1.5) to [out=180,in=270] (0,2.25) node[above=-2pt]{$k$};
\end{tikzpicture}
\mapsto
c_k
\begin{tikzpicture}[scale =.5, smallnodes, anchorbase]
	\draw[very thick, gray] (1,2.25) node[above=-2pt]{$k$} to [out=270,in=0] (.5,1.5) to [out=180,in=270] (0,2.25) node[above=-2pt]{$k$};
\end{tikzpicture} \quad , \quad
\begin{tikzpicture}[scale =.5, smallnodes, anchorbase, yscale=-1]
	\draw[very thick] (1,2.25) node[below=-1pt]{$k$} to [out=270,in=0] (.5,1.5) to [out=180,in=270] (0,2.25) node[below=-1pt]{$k$};
\end{tikzpicture}
\mapsto
c_k^{-1}
\begin{tikzpicture}[scale =.5, smallnodes, anchorbase, yscale=-1]
	\draw[very thick, gray] (1,2.25) node[below=-1pt]{$k$} to [out=270,in=0] (.5,1.5) to [out=180,in=270] (0,2.25) node[below=-1pt]{$k$};
\end{tikzpicture}
\end{equation}
specify a pivotal functor $\Phi \colon \FWeb(n) \to \FRep(U_q(\spn))$ for any nonzero scalars $a_k, b_k, c_k \in \C(q)$. 
To be more precise, $\Phi$ depends on a choice of $a_1 = b_1$, a choice of $a_k$ and $b_k$ for $2 \le k \le n-1$, and a choice of $c_k$ for $1 \le k \le n$.  
Any such functor $\Phi$ is essentially surjective.
It remains to prove that one can choose scalars $a_k, b_k, c_k$ so that $\Phi$ descends to $\Web(\spn)$, 
which by definition is the quotient of $\FWeb(n)$ by the tensor-ideal generated by the relations \eqref{eq:spn}.

In any pivotal category, the value of a circle is equal to the (quantum) dimension of the corresponding object.
Thus, 
\begin{equation}\label{eq:graycircles}
\begin{tikzpicture}[scale =.75, tinynodes,anchorbase,gray]
	\draw[very thick] (0,0) node[left,xshift=-8pt]{$k$} circle (.5);
\end{tikzpicture}
= \qdim(V_k) =  (-1)^k\frac{[n{-}k{+}1]}{[n{+}1]}{2n{+}2 \brack k}
\end{equation}
by equation \eqref{dimVk}. 
This is exactly the image under $\Phi$ of the $k$-labeled black circle (being multiplied by $c_k c_k^{-1} = 1$).
Thus (\text{\ref{eq:spnOther}a}), and hence (\text{\ref{eq:spn}a}), is satisfied for any choice of scalars. 
In what follows, we will use $\qdim_k$ as a shorthand for $\qdim(V_k)$.

Next,
\[
\begin{tikzpicture}[scale=.2,anchorbase, gray, smallnodes]
	\draw [very thick] (0,-2.75) to [out=30,in=0] (0,-.25);
	\draw [very thick] (0,-2.75) to [out=150,in=180] (0,-.25);
	\draw [very thick] (0,-4.5) node[below=-2pt]{$2$} to (0,-2.75);
\end{tikzpicture}
= 0
\]
since $\Hom_{U_q(\spn)}(V_2,\C(q))$ is zero. Hence (\text{\ref{eq:spn}b}) is satisfied for any choice of scalars.

We now turn our attention to (\text{\ref{eq:spn}d}). We claim that there exists a nonzero scalar $\alpha_k \in \C(q)$ such that
\begin{equation}
\begin{tikzpicture}[scale=.2, xscale=-1,tinynodes, anchorbase,gray]
	\draw [very thick] (-1,-1) node[below,yshift=2pt]{$1$} to [out=90,in=210] (0,.75);
	\draw [very thick] (1,-1) node[below,yshift=2pt]{$k$} to [out=90,in=330] (0,.75);
	\draw [very thick] (3,-1) node[below,yshift=2pt]{$1$} to [out=90,in=330] (1,2.5);
	\draw [very thick] (0,.75) to [out=90,in=210] (1,2.5);
	\draw [very thick] (1,2.5) to (1,4.25) node[above,yshift=-3pt]{$k{+}2$};
\end{tikzpicture}
=
\alpha_k
\begin{tikzpicture}[scale=.2, tinynodes, anchorbase,gray]
	\draw [very thick] (-1,-1) node[below,yshift=2pt]{$1$} to [out=90,in=210] (0,.75);
	\draw [very thick] (1,-1) node[below,yshift=2pt]{$k$} to [out=90,in=330] (0,.75);
	\draw [very thick] (3,-1) node[below,yshift=2pt]{$1$} to [out=90,in=330] (1,2.5);
	\draw [very thick] (0,.75) to [out=90,in=210] (1,2.5);
	\draw [very thick] (1,2.5) to (1,4.25) node[above,yshift=-3pt]{$k{+}2$};
\end{tikzpicture}
\end{equation}
This is because $\Hom_{U_q(\spn)}(V_1 \ot V_k \ot V_1, V_{k+2})$ is $1$-dimensional, and both diagrams represent surjective morphisms so they are nonzero. Now, we have
\[
\begin{tikzpicture}[scale=.2, xscale=-1,tinynodes, anchorbase]
	\draw [very thick] (-1,-1) node[below,yshift=2pt]{$1$} to [out=90,in=210] (0,.75);
	\draw [very thick] (1,-1) node[below,yshift=2pt]{$k$} to [out=90,in=330] (0,.75);
	\draw [very thick] (3,-1) node[below,yshift=2pt]{$1$} to [out=90,in=330] (1,2.5);
	\draw [very thick] (0,.75) to [out=90,in=210] (1,2.5);
	\draw [very thick] (1,2.5) to (1,4.25) node[above,yshift=-3pt]{$k{+}2$};
\end{tikzpicture}
\mapsto
a_{k+1} b_k
\begin{tikzpicture}[scale=.2, xscale=-1,tinynodes, anchorbase,gray]
	\draw [very thick] (-1,-1) node[below,yshift=2pt]{$1$} to [out=90,in=210] (0,.75);
	\draw [very thick] (1,-1) node[below,yshift=2pt]{$k$} to [out=90,in=330] (0,.75);
	\draw [very thick] (3,-1) node[below,yshift=2pt]{$1$} to [out=90,in=330] (1,2.5);
	\draw [very thick] (0,.75) to [out=90,in=210] (1,2.5);
	\draw [very thick] (1,2.5) to (1,4.25) node[above,yshift=-3pt]{$k{+}2$};
\end{tikzpicture}
=
a_{k+1} b_k
\alpha_k
\begin{tikzpicture}[scale=.2, tinynodes, anchorbase,gray]
	\draw [very thick] (-1,-1) node[below,yshift=2pt]{$1$} to [out=90,in=210] (0,.75);
	\draw [very thick] (1,-1) node[below,yshift=2pt]{$k$} to [out=90,in=330] (0,.75);
	\draw [very thick] (3,-1) node[below,yshift=2pt]{$1$} to [out=90,in=330] (1,2.5);
	\draw [very thick] (0,.75) to [out=90,in=210] (1,2.5);
	\draw [very thick] (1,2.5) to (1,4.25) node[above,yshift=-3pt]{$k{+}2$};
\end{tikzpicture}
\qquad \text{and} \qquad
\begin{tikzpicture}[scale=.2, tinynodes, anchorbase]
	\draw [very thick] (-1,-1) node[below,yshift=2pt]{$1$} to [out=90,in=210] (0,.75);
	\draw [very thick] (1,-1) node[below,yshift=2pt]{$k$} to [out=90,in=330] (0,.75);
	\draw [very thick] (3,-1) node[below,yshift=2pt]{$1$} to [out=90,in=330] (1,2.5);
	\draw [very thick] (0,.75) to [out=90,in=210] (1,2.5);
	\draw [very thick] (1,2.5) to (1,4.25) node[above,yshift=-3pt]{$k{+}2$};
\end{tikzpicture}
\mapsto
a_k b_{k+1}
\begin{tikzpicture}[scale=.2, tinynodes, anchorbase,gray]
	\draw [very thick] (-1,-1) node[below,yshift=2pt]{$1$} to [out=90,in=210] (0,.75);
	\draw [very thick] (1,-1) node[below,yshift=2pt]{$k$} to [out=90,in=330] (0,.75);
	\draw [very thick] (3,-1) node[below,yshift=2pt]{$1$} to [out=90,in=330] (1,2.5);
	\draw [very thick] (0,.75) to [out=90,in=210] (1,2.5);
	\draw [very thick] (1,2.5) to (1,4.25) node[above,yshift=-3pt]{$k{+}2$};
\end{tikzpicture}.
\]It follows that if we set 
\begin{equation} \label{eq:bk} b_{k+1} = \frac{a_{k+1} b_k \alpha_k}{a_k} \quad \text{ for } 1\leq k \leq n-2, \end{equation}
then (\text{\ref{eq:spn}d}) holds in the image of $\Phi$.

Next we argue that
\begin{equation}
\begin{tikzpicture}[scale=.175,tinynodes, anchorbase,gray]
	\draw [very thick] (0,.75) to (0,2.5) node[above,yshift=-3pt]{$k$};
	\draw [very thick] (0,-2.75) to [out=30,in=330] node[right,xshift=-2pt]{$k{-}1$} (0,.75);
	\draw [very thick] (0,-2.75) to [out=150,in=210] node[left,xshift=2pt]{$1$} (0,.75);
	\draw [very thick] (0,-4.5) node[below,yshift=2pt]{$k$} to (0,-2.75);
\end{tikzpicture}
= \delta_k
\begin{tikzpicture}[scale=.175, tinynodes, anchorbase,gray]
	\draw [very thick] (0,-4.5) node[below,yshift=2pt]{$k$} to (0,2.5);
\end{tikzpicture}
\end{equation}
for some non-zero $\delta_k \in \C(q)$. Since $\End_{U_q(\spn)}(V_k)$ is $1$-dimensional and spanned by the identity map, we need only argue that the left-hand side is nonzero. 
This follows from the complete reducibility of $V_1\otimes V_{k-1}$, which has $V_k$ as a summand with multiplicity $1$. 
Thus any nonzero map $V_k \to V_1 \ot V_{k-1}$ composed with any nonzero map $V_1 \ot V_{k-1} \to V_k$ will give a nonzero endomorphism of $V_k$.

Note that
\[
\begin{tikzpicture}[scale =.5, smallnodes,anchorbase,rotate=180]
	\draw[very thick, gray] (0,0) node[above=-2pt,xshift=4pt]{$k{-}1$} to [out=90,in=210] (.5,.75);
	\draw[very thick, gray] (1,0) node[above=-2pt]{$1$} to [out=90,in=330] (.5,.75);
	\draw[very thick, gray] (.5,.75) to (.5,1.5) node[below]{$k$};
\end{tikzpicture}
=
\begin{tikzpicture}[scale =.35, smallnodes,anchorbase]
	\draw[very thick, gray] (3,2.5) node[above=-2pt,xshift=4pt]{$k{-}1$} to (3,0) to [out=270,in=0] (1.5,-1.5) to [out=180,in=270] (0,0) to [out=90,in=210] (.5,.75);
	\draw[very thick, gray] (2,2.5) node[above=-2pt]{$1$} to (2,0) to [out=270,in=0] (1.5,-.5) to [out=180,in=270] (1,0) to [out=90,in=330] (.5,.75);
	\draw[very thick, gray] (.5,.75) to (.5,1.5) to [out=90,in=0] (-.125,2) to [out=180,in=90] (-.75,1.5) to (-.75,-2) node[below]{$k$};
\end{tikzpicture},
\]
and similarly when colored black. Thus we have
\begin{equation}
\begin{tikzpicture}[scale =.5, smallnodes,anchorbase,rotate=180]
	\draw[very thick, black] (0,0) node[above=-2pt,xshift=4pt]{$k{-}1$} to [out=90,in=210] (.5,.75);
	\draw[very thick, black] (1,0) node[above=-2pt]{$1$} to [out=90,in=330] (.5,.75);
	\draw[very thick, black] (.5,.75) to (.5,1.5) node[below]{$k$};
\end{tikzpicture}
\mapsto 
c_1 c_{k-1} c_k^{-1} b_{k-1}
\begin{tikzpicture}[scale =.5, smallnodes,anchorbase,rotate=180]
	\draw[very thick, gray] (0,0) node[above=-2pt,xshift=4pt]{$k{-}1$} to [out=90,in=210] (.5,.75);
	\draw[very thick, gray] (1,0) node[above=-2pt]{$1$} to [out=90,in=330] (.5,.75);
	\draw[very thick, gray] (.5,.75) to (.5,1.5) node[below]{$k$};
\end{tikzpicture}.
\end{equation}
This implies that
\begin{equation}
\begin{tikzpicture}[scale=.175,tinynodes, anchorbase]
	\draw [very thick] (0,.75) to (0,2.5) node[above,yshift=-3pt]{$k$};
	\draw [very thick] (0,-2.75) to [out=30,in=330] node[right,xshift=-2pt]{$k{-}1$} (0,.75);
	\draw [very thick] (0,-2.75) to [out=150,in=210] node[left,xshift=2pt]{$1$} (0,.75);
	\draw [very thick] (0,-4.5) node[below,yshift=2pt]{$k$} to (0,-2.75);
\end{tikzpicture}
\mapsto
c_k^{-1} b_{k-1} a_{k-1} c_1 c_{k-1} 
\begin{tikzpicture}[scale=.175,tinynodes, anchorbase,gray]
	\draw [very thick] (0,.75) to (0,2.5) node[above,yshift=-3pt]{$k$};
	\draw [very thick] (0,-2.75) to [out=30,in=330] node[right,xshift=-2pt]{$k{-}1$} (0,.75);
	\draw [very thick] (0,-2.75) to [out=150,in=210] node[left,xshift=2pt]{$1$} (0,.75);
	\draw [very thick] (0,-4.5) node[below,yshift=2pt]{$k$} to (0,-2.75);
\end{tikzpicture} = \delta_k c_k^{-1} b_{k-1} a_{k-1} c_1 c_{k-1}
\begin{tikzpicture}[scale=.175, tinynodes, anchorbase,gray]
	\draw [very thick] (0,-4.5) node[below,yshift=2pt]{$k$} to (0,2.5);
\end{tikzpicture}.
\end{equation}
Thus (\text{\ref{eq:spn}c}) is satisfied, provided we set 
\begin{equation} \label{eq:ck} c_k = \frac{a_{k-1}b_{k-1}c_1c_{k-1}\delta_k}{[k]} \quad \text{ for } 2 \leq k \leq n. \end{equation}

It remains to verify that (\text{\ref{eq:spn}e}) is satisfied. 
In fact, the reader should verify that every diagram in (\text{\ref{eq:spn}e}) is rescaled by $c_1 c_k^{-1}$, so long as \eqref{eq:bk} and \eqref{eq:ck} hold. 
Whether (\text{\ref{eq:spn}e}) is satisfied is independent of any further choices of scalars, and the remaining scalars (namely $a_k$ for $1 \le k \le n-1$ and $c_1$) are chosen freely. 
These degrees of freedom lead to automorphisms of $\Web(\spn)$, which we record in Porism \ref{por:auto}.

To avoid writing the constant factor $c_1 c_k^{-1}$ repeatedly in the following computations,
we slightly abuse notation and denote the images in $\FRep(U_q(\spn))$ of webs in $\FWeb(n)$ by the same (\textbf{black}) diagrams
for the duration of the proof. 
Further, all unlabeled edges are assumed to be labeled by $1$.
We will freely use (\text{\ref{eq:spn}a}) -- (\text{\ref{eq:spn}d}), as well as \eqref{eq:graycircles} below, 
since we have already confirmed that they hold in $\FRep(U_q(\spn))$.
Lastly, for the duration we assume that $n>1$ since when $n=1$ this equation holds trivially.

We begin with the $k=1$ case of (\text{\ref{eq:spn}e}).
Equation \eqref{eq:TensorDecomp} implies that the morphisms
\begin{equation}\label{eq:VVbasis}
\begin{tikzpicture}[scale =.5, smallnodes, anchorbase]
	\draw[very thick] (0,0) to (0,2.25);
	\draw[very thick] (1,0) to (1,2.25);
\end{tikzpicture}
\quad , \quad
\begin{tikzpicture}[scale =.5, smallnodes,anchorbase]
	\draw[very thick] (0,0) to [out=90,in=180] (.5,.75);
	\draw[very thick] (1,0) to [out=90,in=0] (.5,.75);
	\draw[very thick] (.5,1.5) to [out=180,in=270] (0,2.25);
	\draw[very thick] (.5,1.5) to [out=0,in=270] (1,2.25);
\end{tikzpicture}
\quad , \quad
\begin{tikzpicture}[scale=.33, tinynodes, anchorbase]
	\draw[very thick] (-1,0) to (0,1);
	\draw[very thick] (1,0) to (0,1);
	\draw[very thick] (0,2.5) to (-1,3.5);
	\draw[very thick] (0,2.5) to (1,3.5);
	\draw[very thick] (0,1) to node[right=-2pt]{$2$} (0,2.5);
\end{tikzpicture}
\end{equation}
are linearly independent, and give a basis for $\End_{U_q(\spn)}(V_1 \ot V_1)$. 
Hence there must be a relation of the form
\[
\begin{tikzpicture}[scale=.35, rotate=90, tinynodes, anchorbase]
	\draw[very thick] (-1,0) to (0,1);
	\draw[very thick] (1,0) to (0,1);
	\draw[very thick] (0,2.5) to (-1,3.5);
	\draw[very thick] (0,2.5) to (1,3.5);
	\draw[very thick] (0,1) to node[below,yshift=2pt]{$2$} (0,2.5);
\end{tikzpicture}
= \alpha_1 \
\begin{tikzpicture}[scale =.5, smallnodes, anchorbase]
	\draw[very thick] (0,0) to (0,2.25);
	\draw[very thick] (1,0) to (1,2.25);
\end{tikzpicture}
+ \alpha_2
\begin{tikzpicture}[scale =.5, smallnodes,anchorbase]
	\draw[very thick] (0,0) to [out=90,in=180] (.5,.75);
	\draw[very thick] (1,0) to [out=90,in=0] (.5,.75);
	\draw[very thick] (.5,1.5) to [out=180,in=270] (0,2.25);
	\draw[very thick] (.5,1.5) to [out=0,in=270] (1,2.25);
\end{tikzpicture} - \gamma
\begin{tikzpicture}[scale=.35, tinynodes, anchorbase]
	\draw[very thick] (-1,0) to (0,1);
	\draw[very thick] (1,0) to (0,1);
	\draw[very thick] (0,2.5) to (-1,3.5);
	\draw[very thick] (0,2.5) to (1,3.5);
	\draw[very thick] (0,1) to node[right=-2pt]{$2$} (0,2.5);
\end{tikzpicture}.
\]
Using the pivotal structure, we can rotate this relation by 90 degrees to obtain another relation. 
It is an easy exercise to compare these and deduce that $\alpha_2 = \alpha_1 \gamma$ and $\gamma^2 = 1$.
So there exists $\alpha \in \C(q)$ and $\gamma \in \{-1,1\}$ so that
\begin{equation}\label{eq:Switching}
\begin{tikzpicture}[scale=.35, rotate=90, tinynodes, anchorbase]
	\draw[very thick] (-1,0) to (0,1);
	\draw[very thick] (1,0) to (0,1);
	\draw[very thick] (0,2.5) to (-1,3.5);
	\draw[very thick] (0,2.5) to (1,3.5);
	\draw[very thick] (0,1) to node[below,yshift=2pt]{$2$} (0,2.5);
\end{tikzpicture}
+\gamma
\begin{tikzpicture}[scale=.35, tinynodes, anchorbase]
	\draw[very thick] (-1,0) to (0,1);
	\draw[very thick] (1,0) to (0,1);
	\draw[very thick] (0,2.5) to (-1,3.5);
	\draw[very thick] (0,2.5) to (1,3.5);
	\draw[very thick] (0,1) to node[right=-2pt]{$2$} (0,2.5);
\end{tikzpicture}
=\alpha\left( \
\begin{tikzpicture}[scale =.5, smallnodes, anchorbase]
	\draw[very thick] (0,0) to (0,2.25);
	\draw[very thick] (1,0) to (1,2.25);
\end{tikzpicture}
+ \gamma
\begin{tikzpicture}[scale =.5, smallnodes,anchorbase]
	\draw[very thick] (0,0) to [out=90,in=180] (.5,.75);
	\draw[very thick] (1,0) to [out=90,in=0] (.5,.75);
	\draw[very thick] (.5,1.5) to [out=180,in=270] (0,2.25);
	\draw[very thick] (.5,1.5) to [out=0,in=270] (1,2.25);
\end{tikzpicture} \
\right).
\end{equation}

Taking the closure of \eqref{eq:Switching} (i.e. applying the quantum trace), we obtain
\[
\gamma [2] \qdim_2 
= \alpha \qdim_1 \left(\qdim_1 + \gamma \right).
\]
The formula for $\qdim_k$ and some quick algebra (using $\frac{[n][2n+2]}{[n+1]}=[2n+1]-1$) gives
\begin{equation}\label{eq:SwitchCoeff}
\alpha = \frac{\gamma[n-1][2n+1]}{[n]([2n+1]-1-\gamma)}.
\end{equation}
If $\gamma = -1$ then we deduce that $\alpha = - \frac{[n-1]}{[n]}$ and (\text{\ref{eq:spn}e}) holds. We thus need only prove that $\gamma = +1$ is impossible.
Our argument here is somewhat roundabout, but it will explicitly describe the braiding on $\FRep(U_q(\spn))$ as a byproduct.

Indeed, recall that $\FRep(U_q(\spn))$ inherits a braiding $\beta$ from $\Rep(U_q(\spn))$. 
As usual, we denote this morphism as a positive crossing.
It follows that
\begin{equation}\label{eq:posbraid}
\beta_{V_{1},V_{1}} =
\begin{tikzpicture}[scale=.7, anchorbase]
	\draw[very thick] (1,0) to [out=90,in=270] (0,1.5);
	\draw[overcross] (0,0) to [out=90,in=270] (1,1.5);
	\draw[very thick] (0,0) to [out=90,in=270] (1,1.5);
\end{tikzpicture}
:=
\kappa \
\begin{tikzpicture}[scale =.5, smallnodes, anchorbase]
	\draw[very thick] (0,0) to (0,2.25);
	\draw[very thick] (1,0) to (1,2.25);
\end{tikzpicture}
+ \lambda
\begin{tikzpicture}[scale =.5, smallnodes,anchorbase]
	\draw[very thick] (0,0) to [out=90,in=180] (.5,.75);
	\draw[very thick] (1,0) to [out=90,in=0] (.5,.75);
	\draw[very thick] (.5,1.5) to [out=180,in=270] (0,2.25);
	\draw[very thick] (.5,1.5) to [out=0,in=270] (1,2.25);
\end{tikzpicture}
+ \mu
\begin{tikzpicture}[scale=.33, tinynodes, anchorbase]
	\draw[very thick] (-1,0) to (0,1);
	\draw[very thick] (1,0) to (0,1);
	\draw[very thick] (0,2.5) to (-1,3.5);
	\draw[very thick] (0,2.5) to (1,3.5);
	\draw[very thick] (0,1) to node[right=-2pt]{$2$} (0,2.5);
\end{tikzpicture}
\end{equation}
for some $\kappa,\lambda,\mu \in \C(q)$. 
Since the inverse of the braiding (in a braided pivotal category with self-duality structure) can be obtained by rotating the braiding by 90 degrees, we deduce that 
\begin{subequations}
\begin{align}\label{eq:negbraid}
\beta^{-1}_{V_{1},V_{1}} =
\begin{tikzpicture} [scale=.7,anchorbase]
	\draw[very thick] (0,0) to [out=90,in=270] (1,1.5);
	\draw[overcross] (1,0) to [out=90,in=270] (0,1.5);
	\draw[very thick] (1,0) to [out=90,in=270] (0,1.5);
\end{tikzpicture}
&=
\kappa
\begin{tikzpicture}[scale =.5, smallnodes,anchorbase]
	\draw[very thick] (0,0) to [out=90,in=180] (.5,.75);
	\draw[very thick] (1,0) to [out=90,in=0] (.5,.75);
	\draw[very thick] (.5,1.5) to [out=180,in=270] (0,2.25);
	\draw[very thick] (.5,1.5) to [out=0,in=270] (1,2.25);
\end{tikzpicture}
+ \lambda \
\begin{tikzpicture}[scale =.5, smallnodes, anchorbase]
	\draw[very thick] (0,0) to (0,2.25);
	\draw[very thick] (1,0) to (1,2.25);
\end{tikzpicture}
+ \mu
\begin{tikzpicture}[scale=.33, tinynodes, anchorbase,rotate=90]
	\draw[very thick] (-1,0) to (0,1);
	\draw[very thick] (1,0) to (0,1);
	\draw[very thick] (0,2.5) to (-1,3.5);
	\draw[very thick] (0,2.5) to (1,3.5);
	\draw[very thick] (0,1) to node[above=-2pt]{$2$} (0,2.5);
\end{tikzpicture}
\\
\label{eq:negbraid2}
&\stackrel{\eqref{eq:Switching}}{=}
(\lambda+\mu \alpha) \
\begin{tikzpicture}[scale =.5, smallnodes, anchorbase]
	\draw[very thick] (0,0) to (0,2.25);
	\draw[very thick] (1,0) to (1,2.25);
\end{tikzpicture}
+ (\kappa+\gamma\alpha\mu)
\begin{tikzpicture}[scale =.5, smallnodes,anchorbase]
	\draw[very thick] (0,0) to [out=90,in=180] (.5,.75);
	\draw[very thick] (1,0) to [out=90,in=0] (.5,.75);
	\draw[very thick] (.5,1.5) to [out=180,in=270] (0,2.25);
	\draw[very thick] (.5,1.5) to [out=0,in=270] (1,2.25);
\end{tikzpicture}
- \mu \gamma
\begin{tikzpicture}[scale=.33, tinynodes, anchorbase]
	\draw[very thick] (-1,0) to (0,1);
	\draw[very thick] (1,0) to (0,1);
	\draw[very thick] (0,2.5) to (-1,3.5);
	\draw[very thick] (0,2.5) to (1,3.5);
	\draw[very thick] (0,1) to node[right=-2pt]{$2$} (0,2.5);
\end{tikzpicture}.
\end{align}
\end{subequations}
Equations \eqref{eq:posbraid} and \eqref{eq:negbraid2}, together with linear independence of the morphisms in \eqref{eq:VVbasis} and
the equality $\beta^{-1}_{V_{1},V_{1}} \circ \beta_{V_{1},V_{1}} = \id_{V\otimes V}$,
imply that
\begin{equation}\label{eq:BBinv1}
\kappa(\lambda + \mu\alpha) = 1, 
\end{equation}
\begin{equation}\label{eq:BBinv2}
\kappa(\kappa+\mu\alpha\gamma)+\lambda(\lambda+\mu\alpha) = -\qdim_1 \lambda(\kappa+\mu\alpha\gamma), 
\end{equation}
and
\begin{equation}\label{eq:BBinv3}
\mu(\lambda + \mu\alpha) = \mu\gamma(\kappa + [2]\mu).
\end{equation}
By convention, the ribbon element $\nu$ in $\Rep(U_q(\spn))$ acts as the \emph{negative} curl (see \cite[Comment 4.4]{ST}).
Equation \eqref{eq:negbraid} implies that
\begin{equation}
\nu|_{V_{1}} \
\begin{tikzpicture}[scale=.5,anchorbase]
	\draw[very thick] (0,-1) to (0,1);
\end{tikzpicture}
=
\begin{tikzpicture}[scale=.5,yscale=-1,anchorbase]
	\draw[very thick] (.8,-.4) to [out=180,in=270] (0,1);
	\draw[overcross] (0,-1) to [out=90,in=180] (.8,.4);
	\draw[very thick] (0,-1) to [out=90,in=180] (.8,.4);
	\draw[very thick] (.8,.4) to [out=0,in=90] (1.1,0) to [out=270,in=0] (.8,-.4);
\end{tikzpicture}
= \big( \kappa + \lambda \qdim_1 \big) \
\begin{tikzpicture}[scale=.5,anchorbase]
	\draw[very thick] (0,-1) to (0,1);
\end{tikzpicture}
\end{equation}
thus the equality $\nu|_{V_{1}} = -q^{-(\varpi_1,\varpi_1+2\rho)} = -q^{-(2n+1)}$ implies that
\begin{equation}\label{eq:Lambda}
\lambda = \frac{\kappa + q^{-(2n+1)}}{-\qdim_1}.
\end{equation}

Similarly, since $\nu^{-1}$ therefore acts as the positive curl,
the naturality of the braiding gives that
\begin{equation}\label{eq:twocurl}
\nu|_{V_{2}}^{-1} \
\begin{tikzpicture}[scale =.3, smallnodes,anchorbase]
	\draw[very thick] (.5,-3) node[below]{$2$} to (.5,3);
\end{tikzpicture}
=
\begin{tikzpicture}[scale=.875, smallnodes, anchorbase]
	\draw[very thick] (.8,-.4) to [out=180,in=270] (0,1);
	\draw[overcross] (0,-1) to [out=90,in=180] (.8,.4);
	\draw[very thick] (0,-1) node[below]{$2$} to [out=90,in=180] (.8,.4);
	\draw[very thick] (.8,.4) to [out=0,in=90] (1.1,0);
	\draw[very thick] (1.1,0) to [out=270,in=0] (.8,-.4);
\end{tikzpicture}
= \frac{1}{[2]}
\begin{tikzpicture}[scale =.3, smallnodes,anchorbase,yscale=-1]
	\draw[very thick] (.5,-3) to (.5,-2.25);
	\draw[very thick] (.5,-2.25) to [out=150,in=180] (2,1.25) to [out=0,in=90] (3.25,0);
	\draw[very thick] (.5,-2.25) to [out=30, in=180] (2,.5) to [out=0,in=90] (2.5,0);
	\begin{scope}
		\clip (-.5,1.25) rectangle (1.75,-2);
		\draw[overcross] (2.5,0) to [out=270,in=0] (2,-.5) to [out=180,in=330] (.5,2.25);
		\draw[overcross] (3.25,0) to [out=270,in=0] (2,-1.25) to [out=180,in=210] (.5,2.25);
	\end{scope}
	\draw[very thick] (2.5,0) to [out=270,in=0] (2,-.5) to [out=180,in=330] (.5,2.25);
	\draw[very thick] (3.25,0) to [out=270,in=0] (2,-1.25) to [out=180,in=210] (.5,2.25);
	\draw[very thick] (.5,2.25) to (.5,3) node[below]{$2$};
\end{tikzpicture}
= - \frac{q^{2n+1}}{[2]}
\begin{tikzpicture}[scale =.3, smallnodes,anchorbase,yscale=-1]
	\draw[very thick] (.5,-3) to (.5,-2.25);
	\draw[very thick] (.5,-2.25) to [out=150,in=180] (2,1.25) to [out=0,in=90] (3.25,0);
	\begin{scope}
		\clip (-.5,2) rectangle (1.75,-2);
		\draw[overcross] (1.5,0) to [out=90,in=330] (.5,2.25);
	\end{scope}
	\draw[very thick] (.5,-2.25) to [out=30,in=270] (1.5,0) to [out=90,in=330] (.5,2.25);
	\begin{scope}
		\clip (-.5,1.25) rectangle (1.75,-2);
		\draw[overcross] (3.25,0) to [out=270,in=0] (2,-1.25) to [out=180,in=210] (.5,2.25);
	\end{scope}
	\draw[very thick] (3.25,0) to [out=270,in=0] (2,-1.25) to [out=180,in=210] (.5,2.25);
	\draw[very thick] (.5,2.25) to (.5,3) node[below]{$2$};
\end{tikzpicture}
= \frac{q^{4n+2}}{[2]}
\begin{tikzpicture}[scale =.3, smallnodes,anchorbase,yscale=-1]
	\draw[very thick] (.5,-3) to (.5,-2.25);
	\draw[very thick] (.5,-2.25) to [out=150,in=270] (0,-1.5) to [out=90,in=270] (1,0) to [out=90,in=270] (0,1.5) to [out=90,in=210] (.5,2.25);
	\draw[very thick] (.5,-2.25) to [out=30,in=270] (1,-1.5) to [out=90,in=270] (0,0) to [out=90,in=270] (1,1.5) to [out=90,in=330] (.5,2.25);
	\begin{scope}
		\clip (.05,-1.2) rectangle (.95,-.3);
		\draw[overcross] (.5,-2.25) to [out=30,in=270] (1,-1.5) to [out=90,in=270] (0,0) to [out=90,in=270] (1,1.5) to [out=90,in=330] (.5,2.25);
	\end{scope}
	\begin{scope}
		\clip (0,-1.25) rectangle (1,-.25);
		\draw[very thick] (.5,-2.25) to [out=30,in=270] (1,-1.5) to [out=90,in=270] (0,0) to [out=90,in=270] (1,1.5) to [out=90,in=330] (.5,2.25);
	\end{scope}
	\begin{scope}
		\clip (.05,1.2) rectangle (.95,.3);
		\draw[overcross] (.5,-2.25) to [out=150,in=270] (0,-1.5) to [out=90,in=270] (1,0) to [out=90,in=270] (0,1.5) to [out=90,in=210] (.5,2.25);
	\end{scope}
	\begin{scope}
		\clip (0,1.25) rectangle (1,.25);
		\draw[very thick] (.5,-2.25) to [out=150,in=270] (0,-1.5) to [out=90,in=270] (1,0) to [out=90,in=270] (0,1.5) to [out=90,in=210] (.5,2.25);
	\end{scope}
	\draw[very thick] (.5,2.25) to (.5,3) node[below]{$2$};
\end{tikzpicture}
= q^{4n+2}(\kappa+[2]\mu)^2 \ 
\begin{tikzpicture}[scale =.3, smallnodes,anchorbase]
	\draw[very thick] (.5,-3) node[below]{$2$} to (.5,3);
\end{tikzpicture}.
\end{equation}
In the last step, we use that
\begin{equation}
\begin{tikzpicture}[scale =.35, smallnodes,anchorbase]
	\draw[very thick] (.5,.75) to (.5,1.5) node[above=-2pt]{$2$};
	\draw[very thick] (1,-1.5) node[below=-1pt]{$1$} to [out=90,in=270] (0,0) to [out=90,in=210] (.5,.75);
	\begin{scope}
		\clip (0,-1.25) rectangle (1,-.25);
		\draw[overcross] (0,-1.5) to [out=90,in=270] (1,0) to [out=90,in=330] (.5,.75);
	\end{scope}
	\draw[very thick] (0,-1.5) node[below=-1pt]{$1$} to [out=90,in=270] (1,0) to [out=90,in=330] (.5,.75);
\end{tikzpicture}
\stackrel{\eqref{eq:posbraid}}{=}
(\kappa+[2]\mu)
\begin{tikzpicture}[scale =.5, smallnodes,anchorbase]
	\draw[very thick] (0,0) node[below=-1pt]{$1$} to [out=90,in=210] (.5,.75);
	\draw[very thick] (1,0) node[below=-1pt]{$1$} to [out=90,in=330] (.5,.75);
	\draw[very thick] (.5,.75) to (.5,1.5) node[above=-2pt]{$2$};
\end{tikzpicture} .
\end{equation}
Since $\nu|_{V_2} = q^{-4n}$, \eqref{eq:twocurl} gives $(\kappa + [2]\mu)^2 = q^{-2}$ or
\begin{equation}\label{eq:Kappa2}
\kappa+[2]\mu = \pm q^{-1}.
\end{equation}

We now observe that $\mu \neq 0$.
Suppose otherwise, then this implies that $\kappa = \pm q^{-1}$, so \eqref{eq:BBinv1} gives that $\lambda = \pm q$.
Equation \eqref{eq:BBinv2} then gives that $q^{-2} + q^2 = -\qdim_1$, a contradiction since $n>1$.
Note further that this implies the stronger condition that $\mu|_{q=1} \neq 0$, a fact we use below.
It follows that we can multiply \eqref{eq:BBinv3} by $\kappa/\mu$ and apply \eqref{eq:BBinv1} to obtain $\kappa\gamma(\kappa + [2]\mu)=1$.
This combines with \eqref{eq:Kappa2} to give that $\kappa\gamma = \pm q$.
Since $\gamma\in\{-1,1\}$ we finally conclude that $\kappa = \pm q$.

We further claim that $\kappa \neq -q$. 
Indeed, if $\kappa = -q$, then \eqref{eq:Lambda} implies that $\lambda|_{q=1} = 0$.
Equation \eqref{eq:BBinv1} then gives that $\mu\alpha|_{q=1} = -1$.
Since $\gamma = \gamma^{-1}$, equation \eqref{eq:BBinv3} gives that $2\mu|_{q=1} = 1-\gamma$. 
This implies that $\gamma=-1$, since otherwise we would have $\mu|_{q= 1} = 0$. 
Thus $\gamma$ must be $-1$.
Equation \eqref{eq:BBinv3} then gives that $\mu|_{q=1} = 1$, 
which implies that $\alpha|_{q=1} = -1$. Equation \eqref{eq:SwitchCoeff} at $q=1$ becomes
\[
-1= \frac{-(n-1)(2n+1)}{n(2n+1)} =\frac{1-n}{n}
\]
which is a contradiction.

Thus, $\kappa=q$. Equation \eqref{eq:Lambda} then gives
\[
\lambda=\frac{(q+q^{-(2n+1)})[n+1]}{[n][2n+2]}=q^{-n}\frac{(q^{n+1}+q^{-(n+1)})[n+1]}{[n][2n+2]}=\frac{q^{-n}}{[n]}
\]
and \eqref{eq:BBinv1} then implies
\[
\mu\alpha = q^{-1}-\frac{q^{-n}}{[n]} = \frac{[n-1]}{[n]}.
\]
Again, we can set $q=1$, and \eqref{eq:BBinv3} then implies that $2\mu|_{q=1}=\gamma-1$. 
Since $\mu|_{q=1} \neq 0$, this gives $\gamma=-1$.
In turn, we deduce that $\mu=-1$ and $\alpha=-\frac{[n-1]}{[n]}$.

In summary, we have shown the following:
\begin{equation}
\begin{tikzpicture}[scale=.7, anchorbase]
	\draw[very thick] (1,0) to [out=90,in=270] (0,1.5);
	\draw[overcross] (0,0) to [out=90,in=270] (1,1.5);
	\draw[very thick] (0,0) to [out=90,in=270] (1,1.5);
\end{tikzpicture}
=
q \
\begin{tikzpicture}[scale =.5, smallnodes, anchorbase]
	\draw[very thick] (0,0) to (0,2.25);
	\draw[very thick] (1,0) to (1,2.25);
\end{tikzpicture}
+ \frac{q^{-n}}{[n]}
\begin{tikzpicture}[scale =.5, smallnodes,anchorbase]
	\draw[very thick] (0,0) to [out=90,in=180] (.5,.75);
	\draw[very thick] (1,0) to [out=90,in=0] (.5,.75);
	\draw[very thick] (.5,1.5) to [out=180,in=270] (0,2.25);
	\draw[very thick] (.5,1.5) to [out=0,in=270] (1,2.25);
\end{tikzpicture}
-
\begin{tikzpicture}[scale=.33, tinynodes, anchorbase]
	\draw[very thick] (-1,0) to (0,1);
	\draw[very thick] (1,0) to (0,1);
	\draw[very thick] (0,2.5) to (-1,3.5);
	\draw[very thick] (0,2.5) to (1,3.5);
	\draw[very thick] (0,1) to node[right=-2pt]{$2$} (0,2.5);
\end{tikzpicture}
\end{equation}
and
\begin{equation}\label{eq:switchingcoeff}
\begin{tikzpicture}[scale=.33, rotate=90, tinynodes, anchorbase]
	\draw[very thick] (-1,0) to (0,1);
	\draw[very thick] (1,0) to (0,1);
	\draw[very thick] (0,2.5) to (-1,3.5);
	\draw[very thick] (0,2.5) to (1,3.5);
	\draw[very thick] (0,1) to node[below,yshift=2pt]{$2$} (0,2.5);
\end{tikzpicture}
-
\begin{tikzpicture}[scale=.33, tinynodes, anchorbase]
	\draw[very thick] (-1,0) to (0,1);
	\draw[very thick] (1,0) to (0,1);
	\draw[very thick] (0,2.5) to (-1,3.5);
	\draw[very thick] (0,2.5) to (1,3.5);
	\draw[very thick] (0,1) to node[right=-2pt]{$2$} (0,2.5);
\end{tikzpicture}
=\frac{[n-1]}{[n]}\left( \
\begin{tikzpicture}[scale =.5, smallnodes,anchorbase]
	\draw[very thick] (0,0) to [out=90,in=180] (.5,.75);
	\draw[very thick] (1,0) to [out=90,in=0] (.5,.75);
	\draw[very thick] (.5,1.5) to [out=180,in=270] (0,2.25);
	\draw[very thick] (.5,1.5) to [out=0,in=270] (1,2.25);
\end{tikzpicture}
- \
\begin{tikzpicture}[scale =.5, smallnodes, anchorbase]
	\draw[very thick] (0,0) to (0,2.25);
	\draw[very thick] (1,0) to (1,2.25);
\end{tikzpicture}
\
\right)
\end{equation}
The latter is the $k=1$ version of (\text{\ref{eq:spn}e}), and it remains to deduce the general case.

For this, we extend \eqref{eq:VVbasis} by again observing that equation \eqref{eq:TensorDecomp} implies that:
\[
\begin{tikzpicture}[scale =.5, smallnodes, anchorbase]
	\draw[very thick] (0,0) node[below=-1pt]{$k$} to (0,2.25) node[above=-3pt]{$k$};
	\draw[very thick] (1,0) to (1,2.25);
\end{tikzpicture}
\quad , \quad
\begin{tikzpicture}[scale=.33, smallnodes, anchorbase]
	\draw[very thick] (-1,0) node[below=-2pt,xshift=-2pt]{$k$} to (0,1);
	\draw[very thick] (1,0) to (0,1);
	\draw[very thick] (0,2.5) to (-1,3.5) node[above=-4pt,xshift=-2pt]{$k$} ;
	\draw[very thick] (0,2.5) to (1,3.5);
	\draw[very thick] (0,1) to node[right=-2pt]{$k{+}1$} (0,2.5);
\end{tikzpicture}
\quad , \quad
\begin{tikzpicture}[scale=.33, smallnodes, anchorbase]
	\draw[very thick] (-1,0) node[below=-2pt,xshift=-2pt]{$k$} to (0,1);
	\draw[very thick] (1,0) to (0,1);
	\draw[very thick] (0,2.5) to (-1,3.5) node[above=-4pt,xshift=-2pt]{$k$} ;
	\draw[very thick] (0,2.5) to (1,3.5);
	\draw[very thick] (0,1) to node[right=-2pt]{$k{-}1$} (0,2.5);
\end{tikzpicture}
\]
give a basis for $\End_{U_q(\spn)}(V_k \otimes V_1)$, thus there exist $a,b,c \in \C(q)$ so that
\begin{equation}\label{eq:reduction}
\begin{tikzpicture}[scale=.4, tinynodes, anchorbase, rotate=-90]
	\draw[very thick] (-1,0) node[above=-3pt,xshift=-2pt]{$k$} to (0,1.5);
	\draw[very thick] (1,0) node[below=-3pt,xshift=-2pt]{$k$} to (0,1.5);
	\draw[very thick] (-.7,.5) to node[left=-2pt]{${k{-}1}$} (.7,.5);
	\draw[very thick] (0,2.5) to (-1,3.5) node[above=-4pt,xshift=2pt]{$1$};
	\draw[very thick] (0,2.5) to (1,3.5) node[below=-3pt,xshift=2pt]{$1$};
	\draw[very thick] (0,1.5) to node[above=-2pt]{$2$} (0,2.5);
\end{tikzpicture}
=
a \
\begin{tikzpicture}[scale =.5, smallnodes, anchorbase]
	\draw[very thick] (0,0) node[below=-1pt]{$k$} to (0,2.25) node[above=-3pt]{$k$};
	\draw[very thick] (1,0) to (1,2.25);
\end{tikzpicture}
+b
\begin{tikzpicture}[scale=.33, smallnodes, anchorbase]
	\draw[very thick] (-1,0) node[below=-2pt,xshift=-2pt]{$k$} to (0,1);
	\draw[very thick] (1,0) to (0,1);
	\draw[very thick] (0,2.5) to (-1,3.5) node[above=-4pt,xshift=-2pt]{$k$} ;
	\draw[very thick] (0,2.5) to (1,3.5);
	\draw[very thick] (0,1) to node[right=-2pt]{$k{+}1$} (0,2.5);
\end{tikzpicture}
+c
\begin{tikzpicture}[scale=.33, smallnodes, anchorbase]
	\draw[very thick] (-1,0) node[below=-2pt,xshift=-2pt]{$k$} to (0,1);
	\draw[very thick] (1,0) to (0,1);
	\draw[very thick] (0,2.5) to (-1,3.5) node[above=-4pt,xshift=-2pt]{$k$} ;
	\draw[very thick] (0,2.5) to (1,3.5);
	\draw[very thick] (0,1) to node[right=-2pt]{$k{-}1$} (0,2.5);
\end{tikzpicture}.
\end{equation}
It remains to deduce the values of $a,b,c \in \C(q)$.
To do so, we begin by deducing some auxiliary relations.

First, we have
\begin{equation}\label{eq:badbigon}
\begin{tikzpicture}[scale=.175,tinynodes, anchorbase]
	\draw [very thick] (0,.75) to (0,2.5) node[above,yshift=-3pt]{$k$};
	\draw [very thick] (0,-2.75) to [out=30,in=330] node[right,xshift=-2pt]{$k{+}1$} (0,.75);
	\draw [very thick] (0,-2.75) to [out=150,in=210] node[left,xshift=2pt]{$1$} (0,.75);
	\draw [very thick] (0,-4.5) node[below,yshift=2pt]{$k$} to (0,-2.75);
\end{tikzpicture}
=
- \frac{[n-k][2n+2-k]}{[n-k+1]}
\begin{tikzpicture}[scale=.175, tinynodes, anchorbase]
	\draw [very thick] (0,-4.5) node[below,yshift=2pt]{$k$} to (0,2.5);
\end{tikzpicture}.
\end{equation}
Indeed, since $\End_{U_q(\spn)}(V_k)$ is $1$-dimensional, the left-hand side must be a scalar multiple of the right-hand side,
and the scalar can be computed by taking the closure (trace) of both sides of this relation, evaluating the relevant bigon using (\text{\ref{eq:spn}c}), and the resulting circles using (\text{\ref{eq:spnOther}a}).

Next, we establish the equality:
\begin{equation}\label{eq:triangle}
\begin{tikzpicture}[scale=.35,smallnodes,anchorbase]
	\draw[very thick,directed=.4,rdirected=.75] (-1,0) to node[below=1pt]{$k{+}1$} (1,0);
	\draw[very thick] (-1,0) to (0,1.732);
	\draw[very thick] (1,0) to (0,1.732);
	\draw[very thick,->] (0,1.732) to (0,3.232) node[above=-2pt]{$2$};
	\draw[very thick] (-2.3,-.75) node[below=-2pt, xshift=-2pt]{$k$} to (-1,0);
	\draw[very thick] (2.3,-.75) node[below=-2pt,xshift=2pt]{$k$} to (1,0);
\end{tikzpicture}
= \frac{[n+2-k]}{[n+1-k]}
\begin{tikzpicture}[scale=.35,smallnodes,anchorbase]
	\draw[very thick] (-1,0) to node[below]{$k{-}1$} (1,0);
	\draw[very thick] (-1,0) to (0,1.732);
	\draw[very thick] (1,0) to (0,1.732);
	\draw[very thick,->] (0,1.732) to (0,3.232) node[above=-2pt]{$2$};
	\draw[very thick,<-] (-2.3,-.75) node[below=-2pt, xshift=-2pt]{$k$} to (-1,0);
	\draw[very thick,<-] (2.3,-.75) node[below=-2pt,xshift=2pt]{$k$} to (1,0);
	\end{tikzpicture}
\end{equation}
via induction on $k$. When $k=1$, this follows immediately by \eqref{eq:switchingcoeff}, and when $k=n$ this holds since $\Hom_{U_q(\spn)}(V_n \ot V_n, V_2)=0$.
For the remaining cases, we first observe that the morphism appearing on the right-hand side of \eqref{eq:triangle} is non-zero when $1 \leq k \leq n-1$.
This follows since we compute the following wherein it appears:
\begin{equation}\label{eq:nonzero}
\begin{tikzpicture}[scale=.25,tinynodes,anchorbase]
	\draw[very thick] (-2,-2) node[below=-2pt]{$k$} to [out=90,in=210] (-1,0);
	\draw[very thick] (1,0) to node[above=-2pt]{$2$} (2,-.5) to (2,-2);
	\draw[very thick] (-1,0) to (1,0);
	\draw[very thick] (-1,0) to node[left=-1pt]{$k{-}1$} (0,1.732);
	\draw[very thick] (1,0) to (0,1.732);
	\draw[very thick] (0,1.732) to [out=90,in=210] node[left=-1pt]{$k$} (1.25,3.5);
	\draw[very thick] (2,-.5) to [out=30,in=330] (1.25,3.5);
	\draw[very thick] (1.25,3.5) to (1.25,5) node[above=-2pt]{$k{+}1$};
\end{tikzpicture}
\stackrel{\eqref{eq:generalassoc}}{=}
\begin{tikzpicture}[scale=.25,tinynodes,anchorbase,xscale=-1]
	\draw[very thick] (-2,-2) to [out=90,in=210] (-1,0);
	\draw[very thick] (1,0) to (2,-.5) to (2,-2) node[below=-2pt]{$k$};
	\draw[very thick] (-1,0) to node[below=-2pt]{$2$}  (1,0);
	\draw[very thick] (-1,0) to (0,1.732);
	\draw[very thick] (1,0) to (0,1.732);
	\draw[very thick] (0,1.732) to [out=90,in=210] node[right=-1pt]{$2$} (1.25,3.5);
	\draw[very thick] (2,-.5) to [out=30,in=330] node[left=-1pt]{$k{-}1$} (1.25,3.5);
	\draw[very thick] (1.25,3.5) to (1.25,5) node[above=-2pt]{$k{+}1$};
\end{tikzpicture}
=
\frac{[n+1]}{[n]}
\begin{tikzpicture}[scale=.25,tinynodes,anchorbase,xscale=-1]
	\draw[very thick] (-1.5,-2) to [out=90,in=210] (0,.573);
	\draw[very thick] (0,.573) to (2,-.5) to (2,-2) node[below=-2pt]{$k$};
	\draw[very thick] (0,.573) to [out=90,in=210] node[right=-3pt]{$2$} (1.25,3.5);
	\draw[very thick] (2,-.5) to [out=30,in=330] node[left=-1pt]{$k{-}1$} (1.25,3.5);
	\draw[very thick] (1.25,3.5) to (1.25,5) node[above=-2pt]{$k{+}1$};
\end{tikzpicture}
\stackrel{\eqref{eq:generalassoc}}{=}
\frac{[n+1]}{[n]}
\begin{tikzpicture}[scale=.25,anchorbase,tinynodes]
	\draw [very thick] (0,-2) node[below=-2pt]{$k$} to (0,-1);
	\draw [very thick] (0,-1) to [out=150,in=210] node[left=-1pt]{$k{-}1$} (0,.75);
	\draw [very thick] (0,-1) to [out=30,in=330] (0,.75);
	\draw [very thick] (3,-2) to [out=90,in=330] (1,2.5);
	\draw [very thick] (0,.75) to [out=90,in=210] (1,2.5);
	\draw [very thick] (1,2.5) to (1,4.25) node[above=-2pt]{$k{+}1$};
\end{tikzpicture}
=
\frac{[k][n+1]}{[n]}
\begin{tikzpicture}[scale =.5, smallnodes,anchorbase]
	\draw[very thick] (0,-.25) node[below]{$k$} to [out=90,in=210] (.5,.75);
	\draw[very thick] (1,-.25) to [out=90,in=330] (.5,.75);
	\draw[very thick] (.5,.75) to (.5,1.5) node[above]{$k{+}1$};
\end{tikzpicture}
\neq 0.
\end{equation}
Here, we use \eqref{eq:generalassoc} twice. 
Note that this relation holds in $\FRep(U_q(\spn))$ by Lemma \ref{lem:extrarels}.

Thus, since $\Hom_{U_q(\spn)}(V_k \ot V_k, V_2)$ is $1$-dimensional for $1 \leq k \leq n-1$, 
there exists $\tau_k$ such that
\begin{equation}
\begin{tikzpicture}[scale=.35,smallnodes,anchorbase]
	\draw[very thick,directed=.4,rdirected=.75] (-1,0) to node[below=1pt]{$k{+}1$} (1,0);
	\draw[very thick] (-1,0) to (0,1.732);
	\draw[very thick] (1,0) to (0,1.732);
	\draw[very thick,->] (0,1.732) to (0,3.232) node[above=-2pt]{$2$};
	\draw[very thick] (-2.3,-.75) node[below=-2pt, xshift=-2pt]{$k$} to (-1,0);
	\draw[very thick] (2.3,-.75) node[below=-2pt,xshift=2pt]{$k$} to (1,0);
\end{tikzpicture}
= \tau_k
\begin{tikzpicture}[scale=.35,smallnodes,anchorbase]
	\draw[very thick] (-1,0) to node[below]{$k{-}1$} (1,0);
	\draw[very thick] (-1,0) to (0,1.732);
	\draw[very thick] (1,0) to (0,1.732);
	\draw[very thick,->] (0,1.732) to (0,3.232) node[above=-2pt]{$2$};
	\draw[very thick,<-] (-2.3,-.75) node[below=-2pt, xshift=-2pt]{$k$} to (-1,0);
	\draw[very thick,<-] (2.3,-.75) node[below=-2pt,xshift=2pt]{$k$} to (1,0);
\end{tikzpicture}.
\end{equation}
This then implies that:
\begin{equation}\label{eq:doubletriangle}
\begin{tikzpicture}[scale=.35,smallnodes,anchorbase]
	\draw[very thick] (-1,0) to node[below]{$k{+}1$} (1,0);
	\draw[very thick] (-1,0) to (0,1.732);
	\draw[very thick] (1,0) to (0,1.732);
	\draw[very thick] (-2.3,-.75) node[below=-2pt,xshift=-2pt]{$k$} to (-1,0);
	\draw[very thick] (3,0) to node[below]{$k{-}1$} (5,0);
	\draw[very thick] (3,0) to (4,1.732);
	\draw[very thick] (5,0) to (4,1.732);
	\draw[very thick] (6.3,-.75) node[below=-2pt,xshift=2pt]{$k$} to (5,0);
	\draw[very thick] (0,1.732) to [out=90,in=180] (2,3) node[above=-2pt]{$2$} to [out=0,in=90] (4,1.732);
	\draw[very thick] (1,0) to [out=-30,in=210] node[below]{$k$} (3,0);
\end{tikzpicture}
=\tau_k
\begin{tikzpicture}[scale=.35,smallnodes,anchorbase]
	\draw[very thick] (-1,0) to node[below]{$k{-}1$} (1,0);
	\draw[very thick] (-1,0) to (0,1.732);
	\draw[very thick] (1,0) to (0,1.732);
	\draw[very thick] (-2.3,-.75) node[below=-2pt,xshift=-2pt]{$k$} to (-1,0);
	\draw[very thick] (3,0) to node[below]{$k{-}1$} (5,0);
	\draw[very thick] (3,0) to (4,1.732);
	\draw[very thick] (5,0) to (4,1.732);
	\draw[very thick] (6.3,-.75) node[below=-2pt,xshift=2pt]{$k$} to (5,0);
	\draw[very thick] (0,1.732) to [out=90,in=180] (2,3) node[above]{$2$} to [out=0,in=90] (4,1.732);
	\draw[very thick] (1,0) to [out=-30,in=210] node[below]{$k$} (3,0);
\end{tikzpicture}.
\end{equation}
Simplifying the left-hand side of \eqref{eq:doubletriangle} using \eqref{eq:nonzero} and \eqref{eq:badbigon},
we obtain:
\begin{equation} \label{eq:simplify1}
\begin{tikzpicture}[scale=.35,smallnodes,anchorbase]
	\draw[very thick] (-1,0) to node[below]{$k{+}1$} (1,0);
	\draw[very thick] (-1,0) to (0,1.732);
	\draw[very thick] (1,0) to (0,1.732);
	\draw[very thick] (-2.3,-.75) node[below=-2pt,xshift=-2pt]{$k$} to (-1,0);
	\draw[very thick] (3,0) to node[below]{$k{-}1$} (5,0);
	\draw[very thick] (3,0) to (4,1.732);
	\draw[very thick] (5,0) to (4,1.732);
	\draw[very thick] (6.3,-.75) node[below=-2pt,xshift=2pt]{$k$} to (5,0);
	\draw[very thick] (0,1.732) to [out=90,in=180] (2,3) node[above=-2pt]{$2$} to [out=0,in=90] (4,1.732);
	\draw[very thick] (1,0) to [out=-30,in=210] node[below]{$k$} (3,0);
\end{tikzpicture}
=-\frac{[k][n+1][n-k][2n+2-k]}{[n][n+1-k]}
\begin{tikzpicture}[scale =.75, smallnodes, anchorbase, yscale=-1]
	\draw[very thick] (1,2.25) node[below=-1pt]{$k$} to [out=270,in=0] (.5,1.5) to [out=180,in=270] (0,2.25) node[below=-1pt]{$k$};
\end{tikzpicture}.
\end{equation}

To simplify the right-hand side of \eqref{eq:doubletriangle}, first apply \eqref{eq:switchingcoeff} to the top of the diagram to obtain a sum of three terms. 
Two terms can be easily simplified using (\text{\ref{eq:spn}c}) and \eqref{eq:badbigon}. 
The third term is as follows:
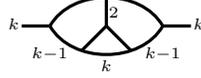
\[
\begin{tikzpicture}[scale=.75,anchorbase,tinynodes]
	\draw[very thick] (-1.5,0) node[left=-3pt]{$k$} to (-1,0);
	\draw[very thick] (-1,0) to [out=60,in=120] (1,0);
	\draw[very thick] (-1,0) node[below=5pt]{$k{-}1$} to [out=300,in=240] node[below=-1pt]{$k$} (1,0) node[below=5pt]{$k{-}1$};
	\draw[very thick] (1,0) to (1.5,0) node[right=-3pt]{$k$};
	\begin{scope}
	\clip (-1,0) to [out=60,in=120] (1,0) to [out=240,in=300] (-1,0);
	\draw[very thick] (0,0) to node[right=-3pt,yshift=-1pt]{$2$} (0,.5);
	\draw[very thick] (0,0) to (1,-1);
	\draw[very thick] (0,0) to (-1,-1);
	\end{scope}
\end{tikzpicture}
\]
To simplify it, 
first apply \eqref{eq:triangle} for $k-1$, and then associativity \eqref{eq:generalassoc}. 
The $k=1$ case of \eqref{eq:triangle} can now be applied and then (\text{\ref{eq:spnOther}c}) will 
resolve\footnote{The equation (\text{\ref{eq:spnOther}c}) holds in $\FRep(U_q(\spn))$ by Lemma \ref{lem:extrarels}.} this diagram. 
Ultimately we obtain
\begin{equation} \label{eq:simplify2}
\begin{tikzpicture}[scale=.35,smallnodes,anchorbase]
	\draw[very thick] (-1,0) to node[below]{$k{-}1$} (1,0);
	\draw[very thick] (-1,0) to (0,1.732);
	\draw[very thick] (1,0) to (0,1.732);
	\draw[very thick] (-2.3,-.75) node[below=-2pt,xshift=-2pt]{$k$} to (-1,0);
	\draw[very thick] (3,0) to node[below]{$k{-}1$} (5,0);
	\draw[very thick] (3,0) to (4,1.732);
	\draw[very thick] (5,0) to (4,1.732);
	\draw[very thick] (6.3,-.75) node[below=-2pt,xshift=2pt]{$k$} to (5,0);
	\draw[very thick] (0,1.732) to [out=90,in=180] (2,3) node[above]{$2$} to [out=0,in=90] (4,1.732);
	\draw[very thick] (1,0) to [out=-30,in=210] node[below]{$k$} (3,0);
\end{tikzpicture}
=-\frac{[k][n+1][n-k][2n+2-k]}{[n][n+2-k]}
\begin{tikzpicture}[scale =.75, smallnodes, anchorbase, yscale=-1]
	\draw[very thick] (1,2.25) node[below=-1pt]{$k$} to [out=270,in=0] (.5,1.5) to [out=180,in=270] (0,2.25) node[below=-1pt]{$k$};
\end{tikzpicture}.
\end{equation}
Dividing \eqref{eq:simplify2} by \eqref{eq:simplify1} we obtain $\tau_k=\frac{[n+2-k]}{[n+1-k]}$ as desired.

The values of $a,b,c \in \C(q)$ from \eqref{eq:reduction} are now easily deduced.
We obtain relations in $1$-dimensional $\Hom$-spaces of $\FRep(U_q(\spn))$ by composing as follows:
\begin{align}
\begin{tikzpicture}[scale =.5, smallnodes,anchorbase]
	\draw[very thick] (0,-1.5) node[below=-1pt]{$k$} to (0,.5) node[above=-2pt]{$k$};
	\draw[very thick] (1,0) to [out=90,in=180] (1.5,.5) to [out=0,in=90] (2,0) to (2,-1) to [out=270,in=0] (1.5,-1.5) to [out=180,in=270] (1,-1);
	\draw[very thick, fill=white] (-.5,0) rectangle (1.5,-1);
	\node at (.5,-.5){\eqref{eq:reduction}};
\end{tikzpicture}
&\Longrightarrow \ 0 = -\frac{[n][2n+2]}{[n+1]}a - \frac{[n-k][2n+2-k]}{[n+1-k]}b + [k]c \\
\begin{tikzpicture}[scale =.5, smallnodes,anchorbase]
	\draw[very thick] (0,-2) node[below=-1pt]{$k$} to (0,.5) node[above=-2pt]{$k$};
	\draw[very thick] (1,0) to [out=90,in=180] (1.5,.5) to [out=0,in=90] (2,0) to (2,-1) to [out=270,in=30] (1.5,-1.5);
	\draw[very thick] (1,-1) to [out=270,in=150] (1.5,-1.5);
	\draw[very thick] (1.5,-1.5) to (1.5,-2) node[below=-1pt]{$2$};
	\draw[very thick, fill=white] (-.5,0) rectangle (1.5,-1);
	\node at (.5,-.5){\eqref{eq:reduction}};
\end{tikzpicture}
&\Longrightarrow \ [2] = \frac{[n+2-k]}{[n+1-k]}b + c \\
\begin{tikzpicture}[scale =.5, smallnodes,anchorbase]
	\draw[very thick] (0,-1.5) node[below=-1pt]{$k$} to (0,0) to [out=90,in=210] (.5,.75);
	\draw[very thick] (1,-1.5) node[below=-1pt]{$1$} to (1,0)  to [out=90,in=330] (.5,.75);
	\draw[very thick] (.5,.75) to (.5,1.5) node[above=-2pt]{$k{+}1$};
	\draw[very thick, fill=white] (-.5,0) rectangle (1.5,-1);
	\node at (.5,-.5){\eqref{eq:reduction}};
\end{tikzpicture}
&\Longrightarrow \ \frac{[k][n+1]}{[n]} = a + [k+1]b.
\end{align}
Elementary linear algebra implies that
\[
a=-\frac{[n-k]}{[n]} \quad , \quad b=1 \quad , \quad c=\frac{[n-k]}{[n+1-k]}
\]
which verifies (\text{\ref{eq:spn}e}).
\end{proof}

We record the following.
(Recall that a \emph{Porism} is a direct consequence of a proof.)

\begin{por}\label{por:auto}
For any invertible scalars $a_k, b_k$ with $1 \le k \le n-1$ and $c_k$ for $1 \le k \le n$ that satisfy \eqref{eq:bk} and \eqref{eq:ck}, 
there is an automorphism of $\Web(\spn)$ which rescales the generating morphisms analogously to \eqref{eq:images} (but with all black strands). 
Moreover, all automorphisms of $\Web(\spn)$ (as a $\C(q)$-linear monoidal category) that fix the generating objects $\{1, \ldots, n\}$ are of this form.
\end{por}

\begin{por}\label{por:braiding}
The elements
\begin{equation}\label{eq:braiding}
\begin{tikzpicture}[scale=.5, anchorbase]
	\draw[very thick] (1,0) to [out=90,in=270] (0,1.5);
	\draw[overcross] (0,0) to [out=90,in=270] (1,1.5);
	\draw[very thick] (0,0) to [out=90,in=270] (1,1.5);
\end{tikzpicture}
=
q \
\begin{tikzpicture}[scale =.4, smallnodes, anchorbase]
	\draw[very thick] (0,0) to (0,2.25);
	\draw[very thick] (1,0) to (1,2.25);
\end{tikzpicture}
+ \frac{q^{-n}}{[n]}
\begin{tikzpicture}[scale =.4, smallnodes,anchorbase]
	\draw[very thick] (0,0) to [out=90,in=180] (.5,.75);
	\draw[very thick] (1,0) to [out=90,in=0] (.5,.75);
	\draw[very thick] (.5,1.5) to [out=180,in=270] (0,2.25);
	\draw[very thick] (.5,1.5) to [out=0,in=270] (1,2.25);
\end{tikzpicture}
-
\begin{tikzpicture}[scale=.25, tinynodes, anchorbase]
	\draw[very thick] (-1,0) to (0,1);
	\draw[very thick] (1,0) to (0,1);
	\draw[very thick] (0,2.5) to (-1,3.5);
	\draw[very thick] (0,2.5) to (1,3.5);
	\draw[very thick] (0,1) to node[right=-2pt]{$2$} (0,2.5);
\end{tikzpicture}
\quad \text{and} \quad
\begin{tikzpicture}[scale=.5, anchorbase,xscale=-1]
	\draw[very thick] (1,0) to [out=90,in=270] (0,1.5);
	\draw[overcross] (0,0) to [out=90,in=270] (1,1.5);
	\draw[very thick] (0,0) to [out=90,in=270] (1,1.5);
\end{tikzpicture}
=
q^{-1} \
\begin{tikzpicture}[scale =.4, smallnodes, anchorbase]
	\draw[very thick] (0,0) to (0,2.25);
	\draw[very thick] (1,0) to (1,2.25);
\end{tikzpicture}
+ \frac{q^{n}}{[n]}
\begin{tikzpicture}[scale =.4, smallnodes,anchorbase]
	\draw[very thick] (0,0) to [out=90,in=180] (.5,.75);
	\draw[very thick] (1,0) to [out=90,in=0] (.5,.75);
	\draw[very thick] (.5,1.5) to [out=180,in=270] (0,2.25);
	\draw[very thick] (.5,1.5) to [out=0,in=270] (1,2.25);
\end{tikzpicture}
-
\begin{tikzpicture}[scale=.25, tinynodes, anchorbase]
	\draw[very thick] (-1,0) to (0,1);
	\draw[very thick] (1,0) to (0,1);
	\draw[very thick] (0,2.5) to (-1,3.5);
	\draw[very thick] (0,2.5) to (1,3.5);
	\draw[very thick] (0,1) to node[right=-2pt]{$2$} (0,2.5);
\end{tikzpicture}
\end{equation}
in $\End_{\Web(\spn)}(1 \ot 1)$ are inverse isomorphisms that are sent by the functor $\Phi$ to the 
braiding $\beta_{V_1,V_1}$ and its inverse.
\end{por}

\subsection{Reduction to the ``all ones'' category}

In Theorem \ref{thm:functor},
we constructed an essentially surjective functor
\[
\Phi \colon \Web(\spn) \to \FRep(U_q(\spn)).
\] 
We now let $\SWeb(\spn)$ denote the full subcategory of $\Web(\spn)$ with objects $1^{\ot k}$ 
and let $\SRep(U_q(\spn))$ denote the full subcategory of $\FRep(U_q(\spn))$ with objects $V_1^{\ot k}$, for $k \geq 0$. 
The restriction of $\Phi$ to $\SWeb(\spn)$ is a functor to $\SRep(U_q(\spn))$ that we conclude is full and faithful in \S \ref{sec:full} and \S \ref{sec:faithful} below.

We now show that these statements will suffice to prove our main result.
Given a linear category $\CS$, we let $\Kar(\CS)$ denote its Karoubi envelope. 
Recall that any linear category embeds fully faithfully into its Karoubi envelope.

\begin{lem}\label{lem:reduction}
Let $\mathcal{D}$ and $\mathcal{R}$ be $\mathbb{C}(q)$-linear monoidal categories. 
Suppose that $\mathcal{D}'$ is a full monoidal subcategory of $\mathcal{D}$ such that
every object in $\mathcal{D}$ is a direct summand of an iterated tensor product of objects in $\mathcal{D}'$, 
after embedding both in $\Kar(\mathcal{D})$.
Let $\Psi \colon \mathcal{D}\to \mathcal{R}$ be a $\mathbb{C}(q)$-linear (monoidal) functor,
then $\Psi$ is full if and only if $\Psi|_{\mathcal{D}'}$ is full and $\Psi$ is faithful if and only if $\Psi|_{\mathcal{D}'}$ is faithful. 
\end{lem}

\begin{proof}
Clearly, if $\Psi$ is full (resp. faithful) then $\Psi|_{\mathcal{D}'}$ is full (resp. faithful), so it suffices to prove the converse statements.
Let $X$ and $Y$ be objects in $\mathcal{D}$. 
By hypothesis, there are objects $X_1, \ldots, X_n$ and $Y_1, \ldots , Y_m$ in $\mathcal{D}'$ along with morphisms:
\[
X\xrightarrow{\iota_X} X_1\ot \cdots \ot X_n \xrightarrow{\pi_X} X
\]
and
\[
Y\xrightarrow{\iota_Y} Y_1\ot \cdots \ot Y_m \xrightarrow{\pi_Y} Y
\]
in $\mathcal{D}$ so that $\pi_X\circ\iota_X= \id_X$ and $\pi_Y \circ \iota_Y = \id_Y$. 

Suppose now that $\Psi|_{\mathcal{D}'}$ is full, and let $T\in \Hom_{\mathcal{R}}(\Psi(X), \Psi(Y))$.
Since $\Psi|_{\mathcal{D}'}$ is full, there exists
\[
D\in \Hom_{\mathcal{D}'}(X_1\ot \cdots \ot X_n, Y_1\ot \cdots \ot Y_m)
\]
so that $\Psi(D) = \Psi(\iota_Y)\circ T\circ \Psi(\pi_X)$. 
It follows that $\pi_Y\circ D\circ \iota_X\in \Hom_{\mathcal{D}}(X, Y)$, thus
\[
\Psi(\pi_Y\circ D\circ \iota_X)=  \Psi(\pi_Y)\circ \Psi(\iota_Y)\circ T \circ \Psi(\pi_X)\circ\Psi(\iota_X)=T
\]
as desired.

Similarly, suppose that $\Psi|_{\mathcal{D}'}$ is faithful, and let $D\in \Hom_{\mathcal{D}}(X, Y)$ be such that $\Psi(D)= 0$. 
This implies that $\Psi(\iota_Y \circ D\circ \pi_X) = \Psi(\iota_Y)\circ 0 \circ \Psi(\pi_X) = 0$.
However, since $\mathcal{D}'$ is a full subcategory of $\mathcal{D}$, we have
\[
\iota_Y \circ D\circ \pi_X\in \Hom_{\mathcal{D}'}(X_1\ot \cdots \ot X_n, Y_1\ot \cdots \ot Y_m).
\] 
Since $\Psi|_{\mathcal{D}'}$ is faithful, this implies that $\iota_Y\circ D \circ \pi_X= 0$, so 
$D= \pi_Y\circ\iota_Y\circ D \circ \pi_X\circ \iota_X = 0$.
\end{proof}

It is easy to see that the hypotheses of this theorem are satisfied for $\mathcal{D}'= \SWeb(\spn)$, $\mathcal{D} = \Web(\spn)$, 
and the functor from Theorem \ref{thm:functor}.
Recursively define the \emph{full split vertex} as
\begin{equation}\label{eq:fullsplit}
\begin{tikzpicture}[scale=.45, tinynodes, anchorbase]
	\draw[very thick] (0,-2) node[below,yshift=2pt]{$k$} to (0,-1);
	\draw[very thick] (0,-1) to (1,1) node[above,yshift=-2pt]{$1$};
	\draw[very thick] (0,-1) to (-1,1) node[above,yshift=-2pt]{$1$};
	\node at (0,.5) {$\cdots$};
\end{tikzpicture}
:=
\begin{tikzpicture}[scale=.45, tinynodes, anchorbase]
	\draw[very thick,<-] (0,-2) node[below,yshift=2pt]{$k$} to (0,-1);
	\draw[very thick] (0,-1) to (1,1) node[above,yshift=-2pt]{$1$};
	\draw[very thick] (-.5,0) to (0,1) node[above,yshift=-2pt]{$1$};
	\draw[very thick] (0,-1) to (-1,1) node[above,yshift=-2pt]{$1$};
	\draw[very thick] (0,-1) to node[pos=.25,left=-2pt]{$k{-}1$} (-1,1) node[above,yshift=-2pt]{$1$};
	\node at (-.5,.75) {$\mydots$};
\end{tikzpicture}
=
\begin{tikzpicture}[scale=.45, tinynodes, anchorbase,xscale=-1]
	\draw[very thick,<-] (0,-2) node[below,yshift=2pt]{$k$} to (0,-1);
	\draw[very thick] (0,-1) to (1,1) node[above,yshift=-2pt]{$1$};
	\draw[very thick] (-.5,0) to (0,1) node[above,yshift=-2pt]{$1$};
	\draw[very thick] (0,-1) to (-1,1) node[above,yshift=-2pt]{$1$};
	\draw[very thick] (0,-1) to node[pos=.25,right=-2pt]{$k{-}1$} (-1,1) node[above,yshift=-2pt]{$1$};
	\node at (-.5,.75) {$\mydots$};
\end{tikzpicture}  .
\end{equation}
Here, the horizontal symmetry encoded in the second equality follows from iterated use of (\text{\ref{eq:spn}d}).

\begin{cor} \label{cor:reduction} If $\Phi |_{\SWeb(\spn)}$ is full, then $\Phi$ is full. If $\Phi |_{\SWeb(\spn)}$ is faithful, then $\Phi$ is faithful. 
\end{cor}

\begin{proof}
It suffices to show that each object in $\Web(\spn)$ is a direct summand of an object in 
\[
\SWeb(\spn) \subset \Kar(\Web(\spn))).
\] 
For this, it further suffices to prove this for the generating objects $k$ in $\Web(\spn)$.
Let
\begin{equation}
\iota_k := \begin{tikzpicture}[scale=.25, tinynodes, anchorbase]
	\draw[very thick] (0,-2) node[below=-2pt]{$k$} to (0,-1);
	\draw[very thick] (0,-1) to (1,1) node[above=-2pt]{$1$};
	\draw[very thick] (0,-1) to (-1,1) node[above=-2pt]{$1$};
	\node at (0,.5) {$\mydots$};
\end{tikzpicture}
\quad \text{and} \quad
\pi_k :=
\frac{1}{[k]!}
\begin{tikzpicture}[scale=.25, tinynodes, anchorbase,yscale=-1]
	\draw[very thick] (0,-2) node[above=-2pt]{$k$} to (0,-1);
	\draw[very thick] (0,-1) to (1,1) node[below=-2pt]{$1$};
	\draw[very thick] (0,-1) to (-1,1) node[below=-2pt]{$1$};
	\node at (0,.5) {$\mydots$};
\end{tikzpicture}.
\end{equation}
Repeated application of (\text{\ref{eq:spn}c}) (and the horizontal symmetry in \eqref{eq:fullsplit}) gives that $\pi_k \circ \iota_k = \id_k$, as desired. 
\end{proof}

\subsection{Fullness and the BMW algebra}\label{sec:full}

It is known that the BMW algebra (see Definition \ref{def:BMW}) is ``quantum Brauer-Schur-Weyl'' dual
to $U_q(\spn)$.
Specifically, let
\begin{equation}\label{eq:BMWmap}
\psi \colon \BMW_k(-q^{2n+1},q-q^{-1}) \to \End_{U_q(\spn)}(V_1^{\ot k})
\end{equation}
be given by the natural action that quantizes the corresponding action of the Brauer algebra.
This homomorphism is known to be surjective; see e.g. \cite[Theorem 1.5]{Hu}.
Furthermore, the image under $\psi$ of $g_i$ is the braiding, and the image of $e_i$ is the composition of the unit and the counit 
(in the graphical language for $\Rep(U_q(\spn))$, the cup above the cap).
We show the analogue of \eqref{eq:BMWmap} for $\Web(\spn)$.

\begin{prop}
There is a homomorphism 
$\rho \colon \BMW(-q^{2n+1},q-q^{-1}) \to \SWeb(\spn)$
defined by
\[
\rho(g_i)=
\begin{tikzpicture}[scale=.75,anchorbase]
	\draw[very thick] (-1,0) to (-1,1.5);
	\node at (-.625,.75) {$\dots$};
	\draw[very thick] (-.25,0) to (-.25,1.5);
	\draw[very thick] (.75,0) to [out=90,in=-90] (0,1.5);
	\draw[overcross] (0,0) to [out=90,in=-90] (.75,1.5);
	\draw[very thick] (0,0) to [out=90,in=-90] (.75,1.5);
	\draw[very thick] (1,0) to (1,1.5);
	\node at (1.375,.75) {$\dots$};
	\draw[very thick] (1.75,0) to (1.75,1.5);
\end{tikzpicture}
\;\; , \;\;
\rho(g_i^{-1})=
\begin{tikzpicture}[scale=.75,xscale=-1,anchorbase]
    \draw[very thick] (-1,0) to (-1,1.5);
    \node at (-.625,.75) {$\dots$};
    \draw[very thick] (-.25,0) to (-.25,1.5);
	\draw[very thick] (.75,0) to [out=90,in=-90] (0,1.5);
	\draw[overcross] (0,0) to [out=90,in=-90] (.75,1.5);
	\draw[very thick] (0,0) to [out=90,in=-90] (.75,1.5);
	\draw[very thick] (1,0) to (1,1.5);
	\node at (1.375,.75) {$\dots$};
	\draw[very thick] (1.75,0) to (1.75,1.5);
\end{tikzpicture}
\;\; , \;\;
\rho(e_i)=
\begin{tikzpicture}[scale=.75,anchorbase]
	\draw[very thick] (-1,0) to (-1,1.5);
	\node at (-.625,.75) {$\dots$};
	\draw[very thick] (-.25,0) to (-.25,1.5);
	\draw[very thick] (0,0) to [out=90,in=180] (.375,.5);
	\draw[very thick] (.375,.5) to [out=0,in=90] (.75,0);
	\draw[very thick] (0,1.5) to [out=-90,in=180] (.375,1);
	\draw[very thick] (.375,1) to [out=0,in=-90] (.75,1.5);
	\draw[very thick] (1,0) to (1,1.5);
	\node at (1.375,.75) {$\dots$};
	\draw[very thick] (1.75,0) to (1.75,1.5);
\end{tikzpicture}.
\]
(Here, the crossings and caps/cups 
occur between the $i^{th}$ and $(i+1)^{st}$ positions.) Moreover, $\psi = \Phi \circ \rho$.
\end{prop}

\begin{proof}

Once we prove that $\rho$ is well-defined, the statement that $\psi = \Phi \circ \rho$ follows quickly. 
Indeed, Porism \ref{por:braiding} showed that $\Phi$ sends the overcrossing in \eqref{eq:braiding} to the braiding map $\beta_{V_1,V_1} \in \End_{U_q(\spn)}(V_1 \ot V_1)$. 
By \eqref{eq:images}, the cup and cap morphisms in $\Web(\spn)$ go to (any choice of) the (co)unit morphisms in $\FRep$ up to scalars $c_1$ and $c_1^{-1}$, 
and these scalars cancel each other in the composition of cup and cap. Thus $\psi$ agrees with $\Phi \circ \rho$ on all the generators of the BMW algebra.

A direct computation shows that $\rho$ is well-defined. Let us summarize the computation.
Here, the numbering corresponds to that in Definition \ref{def:BMW}
\begin{enumerate}
\item Equation \eqref{eq:braiding} implies that 
\begin{equation}
\begin{tikzpicture} [scale=.5,anchorbase]
	\draw[very thick] (1,0) to [out=90,in=270] (0,1.5);
	\draw[overcross] (0,0) to [out=90,in=270] (1,1.5);
	\draw[very thick] (0,0) to [out=90,in=270] (1,1.5);
\end{tikzpicture}
-
\begin{tikzpicture} [scale=.5,anchorbase]
	\draw[very thick] (0,0) to [out=90,in=270] (1,1.5);
	\draw[overcross] (1,0) to [out=90,in=270] (0,1.5);
	\draw[very thick] (1,0) to [out=90,in=270] (0,1.5);
\end{tikzpicture}
=
(q - q^{-1}) \;
\begin{tikzpicture}[scale =.4,anchorbase]
	\draw[very thick] (0,0) to (0,2.25);
	\draw[very thick] (1,0) to (1,2.25);
\end{tikzpicture}
+
\frac{q^{-n} - q^n}{[n]} \;
\begin{tikzpicture}[scale =.4,anchorbase]
	\draw[very thick] (0,0) to [out=90,in=180] (.5,.75);
	\draw[very thick] (1,0) to [out=90,in=0] (.5,.75);
	\draw[very thick] (.5,1.5) to [out=180,in=270] (0,2.25);
	\draw[very thick] (.5,1.5) to [out=0,in=270] (1,2.25);
\end{tikzpicture}
=
(q-q^{-1}) 
\left( \;
\begin{tikzpicture}[scale =.4,anchorbase]
	\draw[very thick] (0,0) to (0,2.25);
	\draw[very thick] (1,0) to (1,2.25);
\end{tikzpicture}
-
\begin{tikzpicture}[scale =.4,anchorbase]
	\draw[very thick] (0,0) to [out=90,in=180] (.5,.75);
	\draw[very thick] (1,0) to [out=90,in=0] (.5,.75);
	\draw[very thick] (.5,1.5) to [out=180,in=270] (0,2.25);
	\draw[very thick] (.5,1.5) to [out=0,in=270] (1,2.25);
\end{tikzpicture}
\; \right).
\end{equation}

\item Since $1+ \frac{r-r^{-1}}{z} = 1-[2n+1] = \qdim_1$, this relation follows by
observing that
$
\frac{1}{\qdim_1}
\begin{tikzpicture}[scale =.3,anchorbase]
	\draw[very thick] (0,0) to [out=90,in=180] (.5,.75);
	\draw[very thick] (1,0) to [out=90,in=0] (.5,.75);
	\draw[very thick] (.5,1.5) to [out=180,in=270] (0,2.25);
	\draw[very thick] (.5,1.5) to [out=0,in=270] (1,2.25);
\end{tikzpicture}
$
is an idempotent in $\Web(\spn)$.

\item First, a straightforward (but tedious) computation shows that

\begin{equation}
\begin{tikzpicture}[scale=.75,anchorbase,tinynodes]
	\draw[very thick] (.5,-1) to [out=90,in=210] (.25,.25);
	\draw[very thick] (1,-1) to [out=90,in=-30] (.25,.25);
	\draw[very thick] (.25,.25) to [out=90,in=-90] node[left=-2pt]{$2$} (.25,.75);
	\draw[overcross] (0,-1) to [out=90,in=-90] (1,.75);
	\draw[very thick] (0,-1) to [out=90,in=-90] (1,.75);
\end{tikzpicture}
=
-q
\begin{tikzpicture}[scale=.75,anchorbase,tinynodes]
	\draw[very thick] (0,-1) to [out=90,in=210] (.25,.25);
	\draw[very thick] (.5,-1) to [out=90,in=210] (.75,-.5);
	\draw[very thick] (1,-1) to [out=90,in=-30] (.75,-.5);
	\draw[very thick] (.75,-.5) to node[right=-2pt]{$2$} (.75,0);
	\draw[very thick] (.75,0) to [out=150,in=-30] (.25,.25);
	\draw[very thick] (.75,0) to [out=30,in=-90] (1,.75);
	\draw[very thick] (.25,.25) to node[left=-2pt]{$2$} (.25,.75);
\end{tikzpicture}
+
\begin{tikzpicture}[scale=.75,anchorbase,tinynodes]
	\draw[very thick] (0,-1) to [out=90,in=210] (.5,-.25);
	\draw[very thick] (.5,-1) to [out=90,in=210] (.75,-.625);
	\draw[very thick] (1,-1) to [out=90,in=-30] (.75,-.625);
	\draw[very thick] (.75,-.625) to [out=90,in=330] node[right=-2pt]{$2$} (.5,-.25);
	\draw[very thick] (.5,-.25) to node[left=-2pt]{$3$} (.5,.25);
	\draw[very thick] (.5,.25) to [out=30,in=270](.875,.75);
	\draw[very thick] (.5,.25) to [out=150,in=270] node[left=-2pt,pos=.8]{$2$} (.125,.75);
\end{tikzpicture}
-\frac{q^{-n+1}}{[n-1]}
\begin{tikzpicture}[scale=.75,anchorbase,tinynodes]
	\draw[very thick] (0,-1) to [out=90,in=210] (.5,-.25);
	\draw[very thick] (.5,-1) to [out=90,in=210] (.75,-.625);
	\draw[very thick] (1,-1) to [out=90,in=-30] (.75,-.625);
	\draw[very thick] (.75,-.625) to [out=90,in=330] node[right=-2pt]{$2$} (.5,-.25);
	\draw[very thick] (.5,-.25) to node[left=-3pt]{$1$} (.5,.25);
	\draw[very thick] (.5,.25) to [out=30,in=270](.875,.75);
	\draw[very thick] (.5,.25) to [out=150,in=270] node[left=-2pt,pos=.8]{$2$} (.125,.75);
\end{tikzpicture}
\end{equation}
so
\begin{equation} \label{eq:foobarrrr}
\begin{tikzpicture}[scale=.75,anchorbase,tinynodes]
    \draw[very thick] (.5,-1) to [out=90,in=210] (.25,.25);
    \draw[very thick] (1,-1) to [out=90,in=-30] (.25,.25);
    \draw[very thick] (.25,.25) to [out=90,in=-90] node[left=-2pt]{$2$} (.25,.75);
    \draw[overcross] (0,-1) to [out=90,in=-90] (1,.75);
    \draw[very thick] (0,-1) to [out=90,in=-90] (1,.75) to (1,1.25);
    \draw[very thick] (.25,.75) to [out=150,in=-90] (0,1.25);
    \draw[very thick] (.25,.75) to [out=30,in=-90] (.5,1.25);
\end{tikzpicture}
=-q
\begin{tikzpicture}[scale=.75,anchorbase,tinynodes]
    \draw[very thick] (0,-1) to [out=90,in=210] (.25,.25);
    \draw[very thick] (.5,-1) to [out=90,in=210] (.75,-.5);
    \draw[very thick] (1,-1) to [out=90,in=-30] (.75,-.5);
    \draw[very thick] (.75,-.5) to node[right]{$2$} (.75,0);
    \draw[very thick] (.75,0) to [out=150,in=-30] (.25,.25);
    \draw[very thick] (.75,0) to [out=30,in=-90] (1,1.25);
    \draw[very thick] (.25,.25) to node[left]{$2$} (.25,.75);
    \draw[very thick] (.25,.75) to [out=150,in=-90] (0,1.25);
    \draw[very thick] (.25,.75) to [out=30,in=-90] (.5,1.25);
\end{tikzpicture}
+
\begin{tikzpicture}[scale=.75,anchorbase,tinynodes]
	\draw[very thick] (0,-1) to [out=90,in=210] (.5,-.125);
	\draw[very thick] (.5,-1) to [out=90,in=210] (.75,-.625);
	\draw[very thick] (1,-1) to [out=90,in=-30] (.75,-.625);
	\draw[very thick] (.75,-.625) to [out=90,in=330] node[right=-2pt]{$2$} (.5,-.125);
	\draw[very thick] (.5,-.125) to node[left=-2pt]{$3$} (.5,.375);
	\draw[very thick] (.5,.375) to [out=150,in=270] node[left=-2pt]{$2$} (.25,.875);
	\draw[very thick] (.25,.875) to [out=150,in=-90] (0,1.25);
	\draw[very thick] (.25,.875) to [out=30,in=-90] (.5,1.25);
	\draw[very thick] (.5,.375) to [out=30,in=270] (1,1.25);
\end{tikzpicture}
-\frac{q^{-n+1}}{[n-1]}
\begin{tikzpicture}[scale=.75,anchorbase,tinynodes]
	\draw[very thick] (0,-1) to [out=90,in=210] (.5,-.125);
	\draw[very thick] (.5,-1) to [out=90,in=210] (.75,-.625);
	\draw[very thick] (1,-1) to [out=90,in=-30] (.75,-.625);
	\draw[very thick] (.75,-.625) to [out=90,in=330] node[right=-2pt]{$2$} (.5,-.125);
	\draw[very thick] (.5,-.125) to node[left=-3pt]{$1$} (.5,.375);
	\draw[very thick] (.5,.375) to [out=150,in=270] node[left=-2pt]{$2$} (.25,.875);
	\draw[very thick] (.25,.875) to [out=150,in=-90] (0,1.25);
	\draw[very thick] (.25,.875) to [out=30,in=-90] (.5,1.25);
	\draw[very thick] (.5,.375) to [out=30,in=270] (1,1.25);
\end{tikzpicture}.
\end{equation}
Because the right-hand side of \eqref{eq:foobarrrr} is rotationally symmetric, we have
\begin{equation}\label{eq:cruxofR3}
\begin{tikzpicture}[scale=.75,anchorbase,tinynodes]
    \draw[very thick] (.5,-1) to [out=90,in=210] (.25,.25);
    \draw[very thick] (1,-1) to [out=90,in=-30] (.25,.25);
    \draw[very thick] (.25,.25) to [out=90,in=-90] node[left=-2pt]{$2$} (.25,.75);
    \draw[overcross] (0,-1) to [out=90,in=-90] (1,.75);
    \draw[very thick] (0,-1) to [out=90,in=-90] (1,.75) to (1,1.25);
    \draw[very thick] (.25,.75) to [out=150,in=-90] (0,1.25);
    \draw[very thick] (.25,.75) to [out=30,in=-90] (.5,1.25);
\end{tikzpicture}
=
\begin{tikzpicture}[scale=.75,anchorbase,tinynodes,rotate=180]
    \draw[very thick] (.5,-1) to [out=90,in=210] (.25,.25);
    \draw[very thick] (1,-1) to [out=90,in=-30] (.25,.25);
    \draw[very thick] (.25,.25) to [out=90,in=-90] node[right=-2pt]{$2$} (.25,.75);
    \draw[overcross] (0,-1) to [out=90,in=-90] (1,.75);
    \draw[very thick] (0,-1) to [out=90,in=-90] (1,.75) to (1,1.25);
    \draw[very thick] (.25,.75) to [out=150,in=-90] (0,1.25);
    \draw[very thick] (.25,.75) to [out=30,in=-90] (.5,1.25);
\end{tikzpicture},
\end{equation}
which gives
\begin{equation}
\begin{tikzpicture}[scale=.55,anchorbase,tinynodes]
	\draw[very thick] (.8,-1) to [out=90,in=270] (0,.5);
	\draw[very thick] (1.6,-1) to [out=90,in=270] (0,2);
	\draw[overcross] (0,.5) to [out=90,in=270] (.8,2);
	\draw[very thick] (0,.5) to [out=90,in=270] (.8,2);
	\draw[overcross] (0,-1) to [out=90,in=270] (1.6,2);
	\draw[very thick] (0,-1) to [out=90,in=270] (1.6,2);
\end{tikzpicture}
=q
\begin{tikzpicture}[scale=.55,anchorbase,tinynodes]
	\draw[very thick] (.8,-1) to [out=90,in=270] (0,1.2) to (0,2);
	\draw[very thick] (1.6,-1) to [out=90,in=270] (.8,1.2) to (.8,2);
	\draw[overcross] (0,-1) to [out=90,in=270] (1.6,1.2);
	\draw[very thick] (0,-1) to [out=90,in=270] (1.6,1.2) to (1.6,2);
\end{tikzpicture}
+\frac{q^{-n}}{[n]}
\begin{tikzpicture}[scale=.55,anchorbase,tinynodes]
	\draw[very thick] (.8,-1) to [out=90,in=270] (0,.4);
	\draw[very thick] (1.6,-1) to [out=90,in=270] (.8,.4);
	\draw[overcross] (0,-1) to [out=90,in=270] (1.6,.8);
	\draw[very thick] (0,-1) to [out=90,in=270] (1.6,.8);
	\draw[very thick] (0,.4) to [out=90,in=180] (.4,.8);
	\draw[very thick] (.8,.4) to [out=90,in=0] (.4,.8);
	\draw[very thick] (.4,1.6) to [out=180,in=270] (0,2);
	\draw[very thick] (.4,1.6) to [out=0,in=270] (.8,2);
	\draw[very thick] (1.6,.8) to (1.6,2);
\end{tikzpicture}
-
\begin{tikzpicture}[scale=.75,anchorbase,tinynodes]
    \draw[very thick] (.5,-1) to [out=90,in=210] (.25,.25);
    \draw[very thick] (1,-1) to [out=90,in=-30] (.25,.25);
    \draw[very thick] (.25,.25) to [out=90,in=-90] node[left=-2pt]{$2$} (.25,.75);
    \draw[overcross] (0,-1) to [out=90,in=-90] (1,.75);
    \draw[very thick] (0,-1) to [out=90,in=-90] (1,.75) to (1,1.25);
    \draw[very thick] (.25,.75) to [out=150,in=-90] (0,1.25);
    \draw[very thick] (.25,.75) to [out=30,in=-90] (.5,1.25);
\end{tikzpicture}
=q
\begin{tikzpicture}[scale=.55,rotate=180,anchorbase,tinynodes]
	\draw[very thick] (.8,-1) to [out=90,in=270] (0,1.2) to (0,2);
	\draw[very thick] (1.6,-1) to [out=90,in=270] (.8,1.2) to (.8,2);
	\draw[overcross] (0,-1) to [out=90,in=270] (1.6,1.2);
	\draw[very thick] (0,-1) to [out=90,in=270] (1.6,1.2) to (1.6,2);
\end{tikzpicture}
+\frac{q^{-n}}{[n]}
\begin{tikzpicture}[scale=.55,rotate=180,anchorbase,tinynodes]
	\draw[very thick] (.8,-1) to [out=90,in=270] (0,.4);
	\draw[very thick] (1.6,-1) to [out=90,in=270] (.8,.4);
	\draw[overcross] (0,-1) to [out=90,in=270] (1.6,.8);
	\draw[very thick] (0,-1) to [out=90,in=270] (1.6,.8);
	\draw[very thick] (0,.4) to [out=90,in=180] (.4,.8);
	\draw[very thick] (.8,.4) to [out=90,in=0] (.4,.8);
	\draw[very thick] (.4,1.6) to [out=180,in=270] (0,2);
	\draw[very thick] (.4,1.6) to [out=0,in=270] (.8,2);
	\draw[very thick] (1.6,.8) to (1.6,2);
\end{tikzpicture}
-
\begin{tikzpicture}[scale=.75,anchorbase,tinynodes,rotate=180]
    \draw[very thick] (.5,-1) to [out=90,in=210] (.25,.25);
    \draw[very thick] (1,-1) to [out=90,in=-30] (.25,.25);
    \draw[very thick] (.25,.25) to [out=90,in=-90] node[right=-2pt]{$2$} (.25,.75);
    \draw[overcross] (0,-1) to [out=90,in=-90] (1,.75);
    \draw[very thick] (0,-1) to [out=90,in=-90] (1,.75) to (1,1.25);
    \draw[very thick] (.25,.75) to [out=150,in=-90] (0,1.25);
    \draw[very thick] (.25,.75) to [out=30,in=-90] (.5,1.25);
\end{tikzpicture}
=
\begin{tikzpicture}[scale=.55,anchorbase,tinynodes]
	\draw[very thick] (1.6,-1) to [out=90,in=270] (0,2);
	\draw[very thick] (1.6,.5) to [out=90,in=270] (.8,2);
	\draw[overcross] (.8,-1) to [out=90,in=270] (1.6,.5);
	\draw[very thick] (.8,-1) to [out=90,in=270] (1.6,.5);
	\draw[overcross] (0,-1) to [out=90,in=270] (1.6,2);
	\draw[very thick] (0,-1) to [out=90,in=270] (1.6,2);
\end{tikzpicture} .
\end{equation}
Here the second equality holds by using \ref{eq:cruxofR3} to rewrite the third term
and by applying the Reidemeister II (RII) move twice to the second term.
Note that the RII move holds by Porism \ref{por:braiding}.

\item This holds by a height exchange isotopy in $\Web(\spn)$.

\item This holds via planar isotopy in $\Web(\spn)$.

\item This follows from the RII move and planar isotopy.

\item This, and relation (8), follow from the framed Reidemeister I moves:
\[
\begin{tikzpicture}[scale=.5,yscale=-1,anchorbase]
	\draw[very thick] (.8,-.4) to [out=180,in=270] (0,1);
	\draw[overcross] (0,-1) to [out=90,in=180] (.8,.4);
	\draw[very thick] (0,-1) to [out=90,in=180] (.8,.4);
	\draw[very thick] (.8,.4) to [out=0,in=90] (1.1,0) to [out=270,in=0] (.8,-.4);
\end{tikzpicture}
= \left( q +\frac{q^{-n}}{[n]}\left(-\frac{[n][2n+2]}{[n+1]}\right) \right) \
\begin{tikzpicture}[scale=.5,anchorbase]
	\draw[very thick] (0,-1) to (0,1);
\end{tikzpicture}
=
-q^{-(2n+1)} \
\begin{tikzpicture}[scale=.5,anchorbase]
	\draw[very thick] (0,-1) to (0,1);
\end{tikzpicture}
\quad \text{and} \quad
\begin{tikzpicture}[scale=.5,anchorbase]
	\draw[very thick] (.8,-.4) to [out=180,in=270] (0,1);
	\draw[overcross] (0,-1) to [out=90,in=180] (.8,.4);
	\draw[very thick] (0,-1) to [out=90,in=180] (.8,.4);
	\draw[very thick] (.8,.4) to [out=0,in=90] (1.1,0) to [out=270,in=0] (.8,-.4);
\end{tikzpicture}
= 
-q^{2n+1}
\begin{tikzpicture}[scale=.5,anchorbase]
	\draw[very thick] (0,-1) to (0,1);
\end{tikzpicture} \ .
\]
(Note that the latter follows from the former, 
since Porism \ref{por:braiding} shows that the positive braiding is obtained from the negative by replacing $q$ with $q^{-1}$.)
\end{enumerate}
\end{proof}

We now prove fullness of $\Phi$.

\begin{thm}\label{thm:Full}
The functor $\Phi \colon \Web(\spn) \to \FRep(U_q(\spn))$ is full.
\end{thm}

\begin{proof}
The argument is analogous to the proof of \cite[Theorem 3.5]{RoseTatham}, 
so we argue succinctly.

As noted above, the homomorphism in \eqref{eq:BMWmap} is surjective. 
Since this factors through 
\[
\Phi \colon  \End_{\Web(\spn)}(1^{\otimes k}) \to  \End_{U_q(\spn)}(V_1^{\otimes k}),
\] 
this latter homomorphism is surjective as well. 
Using adjunction in $\SWeb(\spn)$ and $\SRep(U_q(\spn))$, this implies that 
\[
\Phi \colon  \Hom_{\Web(\spn)}(1^{\otimes k_1},1^{\otimes k_2}) \to \Hom_{U_q(\spn)}(V_1^{\otimes k_1},V_1^{\otimes k_2}),
\]
is surjective when $k_1+k_2$ is even.
(Note that this map is trivially surjective when 
$k_1+k_2$ is odd since in that case the codomain is zero.)
Thus, $\Phi|_{\SWeb(\spn)}$ is full, and the result then follows from Corollary \ref{cor:reduction}.
\end{proof}

\subsection{The quadrivalent category}\label{sec:quadcat}

We now begin the arguments necessary to show that $\Phi$ is faithful.
The first step is the following ``change of basis'' in the $\Hom$-space $\End_{\Web(\spn)} (1 \otimes 1)$.

\begin{defn}\label{def:quad}
Let $\Web^{\times}(\spn)$ be the $\C(q)$-linear pivotal subcategory of $\SWeb(\spn)$ generated by 
$\id_1$ and the rotationally symmetric quadrivalent vertex:
\begin{equation}\label{eq:quad}
\begin{tikzpicture}[scale=.35, tinynodes, anchorbase]
	\draw[very thick] (-1,-1) to (1,1);
	\draw[very thick] (-1,1) to (1,-1);
\end{tikzpicture}
=
\begin{tikzpicture}[scale=.35, tinynodes, anchorbase]
	\draw[very thick] (-1,0) node[below,yshift=2pt,xshift=-2pt]{$1$} to (0,1);
	\draw[very thick] (1,0) node[below,yshift=2pt,xshift=2pt]{$1$} to (0,1);
	\draw[very thick] (0,2.5) to (-1,3.5) node[above,yshift=-4pt,xshift=-2pt]{$1$};
	\draw[very thick] (0,2.5) to (1,3.5) node[above,yshift=-4pt,xshift=2pt]{$1$};
	\draw[very thick] (0,1) to node[right,xshift=-2pt]{$2$} (0,2.5);
\end{tikzpicture}
+
\frac{[n-1]}{[n]} \;
\begin{tikzpicture}[scale=.3, tinynodes, anchorbase]
	\draw[very thick] (-1,0) node[below,yshift=2pt]{$1$} to [out=90,in=180] (0,1) 
		to [out=0,in=90] (1,0);
	\draw[very thick] (-1,3) node[above,yshift=-3pt]{$1$} to [out=270,in=180] (0,2)
		to [out=0,in=270] (1,3);
\end{tikzpicture}
\stackrel{(\text{\ref{eq:spn}e})}{=}
\begin{tikzpicture}[scale=.35, rotate=90, tinynodes, anchorbase]
	\draw[very thick] (-1,0) node[below,yshift=2pt,xshift=2pt]{$1$} to (0,1);
	\draw[very thick] (1,0) node[above,yshift=-4pt,xshift=2pt]{$1$} to (0,1);
	\draw[very thick] (0,2.5) to (-1,3.5) node[below,yshift=2pt,xshift=-2pt]{$1$};
	\draw[very thick] (0,2.5) to (1,3.5) node[above,yshift=-4pt,xshift=-2pt]{$1$};
	\draw[very thick] (0,1) to node[below,yshift=2pt]{$2$} (0,2.5);
\end{tikzpicture}
+
\frac{[n-1]}{[n]} \;
\begin{tikzpicture}[scale=.35, tinynodes, anchorbase]
	\draw[very thick] (-1,-1) node[below,yshift=2pt]{$1$} to (-1,1) node[above,yshift=-2pt]{$1$};
	\draw[very thick] (1,-1) node[below,yshift=2pt]{$1$} to (1,1) node[above,yshift=-2pt]{$1$};
\end{tikzpicture}
\end{equation}
\end{defn}

\begin{conv}
We will omit labels on any strands involved in quadrivalent vertices, 
since they are $1$-labeled by definition.
\end{conv}

Note that morphisms in $\Web^{\times}(\spn)$ are simply given by $\C(q)$-linear combinations of 
quadrivalent graphs, modulo the relations on such graphs implied by \eqref{eq:spn}. 
We will use the shorthand
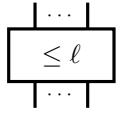
\[
\begin{tikzpicture}[scale=.35, anchorbase]
	\draw[very thick] (-1,-2) to (-1,2);
	\draw[very thick] (1,-2) to (1,2);
	\draw[very thick,fill=white] (-2,-1) rectangle (2,1);
	\node at (0,0) {$\leq \ell$};
	\node at (0,-1.5) {\scriptsize$\cdots$};
	\node at (0,1.5) {\scriptsize$\cdots$};
\end{tikzpicture}
\]
to denote any morphism in $\Hom_{\Web^{\times}(\spn)}(1^{\otimes k_1},1^{\otimes k_2})$ that is a linear combination 
of quadrivalent graphs with at most $\ell$ vertices. 

We now study $\Hom$-spaces in $\Web^{\times}(\spn)$. 
Recursively define the \emph{half twist} elements $\HT_k \in \End_{\Web^{\times}(\spn)}(1^{\otimes k})$ via
\begin{equation}\label{eq:HTdef}
\HT_1 = \id_1
\quad , \quad
\begin{tikzpicture}[scale=.35, anchorbase]
	\draw[very thick] (-1,-2) to (-1,2);
	\draw[very thick] (1,-2) to (1,2);
	\draw[very thick,fill=white] (-2,-1) rectangle (2,1);
	\node at (0,0) {$\HT_k$};
	\node at (0,-1.5) {\scriptsize$\cdots$};
	\node at (0,1.5) {\scriptsize$\cdots$};
\end{tikzpicture}
=
\begin{tikzpicture}[scale=.35, anchorbase]
	\draw[very thick] (-1,-2) to (-1,1) to [out=90,in=270] (1,4);
	\draw[very thick] (1,-2) to (1,1) to [out=90,in=270] (3,4);
	\draw[very thick,fill=white] (-2,-1) rectangle (2,1);
	\node at (0,0) {$\HT_{k{-}1}$};
	\node at (0,-1.5) {\scriptsize$\cdots$};
	\node at (0,1.25) {\scriptsize$\cdots$};
	\draw[very thick] (3,-2) to (3,1) to [out=90,in=270] (-1,4);
\end{tikzpicture}
\end{equation}
We now record a crucial identity, which will have a number of important consequences. 
We will make use of the full split vertices from \eqref{eq:fullsplit}.

\begin{lem}\label{lem:HT}
For $1 \leq k \leq n$:
\begin{equation}\label{eq:HT}
\HT_{k+1} = 
\begin{tikzpicture}[scale=.35, tinynodes, anchorbase]
	\draw[very thick] (-1,2) to (0,.5);
	\node at (0,1.5) {$\mydots$};
	\draw[very thick] (1,2) to (0,.5);
	\draw[very thick] (0,-.5) to node[right,xshift=-2pt]{$k{+}1$} (0,.5);
	\draw[very thick] (-1,-2) to (0,-.5);
	\node at (0,-1.5) {$\mydots$};
	\draw[very thick] (1,-2) to (0,-.5);
\end{tikzpicture}
+
\begin{tikzpicture}[scale=.35, smallnodes, anchorbase]
	\draw[very thick] (-1,-2) to (-1,2);
	\draw[very thick] (1,-2) to (1,2);
	\draw[very thick,fill=white] (-3,-1) rectangle (3,1);
	\node at (0,0) {$\leq \frac{1}{2}k(k{+}1){-}1$};
	\node at (0,-1.5) {\scriptsize$\cdots$};
	\node at (0,1.5) {\scriptsize$\cdots$};
\end{tikzpicture}
\end{equation}
\end{lem}

\begin{proof}
Let $P(k)$ be the equality in the statement of the lemma, and let
\begin{equation}
Q(k) :
\begin{tikzpicture}[scale=.35, tinynodes, anchorbase]
	\draw[very thick] (0,-2) node[below,yshift=2pt]{$k$} to (0,-1);
	\draw[very thick] (0,-1) to (1,2);
	\draw[very thick] (0,-1) to (-1,2);
	\draw[very thick] (1,-2) to (1,.5) to [out=90,in=270] (-2,2);
	\node at (0,.5) {$\mydots$};
\end{tikzpicture}
\ =
\begin{tikzpicture}[scale=.35, tinynodes, anchorbase]
	\draw[very thick] (-1,2) to (0,0);
	\node at (0,1) {$\mydots$};
	\draw[very thick] (1,2) to (0,0);
	\draw[very thick] (0,-1) to node[right,xshift=-2pt]{$k{+}1$} (0,0);
	\draw[very thick] (-.5,-2) node[below,yshift=2pt]{$k$} to (0,-1);
	\draw[very thick] (.5,-2) to (0,-1);
\end{tikzpicture}
+
\begin{tikzpicture}[scale=.35, smallnodes, anchorbase]
	\draw[very thick] (-1.5,1.5) to (-1.5,2);
	\draw[very thick] (1.5,-2) to (1.5,2);
	\draw[very thick] (-.5,-2) node[below,yshift=2pt]{$k$} to (-.5,-1.5);
	\draw[very thick] (-.5,-1.5) to (-1.5,0);
	\node at (-.5,-.5) {$\mydots$};
	\draw[very thick] (-.5,-1.5) to (.5,0);
	\draw[very thick,fill=white] (-2,0) rectangle (2,1.5);
	\node at (0,.75) {$\leq k{-}1$};
	\node at (0,1.75) {\scriptsize$\cdots$};
\end{tikzpicture}.
\end{equation}
We will prove $Q(k)$ by induction, 
and then deduce our result by showing that $P(k-1)$ and $Q(k-1)$ imply $P(k)$.
Note that $Q(k)$ implies the relation
\begin{equation}
Q'(k) :
\begin{tikzpicture}[scale=.35, tinynodes, anchorbase]
	\draw[very thick] (-1,2) to (0,0);
	\node at (0,1) {$\mydots$};
	\draw[very thick] (1,2) to (0,0);
	\draw[very thick] (0,-1) to node[right,xshift=-2pt]{$k{+}1$} (0,0);
	\draw[very thick] (-.5,-2) node[below,yshift=2pt]{$k$} to (0,-1);
	\draw[very thick] (.5,-2) to (0,-1);
\end{tikzpicture}
=
\begin{tikzpicture}[scale=.35, smallnodes, anchorbase]
	\draw[very thick] (-1.5,1.5) to (-1.5,2);
	\draw[very thick] (1.5,-2) to (1.5,2);
	\draw[very thick] (-.5,-2) node[below,yshift=2pt]{$k$} to (-.5,-1.5);
	\draw[very thick] (-.5,-1.5) to (-1.5,0);
	\node at (-.5,-.5) {$\mydots$};
	\draw[very thick] (-.5,-1.5) to (.5,0);
	\draw[very thick,fill=white] (-2,0) rectangle (2,1.5);
	\node at (0,.75) {$\leq k$};
	\node at (0,1.75) {\scriptsize$\cdots$};
\end{tikzpicture}
\end{equation}
which will be used in the induction below.

First, note that $Q(1)$ holds by the definition of the quadrivalent vertex.
Thus, we assume $Q=Q(k-1)$ and $Q'=Q'(k-1)$ and compute
\begin{equation}
\begin{aligned}
\begin{tikzpicture}[scale=.35, tinynodes, anchorbase]
	\draw[very thick] (0,-2) node[below,yshift=2pt]{$k$} to (0,-1);
	\draw[very thick] (0,-1) to (1,2);
	\draw[very thick] (0,-1) to (-1,2);
	\draw[very thick] (1,-2) to (1,.5) to [out=90,in=270] (-2,2);
	\node at (0,.5) {$\mydots$};
\end{tikzpicture} \
&\stackrel{Q}{=}
\begin{tikzpicture}[scale=.35, tinynodes, anchorbase]
	\draw[very thick] (-1,2) to (0,0);
	\node at (0,1) {$\mydots$};
	\draw[very thick] (1,2) to (0,0);
	\draw[very thick,directed=.6] (0,-1) to node[left,xshift=2pt]{$k$} (0,0);
	\draw[very thick] (0,-1) [out=330,in=90] to (1,-3);
	\draw[very thick] (0,-2) to [out=150,in=210] (0,-1);
	\draw[very thick] (0,-2) to [out=30,in=270] (2,2);
	\draw[very thick,<-] (0,-3) node[below,yshift=2pt]{$k$} to (0,-2);
\end{tikzpicture}
+
\begin{tikzpicture}[scale=.35, tinynodes, anchorbase]
	\draw[very thick] (-1.5,1.5) to (-1.5,2);
	\draw[very thick] (1.5,-3) to [out=90,in=270] (1.5,0) to (1.5,2);
	\draw[very thick] (0,-2.5) to [out=150,in=270] (-.5,-1.5);
	\draw[very thick] (0,-2.5) to [out=30,in=270] (2.5,2);
	\draw[very thick] (0,-3) node[below,yshift=2pt]{$k$} to (0,-2.5);
	\draw[very thick] (-.5,-1.5) to (-1.5,0);
	\node at (-.5,-.5) {$\mydots$};
	\draw[very thick] (-.5,-1.5) to (.5,0);
	\draw[very thick,fill=white] (-1.875,0) rectangle (1.875,1.5);
	\node at (0,.75) {\footnotesize$\leq k{-}2$};
	\node at (0,1.75) {\scriptsize$\cdots$};
\end{tikzpicture}
\stackrel{\eqref{eq:quad}}{=}
\begin{tikzpicture}[scale=.35, tinynodes, anchorbase]
	\draw[very thick] (-1,2) to (0,0);
	\node at (0,1) {$\mydots$};
	\draw[very thick] (1,2) to (0,0);
	\draw[very thick,directed=.6] (0,-1) to node[left,xshift=2pt]{$k$} (0,0);
	\draw[very thick] (0,-1) [out=330,in=120] to (.5,-1.5);
	\draw[very thick] (.5,-1.5) to node[above,yshift=-3pt]{$2$} (1.25,-1.5);
	\draw[very thick] (1.25,-1.5) to [out=60,in=270] (2,2);
	\draw[very thick] (1.25,-1.5) to [out=300,in=90] (1.5,-3);
	\draw[very thick] (0,-2) to [out=150,in=210] node[left=-3pt]{$k{-}1$} (0,-1);
	\draw[very thick] (0,-2) to [out=30,in=240] (.5,-1.5);
	\draw[very thick,<-] (0,-3) node[below,yshift=2pt]{$k$} to (0,-2);
\end{tikzpicture}
+
\frac{[n{-}1]}{[n]}
\begin{tikzpicture}[scale=.35, tinynodes, anchorbase]
	\draw[very thick] (-1,2) to (0,0);
	\node at (0,1) {$\mydots$};
	\draw[very thick] (1,2) to (0,0);
	\draw[very thick] (0,-1) to node[left,xshift=2pt]{$k$} (0,0);
	\draw[very thick] (0,-2) to [out=150,in=210] node[left=-3pt]{$k{-}1$} (0,-1);
	\draw[very thick] (0,-2) to [out=30,in=330] (0,-1);
	\draw[very thick] (0,-3) node[below,yshift=2pt]{$k$} to (0,-2);
	\draw[very thick] (1,-3) to [out=90,in=270] (2,2);
\end{tikzpicture}
+
\begin{tikzpicture}[scale=.35, tinynodes, anchorbase]
	\draw[very thick] (-1.5,1.5) to (-1.5,2);
	\draw[very thick] (1.5,-2) to (1.5,2);
	\draw[very thick] (-.5,-2) node[below,yshift=2pt]{$k$} to (-.5,-1.5);
	\draw[very thick] (-.5,-1.5) to (-1.5,0);
	\node at (-.5,-.5) {$\mydots$};
	\draw[very thick] (-.5,-1.5) to (.5,0);
	\draw[very thick,fill=white] (-2,0) rectangle (2,1.5);
	\node at (0,.75) {\footnotesize$\leq k{-}1$};
	\node at (0,1.75) {\scriptsize$\cdots$};
\end{tikzpicture} \\
&\stackrel{(\text{\ref{eq:spn}e})}{=}
\begin{tikzpicture}[scale=.35, tinynodes, anchorbase]
	\draw[very thick] (-1,2) to (0,0);
	\node at (0,1) {$\mydots$};
	\draw[very thick] (1,2) to (0,0);
	\draw[very thick] (0,-1) to node[right,xshift=-2pt]{$k{+}1$} (0,0);
	\draw[very thick] (-.5,-2) node[below,yshift=2pt]{$k$} to (0,-1);
	\draw[very thick] (.5,-2) to (0,-1);
\end{tikzpicture}
+
\frac{[n{-}k]}{[n{-}k{+}1]}
\begin{tikzpicture}[scale=.35, tinynodes, anchorbase]
	\draw[very thick] (-1.5,3) to (0,0) node[left,yshift=3pt,xshift=1pt]{$k$};
	\draw[very thick] (.5,3) to (-.5,1);
	\node at (-.5,2) {$\mydots$};
	\draw[very thick] (1.5,3) to (0,0);
	\draw[very thick] (0,-1) to node[right,xshift=-2pt]{$k{-}1$} (0,0);
	\draw[very thick] (-.5,-2) node[below,yshift=2pt]{$k$} to (0,-1);
	\draw[very thick] (.5,-2) to (0,-1);
\end{tikzpicture}
+
\frac{[n-1]-[n{-}k]}{[n]}
\begin{tikzpicture}[scale=.35, tinynodes, anchorbase]
	\draw[very thick] (-1,2) to (0,0);
	\node at (0,1) {$\mydots$};
	\draw[very thick] (1,2) to (0,0);
	\draw[very thick] (0,-1) to node[left,xshift=2pt]{$k$} (0,0);
	\draw[very thick] (0,-2) to [out=150,in=210] node[left=-3pt]{$k{-}1$} (0,-1);
	\draw[very thick] (0,-2) to [out=30,in=330] (0,-1);
	\draw[very thick] (0,-3) node[below,yshift=2pt]{$k$} to (0,-2);
	\draw[very thick] (1,-3) to [out=90,in=270] (2,2);
\end{tikzpicture}
+
\begin{tikzpicture}[scale=.35, tinynodes, anchorbase]
	\draw[very thick] (-1.5,1.5) to (-1.5,2);
	\draw[very thick] (1.5,-2) to (1.5,2);
	\draw[very thick] (-.5,-2) node[below,yshift=2pt]{$k$} to (-.5,-1.5);
	\draw[very thick] (-.5,-1.5) to (-1.5,0);
	\node at (-.5,-.5) {$\mydots$};
	\draw[very thick] (-.5,-1.5) to (.5,0);
	\draw[very thick,fill=white] (-2,0) rectangle (2,1.5);
	\node at (0,.75) {\footnotesize$\leq k{-}1$};
	\node at (0,1.75) {\scriptsize$\cdots$};
\end{tikzpicture} \\
&\stackrel{Q'}{=}
\begin{tikzpicture}[scale=.35, tinynodes, anchorbase]
	\draw[very thick] (-1,2) to (0,0);
	\node at (0,1) {$\mydots$};
	\draw[very thick] (1,2) to (0,0);
	\draw[very thick] (0,-1) to node[right,xshift=-2pt]{$k{+}1$} (0,0);
	\draw[very thick] (-.5,-2) node[below,yshift=2pt]{$k$} to (0,-1);
	\draw[very thick] (.5,-2) to (0,-1);
\end{tikzpicture}
+
\begin{tikzpicture}[scale=.35, tinynodes, anchorbase]
	\draw[very thick] (-1.5,1.5) to (-1.5,2);
	\draw[very thick] (2.5,2) to (2.5,-.5) to [out=270,in=0] (2,-1) to [out=180,in=270] (1.5,-.5) to (1.5,2);
	\draw[very thick] (-.5,-2.5) to node[left,xshift=2pt]{$k{-}1$} (-.5,-1.5);
	\draw[very thick] (-1,-3) to (-.5,-2.5);
	\draw[very thick] (0,-3) to (-.5,-2.5);
	\draw[very thick] (-.5,-1.5) to (-1.5,0);
	\node at (-.5,-.5) {$\mydots$};
	\draw[very thick] (-.5,-1.5) to (.5,0);
	\draw[very thick,fill=white] (-2,0) rectangle (2,1.5);
	\node at (0,.75) {\footnotesize$\leq k{-}1$};
	\node at (0,1.75) {\scriptsize$\cdots$};
\end{tikzpicture}
+
\begin{tikzpicture}[scale=.35, tinynodes, anchorbase]
	\draw[very thick] (-1.5,1.5) to (-1.5,2);
	\draw[very thick] (1.5,-2) to (1.5,2);
	\draw[very thick] (-.5,-2) node[below,yshift=2pt]{$k$} to (-.5,-1.5);
	\draw[very thick] (-.5,-1.5) to (-1.5,0);
	\node at (-.5,-.5) {$\mydots$};
	\draw[very thick] (-.5,-1.5) to (.5,0);
	\draw[very thick,fill=white] (-2,0) rectangle (2,1.5);
	\node at (0,.75) {\footnotesize$\leq k{-}1$};
	\node at (0,1.75) {\scriptsize$\cdots$};
\end{tikzpicture}
=
\begin{tikzpicture}[scale=.35, tinynodes, anchorbase]
	\draw[very thick] (-1,2) to (0,0);
	\node at (0,1) {$\mydots$};
	\draw[very thick] (1,2) to (0,0);
	\draw[very thick] (0,-1) to node[right,xshift=-2pt]{$k{+}1$} (0,0);
	\draw[very thick] (-.5,-2) node[below,yshift=2pt]{$k$} to (0,-1);
	\draw[very thick] (.5,-2) to (0,-1);
\end{tikzpicture}
+
\begin{tikzpicture}[scale=.35, tinynodes, anchorbase]
	\draw[very thick] (-1.5,1.5) to (-1.5,2);
	\draw[very thick] (1.5,-2) to (1.5,2);
	\draw[very thick] (-.5,-2) node[below,yshift=2pt]{$k$} to (-.5,-1.5);
	\draw[very thick] (-.5,-1.5) to (-1.5,0);
	\node at (-.5,-.5) {$\mydots$};
	\draw[very thick] (-.5,-1.5) to (.5,0);
	\draw[very thick,fill=white] (-2,0) rectangle (2,1.5);
	\node at (0,.75) {\footnotesize$\leq k{-}1$};
	\node at (0,1.75) {\scriptsize$\cdots$};
\end{tikzpicture}
\end{aligned}
\end{equation}
In passing from the first line to the second, we use a rotation of (\text{\ref{eq:spn}e}), as in \eqref{eq:reduction}, to replace the diagram with the triangle.
We use this version of (\text{\ref{eq:spn}e}) below as well.

Next, $P(1)$ again holds by the definition of the quadrivalent vertex. 
Thus, we assume $P=P(\ell)$ for $1 \leq \ell \leq k-1$. 
Note that this implies the relation
\begin{equation}\label{eq:Pprime}
P'(\ell) :
\begin{tikzpicture}[scale=.35, tinynodes, anchorbase]
	\draw[very thick] (-1,2) to (0,.5);
	\node at (0,1.5) {$\mydots$};
	\draw[very thick] (1,2) to (0,.5);
	\draw[very thick] (0,-.5) to node[right,xshift=-2pt]{$\ell{+}1$} (0,.5);
	\draw[very thick] (-1,-2) to (0,-.5);
	\node at (0,-1.5) {$\mydots$};
	\draw[very thick] (1,-2) to (0,-.5);
\end{tikzpicture}
=
\HT_{\ell+1}
-
\begin{tikzpicture}[scale=.35, smallnodes, anchorbase]
	\draw[very thick] (-1,-2) to (-1,2);
	\draw[very thick] (1,-2) to (1,2);
	\draw[very thick,fill=white] (-3,-1) rectangle (3,1);
	\node at (0,0) {$\leq \frac{1}{2}\ell(\ell{+}1){-}1$};
	\node at (0,-1.5) {\scriptsize$\cdots$};
	\node at (0,1.5) {\scriptsize$\cdots$};
\end{tikzpicture}
=
\begin{tikzpicture}[scale=.35, smallnodes, anchorbase]
	\draw[very thick] (-1,-2) to (-1,2);
	\draw[very thick] (1,-2) to (1,2);
	\draw[very thick,fill=white] (-3,-1) rectangle (3,1);
	\node at (0,0) {$\leq \frac{1}{2}\ell(\ell{+}1)$};
	\node at (0,-1.5) {\scriptsize$\cdots$};
	\node at (0,1.5) {\scriptsize$\cdots$};
\end{tikzpicture}
\end{equation}
for $1 \leq \ell \leq k-1$.
Using this, and assuming $k>1$, we now compute:
\begingroup
\allowdisplaybreaks[4]
\begin{align}
\begin{tikzpicture}[scale=.35, anchorbase]
	\draw[very thick] (-1,-2) to (-1,2);
	\draw[very thick] (1,-2) to (1,2);
	\draw[very thick,fill=white] (-2,-1) rectangle (2,1);
	\node at (0,0) {$\HT_{k+1}$};
	\node at (0,-1.5) {\scriptsize$\cdots$};
	\node at (0,1.5) {\scriptsize$\cdots$};
\end{tikzpicture}
\ &\stackrel{\eqref{eq:HTdef}}{=} \
\begin{tikzpicture}[scale=.35, anchorbase]
	\draw[very thick] (-1,-2) to (-1,1) to [out=90,in=270] (1,4);
	\draw[very thick] (1,-2) to (1,1) to [out=90,in=270] (3,4);
	\draw[very thick,fill=white] (-2,-1) rectangle (2,1);
	\node at (0,0) {$\HT_{k}$};
	\node at (0,-1.5) {\scriptsize$\cdots$};
	\node at (0,1.25) {\scriptsize$\cdots$};
	\draw[very thick] (3,-2) to (3,1) to [out=90,in=270] (-1,4);
\end{tikzpicture}
\ \stackrel{P}{=} \
\begin{tikzpicture}[scale=.35, anchorbase,tinynodes]
	\draw[very thick] (-1,-2) to (0,-1);
	\draw[very thick] (1,-2) to (0,-1);
	\draw[very thick] (0,-1) to node[right,xshift=-2pt]{$k$} (0,0);
	\draw[very thick] (0,0) to [out=150,in=270] (-1,1.5) to [out=90,in=270] (1,4);
	\draw[very thick] (0,0) to [out=30,in=270] (1,1.5) to [out=90,in=270] (3,4);
	\node at (0,-1.75) {\scriptsize$\mydots$};
	\node at (0,1.25) {\scriptsize$\cdots$};
	\draw[very thick] (3,-2) to (3,1) to [out=90,in=270] (-1,4);
\end{tikzpicture}
+
\begin{tikzpicture}[scale=.35, anchorbase,tinynodes]
	\draw[very thick] (-1,-2) to (-1,1) to [out=90,in=270] (1,4);
	\draw[very thick] (1,-2) to (1,1) to [out=90,in=270] (3,4);
	\draw[very thick,fill=white] (-2.5,-1) rectangle (2.5,1);
	\node at (0,0) {$\leq \frac{1}{2}k(k{-}1){-}1$};
	\node at (0,-1.5) {\scriptsize$\cdots$};
	\node at (0,1.25) {\scriptsize$\cdots$};
	\draw[very thick] (3,-2) to (3,1) to [out=90,in=270] (-1,4);
\end{tikzpicture}
\ \stackrel{Q}{=} \
\begin{tikzpicture}[scale=.35, tinynodes, anchorbase]
	\draw[very thick] (-1,1.5) to (0,0);
	\node at (0,1) {$\mydots$};
	\draw[very thick] (1,1.5) to (0,0);
	\draw[very thick,directed=.6] (0,-1) to node[left,xshift=2pt]{$k$} (0,0);
	\draw[very thick] (0,-1) [out=330,in=90] to (2,-4.5);
	\draw[very thick] (0,-2) to [out=150,in=210] (0,-1);
	\draw[very thick] (0,-2) to [out=30,in=270] (2,1.5);
	\draw[very thick,rdirected=.7] (0,-3) to node[left,xshift=2pt]{$k$} (0,-2);
	\draw[very thick] (-1,-4.5) to (0,-3);
	\node at (0,-4) {$\mydots$};
	\draw[very thick] (1,-4.5) to (0,-3);
\end{tikzpicture}
+
\begin{tikzpicture}[scale=.35, tinynodes, anchorbase]
	\draw[very thick] (-1.5,1) to (-1.5,1.5);
	\node at (-.5,1.25) {$\cdots$};
	\draw[very thick] (.5,1) to (.5,1.5);
	\draw[very thick,fill=white] (-2,1) rectangle (1,0);
	\node at (-.5,.5) {$\leq k{-}2$};
	\draw[very thick] (-1.5,0) to (-1,-1);
	\draw[very thick] (-.5,0) to (-1,-1);
	\node at (-1,-.25) {$\mydots$};
	\draw[very thick] (.5,0) to [out=270,in=90] (2,-4.5);
	\draw[very thick] (0,-2) to [out=150,in=270] (-1,-1);
	\draw[very thick] (0,-2) to [out=30,in=270] (2,1.5);
	\draw[very thick,rdirected=.7] (0,-3) to node[left,xshift=2pt]{$k$} (0,-2);
	\draw[very thick] (-1,-4.5) to (0,-3);
	\node at (0,-4) {$\mydots$};
	\draw[very thick] (1,-4.5) to (0,-3);
\end{tikzpicture}
+
\begin{tikzpicture}[scale=.35, smallnodes, anchorbase]
	\draw[very thick] (-1,-2) to (-1,2);
	\draw[very thick] (1,-2) to (1,2);
	\draw[very thick,fill=white] (-3,-1) rectangle (3,1);
	\node at (0,0) {$\leq \frac{1}{2}k(k{+}1){-}1$};
	\node at (0,-1.5) {\scriptsize$\cdots$};
	\node at (0,1.5) {\scriptsize$\cdots$};
\end{tikzpicture} \nonumber \\
\stackrel{P'}{=}&
\begin{tikzpicture}[scale=.35, tinynodes, anchorbase]
	\draw[very thick] (-1,1.5) to (0,0);
	\node at (0,1) {$\mydots$};
	\draw[very thick] (1,1.5) to (0,0);
	\draw[very thick,directed=.6] (0,-1) to node[left,xshift=2pt]{$k$} (0,0);
	\draw[very thick] (0,-1) [out=330,in=90] to (2,-4.5);
	\draw[very thick] (0,-2) to [out=150,in=210] (0,-1);
	\draw[very thick] (0,-2) to [out=30,in=270] (2,1.5);
	\draw[very thick,rdirected=.7] (0,-3) to node[left,xshift=2pt]{$k$} (0,-2);
	\draw[very thick] (-1,-4.5) to (0,-3);
	\node at (0,-4) {$\mydots$};
	\draw[very thick] (1,-4.5) to (0,-3);
\end{tikzpicture}
+
\begin{tikzpicture}[scale=.35, anchorbase,tinynodes]
	\draw[very thick] (-1,-2) to (-1,4);
	\draw[very thick] (1,-2) to (1,4);
	\draw[very thick,fill=white] (-2.25,1.75) rectangle (2.25,3.25);
	\node at (0,2.5) {$\leq k{-}2$};
	\draw[very thick,fill=white] (-2.25,-1.25) rectangle (2.25,.25);
	\node at (0,-.5) {$\leq \frac{1}{2}k(k{-}1)$};
	\node at (0,1) {\scriptsize$\cdots$};
	\draw[very thick] (3,4) to (3,2) to [out=270,in=90] (1.75,.25);
	\draw[very thick] (3,-2) to (3,0) to [out=90,in=270] (1.75,1.75);
\end{tikzpicture}
+
\begin{tikzpicture}[scale=.35, smallnodes, anchorbase]
	\draw[very thick] (-1,-2) to (-1,2);
	\draw[very thick] (1,-2) to (1,2);
	\draw[very thick,fill=white] (-3,-1) rectangle (3,1);
	\node at (0,0) {$\leq \frac{1}{2}k(k{+}1){-}1$};
	\node at (0,-1.5) {\scriptsize$\cdots$};
	\node at (0,1.5) {\scriptsize$\cdots$};
\end{tikzpicture} 
\stackrel{\eqref{eq:quad}}{=}
\begin{tikzpicture}[scale=.35, tinynodes, anchorbase]
	\draw[very thick] (-1,1.5) to (0,0);
	\node at (0,1) {$\mydots$};
	\draw[very thick] (1,1.5) to (0,0);
	\draw[very thick,directed=.6] (0,-1) to node[left,xshift=2pt]{$k$} (0,0);
	\draw[very thick] (0,-2) to [out=150,in=210] (0,-1);
	\draw[very thick] (0,-1) to [out=330,in=120] (.75,-1.5); 
	\draw[very thick,directed=.5,rdirected=.8] (.75,-1.5) to node[below,yshift=1pt]{$2$} (1.75,-1.5);
	\draw[very thick] (0,-2) to [out=30,in=240] (.75,-1.5); 
	\draw[very thick] (1.75,-1.5) to [out=60,in=270] (2,1.5);
	\draw[very thick] (1.75,-1.5) [out=300,in=90] to (2,-4.5);	
	\draw[very thick,rdirected=.7] (0,-3) to node[left,xshift=2pt]{$k$} (0,-2);
	\draw[very thick] (-1,-4.5) to (0,-3);
	\node at (0,-4) {$\mydots$};
	\draw[very thick] (1,-4.5) to (0,-3);
\end{tikzpicture}
+
\frac{[k][n{-}1]}{[n]}
\begin{tikzpicture}[scale=.35, tinynodes, anchorbase]
	\draw[very thick] (-1,2) to (0,.5);
	\node at (0,1.5) {$\mydots$};
	\draw[very thick] (1,2) to (0,.5);
	\draw[very thick] (0,-.5) to node[right,xshift=-2pt]{$k$} (0,.5);
	\draw[very thick] (-1,-2) to (0,-.5);
	\node at (0,-1.5) {$\mydots$};
	\draw[very thick] (1,-2) to (0,-.5);
	\draw[very thick] (2,-2) to (2,2);
\end{tikzpicture}
+
\begin{tikzpicture}[scale=.35, smallnodes, anchorbase]
	\draw[very thick] (-1,-2) to (-1,2);
	\draw[very thick] (1,-2) to (1,2);
	\draw[very thick,fill=white] (-3,-1) rectangle (3,1);
	\node at (0,0) {$\leq \frac{1}{2}k(k{+}1){-}1$};
	\node at (0,-1.5) {\scriptsize$\cdots$};
	\node at (0,1.5) {\scriptsize$\cdots$};
\end{tikzpicture} \nonumber \\
\stackrel{P'}{=}&
\begin{tikzpicture}[scale=.35, tinynodes, anchorbase]
	\draw[very thick] (-1,1.5) to (0,0);
	\node at (0,1) {$\mydots$};
	\draw[very thick] (1,1.5) to (0,0);
	\draw[very thick,directed=.6] (0,-1) to node[left,xshift=2pt]{$k$} (0,0);
	\draw[very thick] (0,-2) to [out=150,in=210] (0,-1);
	\draw[very thick] (0,-1) to [out=330,in=120] (.75,-1.5); 
	\draw[very thick,directed=.5,rdirected=.8] (.75,-1.5) to node[below,yshift=1pt]{$2$} (1.75,-1.5);
	\draw[very thick] (0,-2) to [out=30,in=240] (.75,-1.5); 	
	\draw[very thick] (1.75,-1.5) to [out=60,in=270] (2,1.5);
	\draw[very thick] (1.75,-1.5) [out=300,in=90] to (2,-4.5);	
	\draw[very thick,rdirected=.7] (0,-3) to node[left,xshift=2pt]{$k$} (0,-2);
	\draw[very thick] (-1,-4.5) to (0,-3);
	\node at (0,-4) {$\mydots$};
	\draw[very thick] (1,-4.5) to (0,-3);
\end{tikzpicture}
+
\begin{tikzpicture}[scale=.35, smallnodes, anchorbase]
	\draw[very thick] (-1,-2) to (-1,2);
	\draw[very thick] (1,-2) to (1,2);
	\draw[very thick,fill=white] (-3,-1) rectangle (3,1);
	\node at (0,0) {$\leq \frac{1}{2}k(k{+}1){-}1$};
	\node at (0,-1.5) {\scriptsize$\cdots$};
	\node at (0,1.5) {\scriptsize$\cdots$};
\end{tikzpicture}
\stackrel{(\text{\ref{eq:spn}e})}{=}
\begin{tikzpicture}[scale=.35, tinynodes, anchorbase]
	\draw[very thick] (-1,2) to (0,.5);
	\node at (0,1.5) {$\mydots$};
	\draw[very thick] (1,2) to (0,.5);
	\draw[very thick] (0,-.5) to node[right,xshift=-2pt]{$k{+}1$} (0,.5);
	\draw[very thick] (-1,-2) to (0,-.5);
	\node at (0,-1.5) {$\mydots$};
	\draw[very thick] (1,-2) to (0,-.5);
\end{tikzpicture}
+
\frac{[n{-}k]}{[n{-}k{+}1]}
\begin{tikzpicture}[scale=.35, tinynodes, anchorbase]
	\draw[very thick] (-1,2) to (0,0) node[left,yshift=3pt,xshift=1pt]{$k$};
	\draw[very thick] (.25,2) to (-.375,.75);
	\node at (-.375,1.75) {$\mydots$};
	\draw[very thick] (1,2) to (0,0);
	\draw[very thick] (0,-1) to node[right,xshift=-2pt]{$k{-}1$} (0,0);
	\draw[very thick] (0,-1) to (1,-3);
	\draw[very thick] (-1,-3) to (0,-1) node[left,yshift=-3pt,xshift=1pt]{$k$};
	\draw[very thick] (.25,-3) to (-.375,-1.75);
	\node at (-.375,-2.75) {$\mydots$};
\end{tikzpicture}
-
\frac{[k][n{-}k]}{[n]}
\begin{tikzpicture}[scale=.35, tinynodes, anchorbase]
	\draw[very thick] (-1,2) to (0,.5);
	\node at (0,1.5) {$\mydots$};
	\draw[very thick] (1,2) to (0,.5);
	\draw[very thick] (0,-.5) to node[right,xshift=-2pt]{$k$} (0,.5);
	\draw[very thick] (-1,-2) to (0,-.5);
	\node at (0,-1.5) {$\mydots$};
	\draw[very thick] (1,-2) to (0,-.5);
	\draw[very thick] (2,-2) to (2,2);
\end{tikzpicture}
+
\begin{tikzpicture}[scale=.35, smallnodes, anchorbase]
	\draw[very thick] (-1,-2) to (-1,2);
	\draw[very thick] (1,-2) to (1,2);
	\draw[very thick,fill=white] (-3,-1) rectangle (3,1);
	\node at (0,0) {$\leq \frac{1}{2}k(k{+}1){-}1$};
	\node at (0,-1.5) {\scriptsize$\cdots$};
	\node at (0,1.5) {\scriptsize$\cdots$};
\end{tikzpicture} \\
\stackrel{P'}{=}&
\begin{tikzpicture}[scale=.35, tinynodes, anchorbase]
	\draw[very thick] (-1,2) to (0,.5);
	\node at (0,1.5) {$\mydots$};
	\draw[very thick] (1,2) to (0,.5);
	\draw[very thick] (0,-.5) to node[right,xshift=-2pt]{$k{+}1$} (0,.5);
	\draw[very thick] (-1,-2) to (0,-.5);
	\node at (0,-1.5) {$\mydots$};
	\draw[very thick] (1,-2) to (0,-.5);
\end{tikzpicture}
+
\frac{[n{-}k]}{[n{-}k{+}1]}
\begin{tikzpicture}[scale=.35, tinynodes, anchorbase]
	\draw[very thick] (-1,2) to (0,0) node[left,yshift=3pt,xshift=1pt]{$k$};
	\draw[very thick] (.25,2) to (-.375,.75);
	\node at (-.375,1.75) {$\mydots$};
	\draw[very thick] (1,2) to (0,0);
	\draw[very thick] (0,-1) to node[right,xshift=-2pt]{$k{-}1$} (0,0);
	\draw[very thick] (0,-1) to (1,-3);
	\draw[very thick] (-1,-3) to (0,-1) node[left,yshift=-3pt,xshift=1pt]{$k$};
	\draw[very thick] (.25,-3) to (-.375,-1.75);
	\node at (-.375,-2.75) {$\mydots$};
\end{tikzpicture}
+
\begin{tikzpicture}[scale=.35, smallnodes, anchorbase]
	\draw[very thick] (-1,-2) to (-1,2);
	\draw[very thick] (1,-2) to (1,2);
	\draw[very thick,fill=white] (-3,-1) rectangle (3,1);
	\node at (0,0) {$\leq \frac{1}{2}k(k{+}1){-}1$};
	\node at (0,-1.5) {\scriptsize$\cdots$};
	\node at (0,1.5) {\scriptsize$\cdots$};
\end{tikzpicture}
\ \stackrel{Q'}{=}
\begin{tikzpicture}[scale=.35, tinynodes, anchorbase]
	\draw[very thick] (-1,2) to (0,.5);
	\node at (0,1.5) {$\mydots$};
	\draw[very thick] (1,2) to (0,.5);
	\draw[very thick] (0,-.5) to node[right,xshift=-2pt]{$k{+}1$} (0,.5);
	\draw[very thick] (-1,-2) to (0,-.5);
	\node at (0,-1.5) {$\mydots$};
	\draw[very thick] (1,-2) to (0,-.5);
\end{tikzpicture}
+
\begin{tikzpicture}[scale=.35, anchorbase,tinynodes]
	\draw[very thick] (-1,-2) to (-1,-1.25);
	\draw[very thick] (1,-2) to (1,-1.25);
	\draw[very thick] (-1,4) to (-1,3.25);
	\draw[very thick] (1,4) to (1,3.25);
	\draw[very thick,fill=white] (-2.25,2.5) rectangle (2.25,3.5);
	\node at (0,3) {$\leq k{-}1$};
	\draw[very thick] (0,1.5) to (-.75,2.5);
	\node at (0,2.25) {$\mydots$};
	\draw[very thick] (0,1.5) to (.75,2.5);	
	\draw[very thick] (0,.5) to node[right,xshift=-2pt]{$k{-}1$} (0,1.5);
	\draw[very thick] (0,.5) to (-.75,-.5);
	\node at (0,-.25) {$\mydots$};
	\draw[very thick] (0,.5) to (.75,-.5);	
	\draw[very thick,fill=white] (-2.25,-1.5) rectangle (2.25,-.5);
	\node at (0,-1) {$\leq k{-}1$};
	\draw[very thick] (3,4) to (3,2.5) to [out=270,in=0] (2.375,1.5) to [out=180,in=270] (1.75,2.5);
	\draw[very thick] (3,-2) to (3,-.5) to [out=90,in=0] (2.375,.5) to [out=180,in=90] (1.75,-.5);
\end{tikzpicture}
+
\begin{tikzpicture}[scale=.35, smallnodes, anchorbase]
	\draw[very thick] (-1,-2) to (-1,2);
	\draw[very thick] (1,-2) to (1,2);
	\draw[very thick,fill=white] (-3,-1) rectangle (3,1);
	\node at (0,0) {$\leq \frac{1}{2}k(k{+}1){-}1$};
	\node at (0,-1.5) {\scriptsize$\cdots$};
	\node at (0,1.5) {\scriptsize$\cdots$};
\end{tikzpicture} \nonumber \\
\stackrel{P'}{=}&
\begin{tikzpicture}[scale=.35, tinynodes, anchorbase]
	\draw[very thick] (-1,2) to (0,.5);
	\node at (0,1.5) {$\mydots$};
	\draw[very thick] (1,2) to (0,.5);
	\draw[very thick] (0,-.5) to node[right,xshift=-2pt]{$k{+}1$} (0,.5);
	\draw[very thick] (-1,-2) to (0,-.5);
	\node at (0,-1.5) {$\mydots$};
	\draw[very thick] (1,-2) to (0,-.5);
\end{tikzpicture}
+
\begin{tikzpicture}[scale=.35, anchorbase,tinynodes]
	\draw[very thick] (-1,-2) to (-1,4);
	\draw[very thick] (1,-2) to (1,4);
	\draw[very thick,fill=white] (-2.25,2.5) rectangle (2.25,3.5);
	\node at (0,3) {$\leq k{-}1$};
	\node at (0,2) {$\mydots$};
	\draw[very thick,fill=white] (-3,1.625) rectangle (3,.375);
	\node at (0,1) {$\leq \frac{1}{2}(k{-}2)(k{-}1)$};
	\node at (0,0) {$\mydots$};
	\draw[very thick,fill=white] (-2.25,-1.5) rectangle (2.25,-.5);
	\node at (0,-1) {$\leq k{-}1$};
	\draw[very thick] (3,4) to (3,2.5) to [out=270,in=0] (2.375,2) to [out=180,in=270] (1.75,2.5);
	\draw[very thick] (3,-2) to (3,-.5) to [out=90,in=0] (2.375,0) to [out=180,in=90] (1.75,-.5);
\end{tikzpicture}
+
\begin{tikzpicture}[scale=.35, smallnodes, anchorbase]
	\draw[very thick] (-1,-2) to (-1,2);
	\draw[very thick] (1,-2) to (1,2);
	\draw[very thick,fill=white] (-3,-1) rectangle (3,1);
	\node at (0,0) {$\leq \frac{1}{2}k(k{+}1){-}1$};
	\node at (0,-1.5) {\scriptsize$\cdots$};
	\node at (0,1.5) {\scriptsize$\cdots$};
\end{tikzpicture}
\ = \
\begin{tikzpicture}[scale=.35, tinynodes, anchorbase]
	\draw[very thick] (-1,2) to (0,.5);
	\node at (0,1.5) {$\mydots$};
	\draw[very thick] (1,2) to (0,.5);
	\draw[very thick] (0,-.5) to node[right,xshift=-2pt]{$k{+}1$} (0,.5);
	\draw[very thick] (-1,-2) to (0,-.5);
	\node at (0,-1.5) {$\mydots$};
	\draw[very thick] (1,-2) to (0,-.5);
\end{tikzpicture}
+
\begin{tikzpicture}[scale=.35, smallnodes, anchorbase]
	\draw[very thick] (-1,-2) to (-1,2);
	\draw[very thick] (1,-2) to (1,2);
	\draw[very thick,fill=white] (-3,-1) rectangle (3,1);
	\node at (0,0) {$\leq \frac{1}{2}k(k{+}1){-}1$};
	\node at (0,-1.5) {\scriptsize$\cdots$};
	\node at (0,1.5) {\scriptsize$\cdots$};
\end{tikzpicture} \nonumber
\end{align}
\endgroup
\end{proof}

An explicit instance of \eqref{eq:HT} is given below in equation \eqref{eq:preciseR3}.
Before proceeding, we now record our first consequence of Lemma \ref{lem:HT}.

\begin{cor}\label{cor:QuadFull}
$\Web^{\times}(\spn)$ is a full subcategory of $\Web(\spn)$.
Equivalently, $\Web^{\times}(\spn) = \SWeb(\spn)$.
\end{cor}

\begin{proof}
We must show that any $\spn$ web in $\SWeb(\spn)$ can be expressed as a linear combination of quadrivalent graphs.
Recall from Corollary \ref{cor:reduction} that
\begin{equation}
\begin{tikzpicture}[scale=.35, tinynodes, anchorbase]
	\draw[very thick] (0,-2) node[below,yshift=2pt]{$k$} to (0,2);
\end{tikzpicture}
=
\frac{1}{[k]!}
\begin{tikzpicture}[scale=.35, tinynodes, anchorbase]
	\draw[very thick] (0,-2) node[below,yshift=2pt]{$k$} to (0,-1);
	\draw[very thick] (0,-1) to [out=150,in=210] node[left,xshift=2pt]{$1$} (0,1);
	\node at (0,0) {$\mydots$};	
	\draw[very thick] (0,-1) to [out=30,in=330] node[right,xshift=-2pt]{$1$} (0,1);
	\draw[very thick] (0,2) node[above,yshift=-2pt]{$k$} to (0,1);
\end{tikzpicture}.
\end{equation}
We now apply the following transformation:
\begin{equation} \label{eq:applymetotri} 
\begin{tikzpicture}[scale =.75, smallnodes,anchorbase]
	\draw[very thick] (0,0) node[below=-1pt]{$1$} to [out=90,in=210] (.5,.75);
	\draw[very thick] (1,0) node[below=-1pt]{$k$} to [out=90,in=330] (.5,.75);
	\draw[very thick] (.5,.75) to (.5,1.5) node[above=-2pt]{$k{+}1$};
\end{tikzpicture}
=
\frac{1}{[k+1]![k]!}
\begin{tikzpicture}[scale=.25,anchorbase,tinynodes,xscale=-1.5]
	\draw [very thick] (0,-2) node[below=-2pt]{$k$} to (0,-1);
	\draw [very thick] (0,-1) to [out=150,in=210] node[right=-3pt]{$1$} (0,.75);
	\node at (0,-.125) {$\mydots$};
	\draw [very thick] (0,-1) to [out=30,in=330] node[left=-3pt]{$1$} (0,.75);
	\draw [very thick] (3,-2) node[below=-2pt]{$1$} to [out=90,in=330] (1,2.5);
	\draw [very thick] (0,.75) to [out=90,in=210] (1,2.5);
	\draw [very thick] (1,2.5) to (1,3); 
	\draw [very thick] (1,3) to [out=150,in=210] node[right=-3pt]{$1$} (1,4.75);
	\node at (1,3.875) {$\mydots$};
	\draw [very thick] (1,3) to [out=30,in=330] node[left=-3pt]{$1$} (1,4.75);	
	\draw [very thick] (1,4.75) to (1,5.5) node[above=-2pt]{$k{+}1$};
\end{tikzpicture}
=
\frac{1}{[k+1]![k]!}
\begin{tikzpicture}[scale=.25, tinynodes, anchorbase,yscale=-1]
	\draw[very thick] (0,-2) node[above=-2pt]{$k{+}1$} to (0,-1);
	\draw[very thick] (0,-1) to (1,1) node[below=-2pt]{$1$};
	\draw[very thick] (0,-1) to (-1,1) node[below=-2pt]{$1$};
	\node at (0,.5) {$\mydots$};
\end{tikzpicture}
\circ
\begin{tikzpicture}[scale=.25, tinynodes, anchorbase]
	\draw[very thick] (-1,2) to (0,.5);
	\node at (0,1.5) {$\mydots$};
	\draw[very thick] (1,2) to (0,.5);
	\draw[very thick] (0,-.5) to node[right,xshift=-2pt]{$k{+}1$} (0,.5);
	\draw[very thick] (-1,-2) to (0,-.5);
	\node at (0,-1.5) {$\mydots$};
	\draw[very thick] (1,-2) to (0,-.5);
\end{tikzpicture}
\circ
\begin{tikzpicture}[scale=.25, tinynodes, anchorbase]
	\draw[very thick] (-2,-2) node[below=-2pt]{$1$} to (-2,1);
	\draw[very thick] (0,-2) node[below=-2pt]{$k$} to (0,-1);
	\draw[very thick] (0,-1) to (1,1) node[above=-2pt]{$1$};
	\draw[very thick] (0,-1) to (-1,1) node[above=-2pt]{$1$};
	\node at (0,.5) {$\mydots$};
\end{tikzpicture}
\end{equation}
to each $1 \ot k \to k+1$ trivalent vertex, and the analogous transformation to $k \ot 1 \to k+1$ trivalent vertices.
The second equality of \eqref{eq:applymetotri} holds from the definition of the full split vertex \eqref{eq:fullsplit}.

The result is a planar graph all of whose vertices are full split vertices, rather than trivalent vertices. 
Each edge labeled $k > 1$ must meet two full split vertices, because in $\SWeb(\spn)$ no such edges meet the boundary. 
Now, we can use Lemma \ref{lem:HT} in the form \eqref{eq:Pprime}
to replace a neighborhood of the $k$-labeled edge with a linear combination of diagrams built entirely from quadrivalent vertices. 
Applying this transformation to each $k$-labeled edge, the resulting diagram lies in the image of $\Web^{\times}(\spn)$. 
\end{proof}

\subsection{Reduced graphs span the quadrivalent category}\label{sec:RGS}

Our next goal is to use topological arguments to reduce the number of quadrivalent graphs we need to consider to a manageable
spanning set. Let us introduce the terminology needed to discuss this spanning set. For this section, all graphs we consider will be quadrivalent graphs (with boundary) embedded 
in the closed unit disk $\D \subset \R^2$, so we sometimes refer to them merely as \emph{graphs}, and to quadrivalent vertices merely as \emph{crossings}. 
We identify two graphs if they are isotopic relative to the boundary.

\begin{rem} Note that we can identify graphs in the disk $\mathbb{D}$ with those in a planar strip by choosing two points in $\partial \mathbb{D}$ that are disjoint from the graph.
This provides an explicit identification between graphs in $\mathbb{D}$ and morphisms in $\Web^\times(\spn)$; however, for most of our arguments, the choice of these two points is irrelevant. 
As such, for the duration we tacitly identify all $\Hom$-spaces in $\Web^\times(\spn)$ with the corresponding space of quadrivalent graphs in the disk.
\end{rem}

Quadrivalent graphs (with boundary) can be viewed as generically immersed $1$-manifolds with boundary,
i.e. as the image of a generic immersion
\[
\coprod_{j=1}^c S^1\sqcup \coprod_{i=1}^\ell [0,1] \to \D
\]
that sends the boundary of the domain to $\partial \D$.
In this interpretation, a crossing is the (transverse) intersection of the $1$-manifold with itself. Conversely, any sufficiently
generic\footnote{An immersed $1$-manifold is sufficiently generic if all intersections are transverse, 
there are no triple (or higher) point intersections, and no intersections of boundary points. 
For drawing purposes, one should also assume that the $1$-manifold meets the boundary transversely.} immersed $1$-manifold is a quadrivalent graph. 
Said another way, quadrivalent graphs are simply tangle diagrams projected to the plane, where the embedding into the thickened plane is squashed to an immersion in the plane.

Using this topological interpretation of quadrivalent graphs, we can discuss the \emph{strands} in a graph, 
which are the image of the connected components ($S^1$ or $[0,1]$) of the 1-manifold. 

\begin{defn} Given a set of size $2 \ell$, a \emph{matching} is a partition of the set into $\ell$ sets of size $2$, the \emph{pairs} in the matching. Given a quadrivalent graph
with $2 \ell$ boundary points, the equivalence relation of being in the same strand will partition these boundary points, giving a matching on the set of boundary points. The graph
is said to \emph{correspond} to the matching of boundary points, and vice versa. \end{defn}

Now we introduce some terminology from the theory of the symmetric group.

\begin{defn} Consider a planar disk with $2\ell$ boundary points and a matching on the set of boundary points. 
Let $x_1 = \{a_1, b_1\}$ and $x_2 = \{a_2, b_2\}$ be two disjoint pairs occurring in the matching.  
We say that the pair of pairs $\{x_1, x_2\}$ is an \emph{inversion} if, reading counterclockwise around the boundary, we alternate between points in $x_1$ and points in $x_2$. 
For example, one might get the word $a_1 a_2 b_1 b_2$ or $a_1 b_2 b_1 a_2$. The \emph{length} of a matching is the number of inversions.

Let $\Gamma$ be a quadrivalent graph corresponding to this matching. Each non-closed strand corresponds to a pair of boundary points. 
If $C_1$ and $C_2$ are non-closed strands corresponding to pairs $x_1$ and $x_2$, we say that $\{C_1, C_2\}$ is an inversion if $\{x_1, x_2\}$ is. 
\end{defn}

Topologically, the Jordan Curve Theorem (JCT) implies that a pair of pairs $\{x_1, x_2\}$ is an inversion if the corresponding pair of strands $\{C_1, C_2\}$ must intersect (an odd number of times) 
for any graph corresponding to this matching. Strands which do not form an inversion may intersect (an even number of times), but their topological intersection number is zero, and there is
some graph representing the same matching where the strands do not cross. 

\begin{example}\label{ex:inversion}
In the graph:
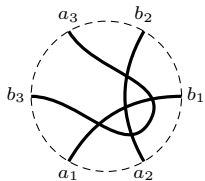
\[
\begin{tikzpicture}[scale=1, smallnodes, anchorbase]
	\draw[densely dashed] (0,0) circle (1);
	\draw[very thick] (1,0) node[right,xshift=-2pt]{$b_1$} to [out=180,in=60] (-.5,-.866) node[below,yshift=1pt]{$a_1$};
	\draw[very thick] (.5,.866) node[above]{$b_2$} to [out=240,in=120] (.5,-.866) node[below,yshift=1pt]{$a_2$};
	\draw[very thick] (-.5,.866) node[above]{$a_3$} to [out=300,in=60] (.577,-.333) to [out=240,in=0] (-1,0) node[left,xshift=2pt]{$b_3$};
\end{tikzpicture}
\]
the strands $\{C_1,C_2\}$ are an inversion, while $\{C_1,C_3\}$ and $\{C_2,C_3\}$ are not. Here, $C_i$ corresponds to the boundary points $\{a_i,b_i\}$.
\end{example}

\begin{defn} A quadrivalent graph $\Gamma$ is said to be \emph{crossing-minimal} if no strands cross themselves, and no two distinct strands intersect more than once. 
A graph $\Gamma$ is said to be \emph{reduced} if it is crossing-minimal and has no closed strands. \end{defn}

The number of crossings in a crossing-minimal graph is equal to the number of inversions between strands, which is the length of the corresponding matching. Any graph corresponding
to this matching has at least one crossing for each inversion. Thus a crossing-minimal graph does, indeed, have the minimal number of crossings required to represent this matching.
Reduced graphs are the graphs corresponding to a given matching with the least crossings and the least strands.

It is easy to argue that some reduced graph corresponds to each matching of the boundary points. One can connect the paired points by straight lines (chords), 
and then perturb to make the immersion generic (removing higher-order intersections). 
Since chords in a disk cross at most once and are never tangent, a sufficiently small perturbation will never produce a diagram which is not reduced.

The following relations, which are certain analogues of the Reidemeister moves in knot theory, will allow us to argue that reduced graphs span $\Web^\times(\spn)$.

\begin{lem} \label{lem:webReidemeister} 
The following relations hold in $\Web^\times(\spn)$:
\begin{subequations} \label{subeq:Reidemeister}
\begin{equation} \label{eq:R1}
\begin{tikzpicture} [scale=.5,anchorbase]
	\draw[very thick] (1,1) to [out=270,in=0] (.75,.75) to [out=180,in=270] (0,2);
	\draw[very thick] (0,0) to [out=90,in=180] (.75,1.25) to [out=0,in=90] (1,1);
\end{tikzpicture}
=
-\frac{[n-1][2n+2]}{[n+1]} \
\begin{tikzpicture} [scale=.5,anchorbase]
	\draw[very thick] (0,0) to (0,2);
\end{tikzpicture}
\end{equation}
\begin{equation} \label{eq:R2}
\begin{tikzpicture} [scale=.35,anchorbase]
	\draw[very thick] (1,0) to [out=90,in=270] (0,1.5) to [out=90,in=270] (1,3);
	\draw[very thick] (0,0) to [out=90,in=270] (1,1.5) to [out=90,in=270] (0,3);
\end{tikzpicture}
=
\begin{tikzpicture}[scale=.35, anchorbase]
	\draw[very thick] (1,0) to (1,3);
	\draw[very thick] (0,0) to (0,3);
	\draw[very thick,fill=white] (-1,.75) rectangle (2,2.25);
	\node at (.5,1.5) {$\leq 1$};
\end{tikzpicture}
\end{equation}
\begin{equation} \label{eq:R3}
\begin{tikzpicture}[scale=.25, tinynodes, anchorbase]
	\draw[very thick] (-2,-3) to [out=90,in=270] (2,3);
	\draw[very thick] (0,-3) to [out=90,in=270] (-2,0) to [out=90,in=270] (0,3);
	\draw[very thick] (2,-3) to [out=90,in=270] (-2,3);	
\end{tikzpicture}
=
\begin{tikzpicture}[scale=.25, tinynodes, anchorbase]
	\draw[very thick] (-2,-3) to [out=90,in=270] (2,3);
	\draw[very thick] (0,-3) to [out=90,in=270] (2,0) to [out=90,in=270] (0,3);
	\draw[very thick] (2,-3) to [out=90,in=270] (-2,3);	
\end{tikzpicture}
+
\begin{tikzpicture}[scale=.25, anchorbase]
	\draw[very thick] (-2,-3) to (-2,3);
	\draw[very thick] (0,-3) to (0,3);
	\draw[very thick] (2,-3) to (2,3);
	\draw[very thick,fill=white] (-3,-1) rectangle (3,1);
	\node at (0,0) {$\leq 2$};
\end{tikzpicture}
\end{equation}
\end{subequations}
More precisely, we have
\begin{subequations}
\begin{equation}\label{eq:preciseR2}
\begin{tikzpicture} [scale=.35,anchorbase]
	\draw[very thick] (1,0) to [out=90,in=270] (0,1.5) to [out=90,in=270] (1,3);
	\draw[very thick] (0,0) to [out=90,in=270] (1,1.5) to [out=90,in=270] (0,3);
\end{tikzpicture}
=
[2]
\begin{tikzpicture} [scale=.35,anchorbase]
	\draw[very thick] (1,0) to [out=90,in=270] (0,2.5);
	\draw[very thick] (0,0) to [out=90,in=270] (1,2.5);
\end{tikzpicture}
-
\frac{[n-1][2n]}{[n]}
\begin{tikzpicture} [scale=.35,anchorbase]
	\draw[very thick] (1,0) to [out=90,in=0] (.5,1) to [out=180,in=90] (0,0);
	\draw[very thick] (1,2.5) to [out=270,in=0] (.5,1.5) to [out=180,in=270] (0,2.5);
\end{tikzpicture}
\end{equation}
\begin{equation}\label{eq:preciseR3}
\begin{tikzpicture}[scale=.225, tinynodes, anchorbase]
	\draw[very thick] (-2,-3) to [out=90,in=270] (2,3);
	\draw[very thick] (0,-3) to [out=90,in=270] (-2,0) to [out=90,in=270] (0,3);
	\draw[very thick] (2,-3) to [out=90,in=270] (-2,3);	
\end{tikzpicture}
=
\begin{tikzpicture}[scale=.15, tinynodes, anchorbase]
	\draw [very thick] (-1,-1) to [out=90,in=210] (0,.75);
	\draw [very thick] (1,-1) to [out=90,in=330] (0,.75);
	\draw [very thick] (3,-1) to [out=90,in=330] (1,2.5);
	\draw [very thick] (1,2.5) to [out=210,in=90] node[left=-2pt]{$2$} (0,.75);
	\draw [very thick] (1,2.5) to node[right=-3pt]{$3$} (1,4.5);
	\draw[very thick] (1,4.5) to [out=30,in=270] (3,8);
	\draw[very thick] (0,6.25) to [out=150,in=270] (-1,8);
	\draw[very thick] (1,4.5) to [out=150,in=270] node[left=-2pt]{$2$} (0,6.25);
	\draw[very thick] (0,6.25) to [out=30,in=270] (1,8);
\end{tikzpicture}
+
\frac{[n-2]}{[n-1]}
\begin{tikzpicture}[scale=.15, tinynodes, anchorbase]
	\draw[very thick] (-2,0) to [out=90,in=180] (0,3) to [out=0,in=90] (2,0);
	\draw[very thick] (-2,9) to [out=270,in=180] (0,6) to [out=0,in=270] (2,9);
	\draw[very thick] (0,0) to (0,9);
\end{tikzpicture}
+
\begin{tikzpicture}[scale=.15, tinynodes, anchorbase,yscale=-1]
	\draw[very thick] (-2,0) to [out=90,in=180] (-1,2) to [out=0,in=90] (0,0);
	\draw[very thick] (-2,9) to [out=270,in=180] (0,6) to [out=0,in=270] (2,9);
	\draw[very thick] (2,0) to [out=90,in=270] (0,9);
\end{tikzpicture}
+
\begin{tikzpicture}[scale=.15, tinynodes, anchorbase]
	\draw[very thick] (-2,0) to [out=90,in=180] (-1,2) to [out=0,in=90] (0,0);
	\draw[very thick] (-2,9) to [out=270,in=180] (0,6) to [out=0,in=270] (2,9);
	\draw[very thick] (2,0) to [out=90,in=270] (0,9);
\end{tikzpicture}
+
\begin{tikzpicture}[scale=.15, tinynodes, anchorbase]
	\draw[very thick] (-2,0) to [out=90,in=270] (0,9);
	\draw[very thick] (0,0) to [out=90,in=270] (-2,9);
	\draw[very thick] (2,0) to [out=90,in=270] (2,9);
\end{tikzpicture}
-
[2n-2]
\begin{tikzpicture}[scale=.15, tinynodes, anchorbase]
	\draw[very thick] (-2,0) to [out=90,in=180] (-1,2) to [out=0,in=90] (0,0);
	\draw[very thick] (-2,9) to [out=270,in=180] (-1,7) to [out=0,in=270] (0,9);
	\draw[very thick] (2,0) to [out=90,in=270] (2,9);
\end{tikzpicture}
\end{equation}
\end{subequations}
\end{lem}

\begin{proof}
Equations \eqref{eq:R1}, \eqref{eq:preciseR2}, and \eqref{eq:preciseR3} are direct computations using Definition \ref{def:quad}.
We leave them to the reader, since \eqref{eq:R1} and \eqref{eq:preciseR2} are straightforward, 
and we do not explicitly use \eqref{eq:preciseR3} in this form.
In its place, we only require \eqref{eq:R3} for our arguments below. For this, Lemma \ref{lem:HT} gives that
\[
\begin{tikzpicture}[scale=.25, tinynodes, anchorbase]
	\draw[very thick] (-2,-3) to [out=90,in=270] (2,3);
	\draw[very thick] (0,-3) to [out=90,in=270] (-2,0) to [out=90,in=270] (0,3);
	\draw[very thick] (2,-3) to [out=90,in=270] (-2,3);	
\end{tikzpicture}
=
\begin{tikzpicture}[scale=.175, tinynodes, anchorbase]
	\draw [very thick] (-1,-1) to [out=90,in=210] (0,.75);
	\draw [very thick] (1,-1) to [out=90,in=330] (0,.75);
	\draw [very thick] (3,-1) to [out=90,in=330] (1,2.5);
	\draw [very thick] (1,2.5) to [out=210,in=90] node[left=-2pt]{$2$} (0,.75);
	\draw [very thick] (1,2.5) to node[right=-3pt]{$3$} (1,4.5);
	\draw[very thick] (1,4.5) to [out=30,in=270] (3,8);
	\draw[very thick] (0,6.25) to [out=150,in=270] (-1,8);
	\draw[very thick] (1,4.5) to [out=150,in=270] node[left=-2pt]{$2$} (0,6.25);
	\draw[very thick] (0,6.25) to [out=30,in=270] (1,8);
\end{tikzpicture}
+
\begin{tikzpicture}[scale=.25, anchorbase]
	\draw[very thick] (-2,-3) to (-2,3);
	\draw[very thick] (0,-3) to (0,3);
	\draw[very thick] (2,-3) to (2,3);
	\draw[very thick,fill=white] (-3,-1) rectangle (3,1);
	\node at (0,0) {$\leq 2$};
\end{tikzpicture}
\stackrel{(\text{\ref{eq:spn}d})}{=}
\begin{tikzpicture}[scale=.175, tinynodes, anchorbase,xscale=-1]
	\draw [very thick] (-1,-1) to [out=90,in=210] (0,.75);
	\draw [very thick] (1,-1) to [out=90,in=330] (0,.75);
	\draw [very thick] (3,-1) to [out=90,in=330] (1,2.5);
	\draw [very thick] (1,2.5) to [out=210,in=90] node[right=-2pt]{$2$} (0,.75);
	\draw [very thick] (1,2.5) to node[left=-3pt]{$3$} (1,4.5);
	\draw[very thick] (1,4.5) to [out=30,in=270] (3,8);
	\draw[very thick] (0,6.25) to [out=150,in=270] (-1,8);
	\draw[very thick] (1,4.5) to [out=150,in=270] node[right=-2pt]{$2$} (0,6.25);
	\draw[very thick] (0,6.25) to [out=30,in=270] (1,8);
\end{tikzpicture}
+
\begin{tikzpicture}[scale=.25, anchorbase]
	\draw[very thick] (-2,-3) to (-2,3);
	\draw[very thick] (0,-3) to (0,3);
	\draw[very thick] (2,-3) to (2,3);
	\draw[very thick,fill=white] (-3,-1) rectangle (3,1);
	\node at (0,0) {$\leq 2$};
\end{tikzpicture} \ .
\]
Rotating this relation by $180^\circ$ then gives the result.
\end{proof}

These relations will combine with several skein-theoretic arguments to give the following.

\begin{thm} \label{thm:reducedgraphsspan} 
For $\ell \ge 0$ let $\mathcal{S}(2\ell)$ be any set of reduced graphs that contains exactly one graph corresponding to each matching of $2 \ell$ points on the boundary of the disk. 
Then each $\Hom$-space in $\Web^{\times}(\spn)$ with $2 \ell$ points on the boundary is spanned by $\mathcal{S}(2\ell)$.
Moreover, if $\Gamma$ is any quadrivalent graph with $N$ crossings, then it is in the span of graphs in $\mathcal{S}(2\ell)$ with $\le N$ crossings. 
\end{thm}
	
The theorem is an immediate consequence of the following three lemmas. 
All three begin with the statement ``Let $\Gamma$ be a quadrivalent graph with $N$ crossings, viewed as a morphism in $\Web^\times(\spn)$.''

\begin{lem} \label{lem:removecircles} If $\Gamma$ is a crossing-minimal graph, then it is equal to an invertible scalar times a reduced graph with $N$ crossings. \end{lem}
	
\begin{lem} \label{lem:toomanycrossings} If $\Gamma$ is not a crossing-minimal graph, then $\Gamma$ is in the span of graphs with $< N$ crossings. \end{lem}

\begin{lem} \label{lem:changerex} If $\Gamma$ is reduced and $\Lambda$ is any other reduced graph corresponding to the same matching, 
then $\Gamma$ and $\Lambda$ are equal in $\Web^\times(\spn)$ modulo the span of graphs with $< N$ crossings. \end{lem}

Let us first prove the theorem given the lemmas.

\begin{proof}[Proof of Theorem \ref{thm:reducedgraphsspan}] Let $\Gamma$ be a quadrivalent graph with $N$ crossings, viewed as a morphism in $\Web^\times(\spn)$. We first prove
that $\Gamma$ is in the span of reduced graphs with $\le N$ crossings, by induction on $N$. If $N$ is zero then $\Gamma$ is necessarily crossing-minimal, and the result follows
immediately from Lemma \ref{lem:removecircles}. Assume $N > 0$, and assume that any graph with $M$ crossings for $M < N$ is in the span of reduced graphs with $\le M$ crossings. If
$\Gamma$ is not crossing minimal, it is in the span of graphs with $< N$ crossings by Lemma \ref{lem:toomanycrossings}, and hence by induction in the span of reduced graphs with $<
N$ crossings. If $\Gamma$ is crossing-minimal, then it is in the span of reduced graphs with $N$ crossings by Lemma \ref{lem:removecircles}. Either way, $\Gamma$ is in the span of
reduced graphs with $\le N$ crossings.

Pick one ``special'' reduced graph for each matching, and let $\mathcal{S}(2\ell)$ be the set of these graphs. 
It follows from an induction using Lemma \ref{lem:changerex} that the span of reduced graphs with $\le N$ crossings 
is equal to the span of the set of special reduced graphs with $\le N$ crossings.
\end{proof} 

\begin{proof}[Proof of Lemma \ref{lem:removecircles}] 
Suppose that $\Gamma$ is a crossing-minimal graph. 
This implies that any closed strand in $\Gamma$ has no self-intersections, thus is an embedded closed $1$-manifold, i.e. an embedded circle. 
The JCT, together with crossing-minimality, implies that no other strands intersect such circles.
Further, the JCT implies the existence of an innermost circle, which contains no strands in its interior.
Up to a scalar, we can replace this circle with the empty diagram using the first relation in \eqref{eq:spn}. 
The result now follows by induction on the number of closed strands.
\end{proof}

The other two lemmas are related to an analogue for quadrivalent graphs of Reidemeister's Theorem \cite{Reidemeister,AlexanderBriggs} on knot diagrams. 
Informally, Lemma \ref{lem:webReidemeister} shows that the \emph{graph Reidemeister moves}:
\begin{equation}\label{eq:graphR}
\mathrm{gRI} :
\begin{tikzpicture} [scale=.5,anchorbase]
	\draw[very thick] (1,1) to [out=270,in=0] (.75,.75) to [out=180,in=270] (0,2);
	\draw[very thick] (0,0) to [out=90,in=180] (.75,1.25) to [out=0,in=90] (1,1);
\end{tikzpicture}
\sim
\begin{tikzpicture} [scale=.5,anchorbase]
	\draw[very thick] (0,0) to (0,2);
\end{tikzpicture}
\quad , \quad 
\mathrm{gRII} :
\begin{tikzpicture} [scale=.35,anchorbase]
	\draw[very thick] (1,0) to [out=90,in=270] (0,1.5) to [out=90,in=270] (1,3);
	\draw[very thick] (0,0) to [out=90,in=270] (1,1.5) to [out=90,in=270] (0,3);
\end{tikzpicture}
\sim
\begin{tikzpicture}[scale=.35, anchorbase]
	\draw[very thick] (1,0) to (1,3);
	\draw[very thick] (0,0) to (0,3);
\end{tikzpicture}
\quad , \quad
\mathrm{gRIII} :
\begin{tikzpicture}[scale=.2, tinynodes, anchorbase]
	\draw[very thick] (-2,-3) to [out=90,in=270] (2,3);
	\draw[very thick] (0,-3) to [out=90,in=270] (-2,0) to [out=90,in=270] (0,3);
	\draw[very thick] (2,-3) to [out=90,in=270] (-2,3);	
\end{tikzpicture}
\sim
\begin{tikzpicture}[scale=.2, tinynodes, anchorbase]
	\draw[very thick] (-2,-3) to [out=90,in=270] (2,3);
	\draw[very thick] (0,-3) to [out=90,in=270] (2,0) to [out=90,in=270] (0,3);
	\draw[very thick] (2,-3) to [out=90,in=270] (-2,3);	
\end{tikzpicture}
\end{equation}
are satisfied in $\Web^\times(\spn)$, up to lower-order terms and multiplication by invertible scalars.

Given two graphs that correspond to the same matching and have the same number of components, 
there exists a sequence of graph Reidemeister moves relating them. 
(This can be proved by ``lifting'' both to diagrams of the corresponding trivial tangle, 
and applying the classical Reidemeister theorem.)
We would like to deduce Lemmas \ref{lem:toomanycrossings} and \ref{lem:changerex} 
by applying this result to a given graph and a chosen crossing-minimal (or reduced) graph for the corresponding matching.
However, since the graph Reidemeister moves only hold up to lower order terms, 
we must guarantee that the moves gRI and gRII need only be used in the direction that does not introduce
additional crossings:
\begin{equation} \label{eq:GRdirectional}
\mathrm{gRI'} :
\begin{tikzpicture} [scale=.5,anchorbase]
	\draw[very thick] (1,1) to [out=270,in=0] (.75,.75) to [out=180,in=270] (0,2);
	\draw[very thick] (0,0) to [out=90,in=180] (.75,1.25) to [out=0,in=90] (1,1);
\end{tikzpicture}
\mapsto
\begin{tikzpicture} [scale=.5,anchorbase]
	\draw[very thick] (0,0) to (0,2);
\end{tikzpicture}
\quad , \quad 
\mathrm{gRII'} :
\begin{tikzpicture} [scale=.35,anchorbase]
	\draw[very thick] (1,0) to [out=90,in=270] (0,1.5) to [out=90,in=270] (1,3);
	\draw[very thick] (0,0) to [out=90,in=270] (1,1.5) to [out=90,in=270] (0,3);
\end{tikzpicture}
\mapsto
\begin{tikzpicture}[scale=.35, anchorbase]
	\draw[very thick] (1,0) to (1,3);
	\draw[very thick] (0,0) to (0,3);
\end{tikzpicture}
\end{equation}
It is not true that two graphs related by graph Reidemeister moves are necessarily related when gRI and gRII can be used in only one direction. 
For an example, compare two nested circles to two non-nested circles. 
However, we claim that this counterexample captures the only obstructions: circles appearing as subdiagrams. 
In $\Web^\times(\spn)$ these are handled by the circle removal relation (\text{\ref{eq:spn}a}).

In the next section we make these arguments precise, and prove Lemma \ref{lem:toomanycrossings} and Lemma \ref{lem:changerex} rigorously. 
As we learned in the final stages of preparation of this manuscript, similar results have previously been established; see \cite[Lemma 2]{Carpentier}.
Nevertheless, we include our arguments to keep the exposition thorough and self-contained.
The reader who is content with this informal justification for Lemmas \ref{lem:toomanycrossings} and \ref{lem:changerex}
will not lose anything by skipping the next section.

\subsection{Arguments with graph Reidemeister moves} \label{sec:topologicalstuff}

To prove Lemma \ref{lem:toomanycrossings}, we will show that a non-crossing-minimal graph must contain one of two configurations that we call \emph{big curls} and \emph{big bigons}. 
Informally, a big curl is a portion of $\Gamma$ that (up to isotopy) resembles the left-hand side of the
graph Reidemeister move gRI, with additional vertices along the curl and edges/vertices in its interior:
\begin{equation} \label{eq:bigcurl}
\begin{tikzpicture} [scale=.75,anchorbase]
	\draw[very thick, red] (.25,-.25) to (0,0) to [out=45,in=270] (1,1.25) to [out=90,in=0] (0,2)
		to [out=180,in=90] (-1,1.25) to [out=270,in=135] (0,0) to (-.25,-.25);
	\draw[very thick] (-1.5,1.375) to (1.5,1.375);
	\draw[very thick] (-1.5,.875) to (1.5,.875);
	\draw[very thick,fill=white] (-.5,.625) rectangle (.5,1.625);
	\node at (0,1.125) {$\tilde{\Gamma}$};
	\node[rotate=90] at (.75,1.125) {$\mydots$};
	\node[rotate=90] at (-.75,1.125) {$\mydots$};
\end{tikzpicture}
\end{equation}
Here, the big curl is depicted in red. Similarly, a big bigon resembles the left-hand side of gRII. 
\begin{equation}\label{eq:bigbigon}
\begin{tikzpicture} [scale=.75,anchorbase]
	\draw[very thick, red] (-.25,-.25) to (0,0) to [out=45,in=270] (1,1.25) to [out=90,in=315] (0,2.5) to (-.25,2.75);
	\draw[very thick, red] (.25,-.25) to (0,0) to [out=135,in=270] (-1,1.25) to [out=90,in=225] (0,2.5) to (.25,2.75);
	\draw[very thick] (-1.5,1.5) to (1.5,1.5);
	\draw[very thick] (-1.5,1) to (1.5,1);
	\draw[very thick,fill=white] (-.5,.75) rectangle (.5,1.75);
	\node at (0,1.25) {$\tilde{\Gamma}$};
	\node[rotate=90] at (.75,1.25) {$\mydots$};
	\node[rotate=90] at (-.75,1.25) {$\mydots$};
\end{tikzpicture}
\end{equation}
A consequence of our proofs below is that any graph containing a big curl or big bigon can be transformed into a diagram with fewer crossings, 
using circle removal and graph Reidemeister moves, where gRI and gRII are only applied in the (simplifying) direction indicated in \eqref{eq:GRdirectional}.
See Porism \ref{por:tangle} and the discussion preceding it.

Let us now give a rigorous definition of a big curl. 

\begin{defn}
Let $\Gamma$ be a graph, which we view as the image of a generic immersion
\[
\coprod_{j=1}^c S^1\sqcup \coprod_{i=1}^\ell [0,1] \xrightarrow{\gamma} \D.
\]
Given a crossing $x \in \Gamma$, we say that there is a \emph{big curl at $x$} provided that we can find a topological 
interval\footnote{The notation $[a,b]$ for a topological interval inside $[0,1]$ or $S^1$ is just useful shorthand. It does not indicate that $a < b$, 
only that $\{a,b\}$ is the boundary of the interval. We use $(a,b)$ to indicate the interior of the interval.} 
$I = [a,b]$ inside one component of the domain ($S^1 \text{ or } [0,1]$) such that 
\begin{enumerate}
\item $\gamma^{-1}(x) = \{a,b\}$. 
\item $\gamma$ is injective on $(a,b)$.
\item The above conditions imply that $\gamma(I)$ is a simple closed curve, whose complement has an interior (bounded component) and an exterior (unbounded component) by the JCT. 
We require that the preimage under $\gamma$ of the interior does not meet a neighborhood of $I$.
\end{enumerate}
We refer to the simple closed curve $\gamma(I)$ as a \emph{big curl}.
\end{defn}

Here, $x$ represents the red crossing in \eqref{eq:bigcurl}, and $I$ is the segment of the red strand which forms the curl. 
Note that our definition rules out unwanted behavior. Specifically, we do not want the red strand to have any complicating features along the curl $I$, such as additional self-intersections. 
Additionally, we want the two other edges coming out of $x$ to be on the ``outside'' of the curl rather than the ``inside.'' 
We do, however, want to permit some of the strands in $\tilde{\Gamma}$ to be in the same (red) component in the whole diagram $\Gamma$, 
even if they are not part of the same component in the pictured subdiagram (after all, the whole graph might be the projection of a complicated knot, with only one component).

\begin{example}
The following contrasts a big curl (in red) with an interval that is not a big curl (in blue):
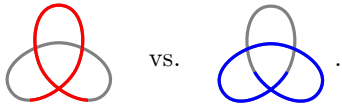
\[
\begin{tikzpicture} [scale=.75,anchorbase]
	\draw[very thick,gray] (0,1) to [out=180,in=120] (-.289,-.167) to [out=300,in=240] (.866,-.5) to [out=60,in=0] (0,.333)
		to [out=180,in=120] (-.866,-.5) to [out=300,in=240] (.289,-.167) to [out=60,in=0] (0,1);
	\begin{scope}
	\clip (-.5,1.05) rectangle (.5,-1);
	\draw[very thick, red] (-.866,-.5) to [out=300,in=240] (.289,-.167) to [out=60,in=0] (0,1) to [out=180,in=120] 
		(-.289,-.167) to [out=300,in=240] (.866,-.5);
	\end{scope}
\end{tikzpicture}
\quad \text{vs.} \quad
\begin{tikzpicture} [scale=.75,anchorbase]
	\draw[very thick,gray] (0,1) to [out=180,in=120] (-.289,-.167) to [out=300,in=240] (.866,-.5) to [out=60,in=0] (0,.333)
		to [out=180,in=120] (-.866,-.5) to [out=300,in=240] (.289,-.167) to [out=60,in=0] (0,1);
	\begin{scope}
	\draw[very thick,blue] (-.289,-.167) to [out=300,in=240] (.866,-.5) to [out=60,in=0] (0,.333)
		to [out=180,in=120] (-.866,-.5) to [out=300,in=240] (.289,-.167);
	\end{scope}
\end{tikzpicture}.
\]
\end{example}

\begin{lem} Let $C$ be a strand in $\Gamma$, and view $C$ as a quadrivalent graph in its own right. If $C$ has a big curl, then so does $\Gamma$. \end{lem}
	
\begin{proof} Since the definition of a big curl only depends on the restriction of $\gamma$ to one component of its domain, this is immediate. \end{proof}

\begin{lem}\label{lem:bigcurl}
Let $\Gamma$ be a graph in which a strand intersects itself. 
Then $\Gamma$ contains a big curl.
\end{lem}

\begin{proof}
By the previous lemma, it suffices to consider the case when $\Gamma \subset \D$ has a single strand.
First, suppose that $\Gamma$ is not a closed strand, and view it as the image of a generic immersion $[0,1] \xrightarrow{\gamma} \D$. 
Let 
\begin{equation}
b = \sup \{ t \in (0,1) \mid \gamma|_{(0,t)} \text{ is injective} \}.
\end{equation}
Then $b < 1$ since $\Gamma$ intersects itself.
It follows that $\gamma(b)$ is a crossing $x \in \Gamma$ and $\gamma^{-1}(x) = \{ a, b \}$ for some $a < b$.
We claim that the interval $I = [a,b]$ is a big curl at $x$. The first two conditions are satisfied by construction.
One of the other edges incident at $x$ is a subset of $\gamma\big([0,a)\big)$, 
which runs to the boundary and does not intersect $\gamma(I)$. By the JCT, this edge must lie on the exterior of $\gamma(I)$. 
The transversality of the intersection at $x$ implies that the remaining edge lies in the exterior as well.

An analogous argument works in the case that $\Gamma$ is a closed strand.
Choose a point $p \in \Gamma$ that is not on a crossing, and such that there exists a ray from $p$ to the boundary that does not meet $\Gamma$.
Now view $\Gamma$ as the image of a generic immersion $S^1 \xrightarrow{\gamma} \D$ such that $\gamma(0) = \gamma(1) = p$. 
Here, we identify $S^1 \cong [0,1] \big/ 0 \sim 1$.
Repeating the above argument then produces the big curl. Note that $p$ must be in the exterior of $\gamma(I)$ by assumption\footnote{Alternatively, 
one can excise a neighborhood of the ray from $p$ to the boundary, to obtain a subgraph where the strand is not closed, but instead has two boundary points near $p$.}.
\end{proof}

Next we wish to prove that any non-crossing-minimal graph without self-intersecting strands (e.g. Example \ref{ex:inversion}) will have a big bigon. First we must rigorously define a big bigon. 
Here is a reminder; the big bigon is in red.
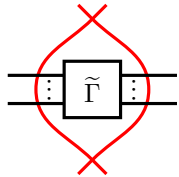
\[
\begin{tikzpicture} [scale=.75,anchorbase]
	\draw[very thick, red] (-.25,-.25) to (0,0) to [out=45,in=270] (1,1.25) to [out=90,in=315] (0,2.5) to (-.25,2.75);
	\draw[very thick, red] (.25,-.25) to (0,0) to [out=135,in=270] (-1,1.25) to [out=90,in=225] (0,2.5) to (.25,2.75);
	\draw[very thick] (-1.5,1.5) to (1.5,1.5);
	\draw[very thick] (-1.5,1) to (1.5,1);
	\draw[very thick,fill=white] (-.5,.75) rectangle (.5,1.75);
	\node at (0,1.25) {$\tilde{\Gamma}$};
	\node[rotate=90] at (.75,1.25) {$\mydots$};
	\node[rotate=90] at (-.75,1.25) {$\mydots$};
\end{tikzpicture}
\]

\begin{defn}
Let $\Gamma$ be a graph, which we view as the image of a generic immersion
\[
\coprod_{j=1}^c S^1\sqcup \coprod_{i=1}^\ell [0,1] \xrightarrow{\gamma} \D.
\]
Given distinct crossings $x, y \in \Gamma$, we say that there is a 
\emph{big bigon between $x$ and $y$} provided that we can find disjoint topological intervals\footnote{Again, there is no requirement that $x_i < y_i$. 
This is important here because e.g. it is permitted that $I_1$ and $I_2$ live inside $[0,1]$ with $x_1 < y_1 < y_2 < x_2$. 
We do not wish our notation  to imply that we care about the ``orientation'' of $I_1$ relative to $I_2$.} 
$I_1 = [x_1, y_1]$ and $I_2 = [x_2, y_2]$ in the domain (possibly on the same component) such that
\begin{enumerate}
\item $\gamma^{-1}(x) = \{x_1, x_2\}$ and $\gamma^{-1}(y) = \{y_1, y_2\}$,
\item $\gamma$ is injective on $(x_1, y_1) \cup (x_2, y_2)$,
\item The above conditions imply that $\gamma(I_1 \cup I_2)$ is a simple closed curve, whose complement has an interior (bounded component) and an exterior (unbounded component) by the JCT. We require that the preimage under $\gamma$ of the interior does not meet a neighborhood of $I_1 \cup I_2$.
\end{enumerate}
We will refer to the simple closed curve $\gamma(I_1 \cup I_2)$ as a big bigon.
\end{defn}

\begin{lem} Let $C_1$ and $C_2$ be two strands inside $\Gamma$ (possibly equal), and view $C_1 \cup C_2$ as a quadrivalent graph of its own right. If $C_1 \cup C_2$ has a big bigon, then so does $\Gamma$. \end{lem}
	
\begin{proof} Since the definition of a big bigon only depends on the restriction of $\gamma$ to (up to) two components of its domain, this is immediate. \end{proof}

\begin{lem}\label{lem:bigbigon}
Let $\Gamma$ be a graph that is not crossing-minimal and contains no big curls. 
Then $\Gamma$ contains a big bigon.
\end{lem}

\begin{proof}
Since $\Gamma$ contains no big curls, Lemma \ref{lem:bigcurl} implies that it has no self-intersecting strands. 
Since $\Gamma$ is not crossing-minimal, it must contain two distinct strands that intersect more than once. 
By the previous lemma it suffices to consider the case when $\Gamma$ consists of these two strands.
However, in this case the result is a trivial consequence of the JCT\footnote{In fact, in this case the JCT can be used to prove the existence of a \emph{small bigon}, 
i.e. a subgraph as in the left-hand side of the graph Reidemeister move gRII, but the argument is more complex.}. 
One strand of $\Gamma$ will divide the disk into two regions. The other strand must cross the first at least twice, and two consecutive crossings will yield a big bigon. \end{proof}

\begin{proof}[Proof of Lemma \ref{lem:toomanycrossings}] 
Let $\Gamma$ be a graph that is not crossing-minimal, and suppose that $\Gamma$ has $N$ crossings. 
We prove that $\Gamma \in \Web^\times(\spn)$ is in the span of diagrams with $< N$ crossings, by induction on $N$. 
Lemma \ref{lem:bigbigon} implies that $\Gamma$ contains either a big curl or a big bigon.

Suppose that $\Gamma$ contains a big curl with crossing $x$. Consider the closed simply-connected region $K$ enclosed by the dotted lines in \eqref{eq:curlnbhd}.
\begin{equation} \label{eq:curlnbhd}
\begin{tikzpicture} [scale=.75,anchorbase]
	\draw[very thick, red] (.25,-.25) to (0,0) to [out=45,in=270] (1,1.25) to [out=90,in=0] (0,2)
		to [out=180,in=90] (-1,1.25) to [out=270,in=135] (0,0) to (-.25,-.25);
	\draw[very thick] (-1.5,1.375) to (1.5,1.375);
	\draw[very thick] (-1.5,.875) to (1.5,.875);
	\draw[very thick,fill=white] (-.5,.625) rectangle (.5,1.625);
	\node at (0,1.125) {$\tilde{\Gamma}$};
	\node[rotate=90] at (.75,1.125) {$\mydots$};
	\node[rotate=90] at (-.75,1.125) {$\mydots$};
	\path[fill=gray,opacity=.75] (0,0) circle (.3);
	\draw[densely dashed] (.25,0) to [out=45,in=270] (1.25,1.25) to [out=90,in=0] (0,2.25)
		to [out=180,in=90] (-1.25,1.25) to [out=270,in=135] (-.25,0) to [out=45,in=180] (0,.25) to [out=0,in=135] (.25,0);
\end{tikzpicture}
\Longrightarrow
\Gamma' =
\begin{tikzpicture} [scale=.75,anchorbase]
\begin{scope} 
	\clip (0,1.125) circle (1.25);
	\draw[very thick, red] (-.875,-2) to (-.875,1.125) to [out=90,in=180] (0,2) to [out=0,in=90] (.875,1.125) to (.875,-2);
	\draw[very thick] (-1.5,1.375) to (1.5,1.375);
	\draw[very thick] (-1.5,.875) to (1.5,.875);
	\draw[very thick,fill=white] (-.5,.625) rectangle (.5,1.625);
	\node at (0,1.125) {$\tilde{\Gamma}$};
	\node[rotate=90] at (.65,1.125) {$\mydots$};
	\node[rotate=90] at (-.65,1.125) {$\mydots$};
 \end{scope}
	 \draw[densely dashed] (0,1.125) circle (1.25);
\end{tikzpicture}
\end{equation}
To be precise, we can let $B$ be the union of the big curl and its interior, excise from $B$ a neighborhood of the crossing $x$, take an open neighborhood of the result, and let $K$ be its closure. 
Set $\Gamma' = K \cap \Gamma$, which we can view as a graph in the planar disk; its boundary is $\partial K \cap \Gamma$. 
Set $\tilde{\Gamma}$ to be the subgraph in the interior\footnote{Technically, we should look at the interior, excise a neighborhood of the big curl, take the closure, 
and intersect with $\Gamma$. The concept is clear from the picture, so we will no longer stress such technicalities.} of the big curl.

Note that $\Gamma'$ has strictly fewer than $N$ crossings, because it is missing the crossing $x$. Suppose $\Gamma'$ has $M$ crossings. 
If $\Gamma'$ is not crossing-minimal, then by induction we can replace $\Gamma'$ in $\Web^\times(\spn)$ with a linear combination of diagrams with $< M$ crossings. 
Applying this local move to $\Gamma$, we can replace $\Gamma$ by a linear combination of diagrams with $< N$ crossings, as desired. 
Thus, we can assume that $\Gamma'$ is crossing minimal. By Lemma \ref{lem:removecircles}, we can take $\Gamma'$ to be reduced. By the same argument, we can assume $\tilde{\Gamma}$ is reduced.

Further, $\Gamma'$ has a distinguished (red) strand which was part of the big curl. 
The boundary points of the distinguished strand are adjacent, so it is not involved in any inversions. Since $\Gamma'$ is reduced, no other strands cross the distinguished strand. 
By construction, every strand on the boundary of $\tilde{\Gamma}$ will cross the distinguished strand in $\Gamma'$, so the boundary of $\tilde{\Gamma}$ is empty. 
Since $\tilde{\Gamma}$ is reduced, it must be the empty diagram. Consequently, the big curl is simply an ``ordinary curl,'' i.e. the left-hand side of \eqref{eq:R1}, 
and this relation expresses $\Gamma'$ as a linear combination of graphs with fewer crossings. This concludes the case when $\Gamma$ has a big curl.

Suppose that $\Gamma$ does not contain a big curl. 
Lemma \ref{lem:bigbigon} then implies that $\Gamma$ contains a big bigon, with crossings $x$ and $y$. 
Once again, consider the closed region $K$ enclosed by the dotted lines in \eqref{eq:bigonnbhd}, and the corresponding subgraph $\Gamma' = K \cap \Gamma$.
\begin{equation} \label{eq:bigonnbhd}
\begin{tikzpicture} [scale=.75,anchorbase]
	\draw[very thick, red] (-.25,-.25) to (0,0) to [out=45,in=270] (1,1.25) to [out=90,in=315] (0,2.5) to (-.25,2.75);
	\draw[very thick, red] (.25,-.25) to (0,0) to [out=135,in=270] (-1,1.25) to [out=90,in=225] (0,2.5) to (.25,2.75);
	\draw[very thick] (-1.5,1.5) to (1.5,1.5);
	\draw[very thick] (-1.5,1) to (1.5,1);
	\draw[very thick,fill=white] (-.5,.75) rectangle (.5,1.75);
	\node at (0,1.25) {$\tilde{\Gamma}$};
	\node[rotate=90] at (.75,1.25) {$\mydots$};
	\node[rotate=90] at (-.75,1.25) {$\mydots$};
	\path[fill=gray,opacity=.75] (0,0) circle (.3);
	\path[fill=gray,opacity=.75] (0,2.5) circle (.3);
	\draw[densely dashed] (.25,0) to [out=45,in=270] (1.25,1.25) to [out=90,in=315] (.25,2.5) to [out=225,in=0] (0,2.25) 
		to [out=180,in=315] (-.25,2.5) to [out=225,in=90] (-1.25,1.25) to [out=270,in=135] (-.25,0) to [out=45,in=180] (0,.25) 
			to[out=0,in=135] (.25,0);
\end{tikzpicture}
\Longrightarrow
\Gamma' =
\begin{tikzpicture} [scale=.75,anchorbase]
\begin{scope} 
	\clip (0,1.125) circle (1.25);
	\draw[very thick, red] (-.875,-2) to (-.875,2);
	\draw[very thick, red] (.875,-2) to (.875,2);
	\draw[very thick] (-1.5,1.375) to (1.5,1.375);
	\draw[very thick] (-1.5,.875) to (1.5,.875);
	\draw[very thick,fill=white] (-.5,.625) rectangle (.5,1.625);
	\node at (0,1.125) {$\tilde{\Gamma}$};
	\node[rotate=90] at (.65,1.125) {$\mydots$};
	\node[rotate=90] at (-.65,1.125) {$\mydots$};
 \end{scope}
	 \draw[densely dashed] (0,1.125) circle (1.25);
\end{tikzpicture}
\end{equation}
We can define $K$ formally by letting $B$ be the union of the big bigon and its interior, excising neighborhoods of $x$ and $y$ from $B$, taking an open neighborhood of the result, and letting $K$ be its closure. 
We let $\tilde{\Gamma}$ be the subgraph in the interior of the big bigon. 
We observe, as before, that $\Gamma'$ and $\tilde{\Gamma}$ have $<N$ crossings, and thus can be assumed reduced.

The graph $\Gamma'$ has a \emph{left} and \emph{right} distinguished strand, which formed the two edges of the original bigon. 
Because $\Gamma'$ is reduced, no strand in $\tilde{\Gamma}$ will meet either distinguished strand twice. 
Since $\tilde{\Gamma}$ is reduced it has no closed strands, so any strand must have two endpoints on the boundary, one on the left and one on the right. 
Thus, the matching induced by $\tilde{\Gamma}$ is a permutation between the left boundary points and the right boundary points.
Since $\tilde{\Gamma}$ is reduced, it is a string diagram for a reduced expression of an element in the symmetric group.

Suppose $\Gamma'$ has $M$ crossings. 
Relation \eqref{eq:R3} now allows us to slide all crossings in $\tilde{\Gamma}$ across the left distinguished strand, modulo graphs with strictly fewer crossings.
\begin{equation}
\Gamma' = \begin{tikzpicture} [scale=.75,anchorbase]
\begin{scope} 
	\clip (0,1.125) circle (1.25);
	\draw[very thick, red] (-.875,-2) to (-.875,2);
	\draw[very thick, red] (.875,-2) to (.875,2);
	\draw[very thick] (-1.5,1.375) to (1.5,1.375);
	\draw[very thick] (-1.5,.875) to (1.5,.875);
	\draw[very thick,fill=white] (-.5,.625) rectangle (.5,1.625);
	\node at (0,1.125) {$\tilde{\Gamma}$};
	\node[rotate=90] at (.65,1.125) {$\mydots$};
	\node[rotate=90] at (-.65,1.125) {$\mydots$};
 \end{scope}
	 \draw[densely dashed] (0,1.125) circle (1.25);
\end{tikzpicture} = 
\begin{tikzpicture} [scale=.75,anchorbase]
\begin{scope} 
	\clip (.375,1.125) circle (1.25);
	\draw[very thick, red] (1.25,-2) to (1.25,3);
	\draw[very thick, red] (.875,-2) to (.875,3);
	\draw[very thick] (-1.75,1.375) to (1.75,1.375);
	\draw[very thick] (-1.75,.875) to (1.75,.875);
	\draw[very thick,fill=white] (-.5,.625) rectangle (.5,1.625);
	\node at (0,1.125) {$\tilde{\Gamma}$};
	\node[rotate=90] at (.65,1.125) {$\mydots$};
	\node[rotate=90] at (-.65,1.125) {$\mydots$};
 \end{scope}
	 \draw[densely dashed] (.375,1.125) circle (1.25);
\end{tikzpicture}
+
\begin{tikzpicture} [scale=.75,anchorbase]
\begin{scope} 
	\clip (0,1.125) circle (1.25);
	\draw[very thick, red] (-.375,-2) to (-.375,3);
	\draw[very thick, red] (.375,-2) to (.375,3);
	\draw[very thick] (-1.5,1.375) to (1.5,1.375);
	\draw[very thick] (-1.5,.875) to (1.5,.875);
	\draw[very thick,fill=white] (-.75,.625) rectangle (.75,1.625);
	\node at (0,1.125) {$<M$};
	\node[rotate=90] at (1,1.125) {$\mydots$};
	\node[rotate=90] at (-1,1.125) {$\mydots$};
 \end{scope}
	 \draw[densely dashed] (0,1.125) circle (1.25);
\end{tikzpicture} \quad \text{in } \Web^\times(\spn).
\end{equation}
Considering the first diagram on the right-hand side, 
we have essentially reduced to the case where $\tilde{\Gamma}$ is the identity permutation on $k$ elements, whence $\Gamma'$ has $2k$ crossings.
Returning to our big bigon, we then compute
\begin{equation}
\begin{tikzpicture} [scale=.5,anchorbase]
	\draw[very thick, red] (-.25,-.25) to (0,0) to [out=45,in=270] (1,1.25) to [out=90,in=315] (0,2.5) to (-.25,2.75);
	\draw[very thick, red] (.25,-.25) to (0,0) to [out=135,in=270] (-1,1.25) to [out=90,in=225] (0,2.5) to (.25,2.75);
	\draw[very thick] (-1.5,1.75) to (1.5,1.75);
	\draw[very thick] (-1.5,.75) to (1.5,.75);
	\node[rotate=90] at (0,1.25) {$\mydots$};
\end{tikzpicture}
\stackrel{\eqref{eq:R3}}{=}
\begin{tikzpicture} [scale=.5,anchorbase]
	\draw[very thick, red] (-.25,-.25) to (0,0) to [out=45,in=270] (.25,.375) to [out=90,in=315] (0,.75) to [out=135,in=270] (-.25,2.75);
	\draw[very thick, red] (.25,-.25) to (0,0) to [out=135,in=270] (-.25,.375) to [out=90,in=225] (0,.75) to [out=45,in=270] (.25,2.75);
	\draw[very thick] (-1.5,2.25) to (1.5,2.25);
	\draw[very thick] (-1.5,1.25) to (1.5,1.25);
	\node[rotate=90] at (0,1.75) {$\mydots$};
\end{tikzpicture}
+ 
\begin{tikzpicture} [scale=.5,anchorbase,smallnodes]
	\draw[very thick,red] (-.25,-.625) to (-.25,2.125);
	\draw[very thick,red] (.25,-.625) to (.25,2.125);
	\draw[very thick] (-1.625,0.625) to (1.625,0.625);
	\draw[very thick] (-1.625,1.375) to (1.625,1.375);
	\draw[very thick,fill=white] (-1.125,.375) rectangle (1.125,1.625);
	\node at (0,1) {$< \! 2k{+}2$};
	\node[rotate=90] at (1.375,1) {$\mydots$};
	\node[rotate=90] at (-1.3,1) {$\mydots$};
\end{tikzpicture}
\stackrel{\eqref{eq:R2}}{=}
\begin{tikzpicture} [scale=.5,anchorbase,smallnodes]
	\draw[very thick,red] (-.25,-.625) to (-.25,2.125);
	\draw[very thick,red] (.25,-.625) to (.25,2.125);
	\draw[very thick] (-1.625,0.625) to (1.625,0.625);
	\draw[very thick] (-1.625,1.375) to (1.625,1.375);
	\draw[very thick,fill=white] (-1.125,.375) rectangle (1.125,1.625);
	\node at (0,1) {$< \! 2k{+}2$};
	\node[rotate=90] at (1.375,1) {$\mydots$};
	\node[rotate=90] at (-1.3,1) {$\mydots$};
\end{tikzpicture}.
\end{equation}
Applying this local relation inside $\Gamma$ will express it as a linear combination of graphs with strictly fewer than $N$ crossings.
\end{proof}

If we repeat the argument from the proof of Lemma \ref{lem:toomanycrossings} in the context of graphs modulo the 
graph Reidemeister moves \eqref{eq:graphR} and the ``circle removal'' relation
$
\begin{tikzpicture}[scale =.35, tinynodes,anchorbase]
	\draw[very thick] (0,0) circle (.5);
\end{tikzpicture} \mapsto \varnothing,
$
it shows that any graph can be taken to a reduced graph corresponding to the same matching 
using circle removal, $\mathrm{gRIII}$, and the simplifying moves $\mathrm{gRI'}$ and $\mathrm{gRII'}$ from \eqref{eq:GRdirectional}.
Since quadrivalent graphs can be interpreted as tangle diagrams modulo the relation of crossing change (over $\leftrightarrow$ under), 
we pause to record the following (which presumably is known to experts):

\begin{por}\label{por:tangle}
Any tangle diagram can be taken to a diagram for a corresponding untangle with all closed components removed
using a sequence of the following operations:
\begin{itemize}
	\item Reidemeister moves that do not increase the number of crossings,
	\item crossing change moves, and
	\item the circle removal move.
\end{itemize}
As a consequence, any knot diagram can be taken to the standard crossingless diagram of the unknot using only 
Reidemeister moves that do not increase the number of crossings and crossing change moves.
\end{por}

Note that the ``hard unknots'' discussed in \cite{HardU} show that, even for unknots, 
crossing change is a necessary move in the latter statement.

Finally, we turn our attention to Lemma \ref{lem:changerex}. 
This is an immediate consequence of the following result, 
which is a quadrivalent graph analogue of Matsumoto's theorem \cite{Matsumoto} for reduced expressions in Coxeter groups.

\begin{prop}\label{prop:Matsumoto}
Let $\Gamma_1,\Gamma_2 \subset \D$ be reduced graphs that correspond to the same matching, 
then there is a sequence of graph Reidemeister moves gRIII that take $\Gamma_1$ to $\Gamma_2$.
\end{prop}

\begin{proof}
Let $\Gamma$ be a reduced graph. We begin by identifying a distinguished strand in $\Gamma$, and use gRIII moves to simplify $\Gamma$ with respect to this strand.
There exists an ordered pair $(a,b)$ of matched points so that, along the counterclockwise segment $I = [a,b] \subset \partial \D$ from $a$ to $b$, 
there are no matched points between $a$ and $b$. 
Denote this distinguished strand between $a$ and $b$ by $S$ (in red below), and let $U$ denote the interior of the simple closed curve $S \cup I$. Let $\tilde{\Gamma}$ be the intersection of $\Gamma$ with $U$.
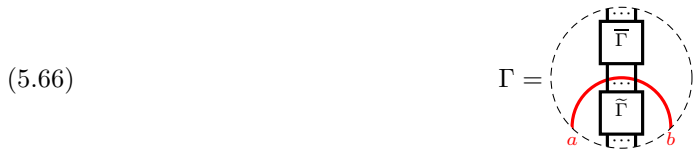
\begin{equation}
\Gamma =
\begin{tikzpicture} [scale=.75,anchorbase,smallnodes]
\begin{scope} 
	\clip (0,2) circle (1.25);
	\draw[very thick, red] (-.875,-2) to (-.875,1.125) to [out=90,in=180] (0,2) to [out=0,in=90] (.875,1.125) to (.875,-2);
	\draw[very thick] (-.25,-5) to (-.25,5);
	\draw[very thick] (.25,-5) to (.25,5);
	\draw[very thick,fill=white] (-.375,2.25) rectangle (.375,3);
	\node at (0,2.625) {$\overline{\Gamma}$};
	\draw[very thick,fill=white] (-.375,1) rectangle (.375,1.75);
	\node at (0,1.375) {$\tilde{\Gamma}$};
	\node at (0,.875) {\scriptsize$\mydots$};
	\node at (0,1.875) {\scriptsize$\mydots$};
	\node at (0,3.125) {\scriptsize$\mydots$};
 \end{scope}
	 \draw[densely dashed] (0,2) circle (1.25);
	 \node[red] at (-.875,.875) {$a$};
	 \node[red] at (.875,.875) {$b$};
\end{tikzpicture}
\end{equation}
No strand in $\tilde{\Gamma}$ will match two points on its bottom boundary (by our choice of interval $[a,b]$), nor will it match two points on the top boundary 
(since this strand would intersect $S$ twice, violating the reducedness of $\Gamma$). 
Thus, as in the big bigon case of the proof of Lemma \ref{lem:toomanycrossings}, 
$\tilde{\Gamma}$ is a string diagram for a (reduced) expression of an element in the symmetric group.

Hence, we can apply a sequence of gRIII moves to slide all crossings in $\tilde{\Gamma}$ 
across $S$ into $\D \smallsetminus U$. 
This produces a graph $\Gamma'$ corresponding to the same matching with the same distinguished strand, wherein the region $U$ contains no crossings:
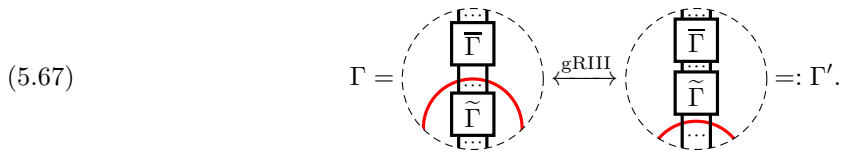
\begin{equation}
\Gamma= 
\begin{tikzpicture} [scale=.75,anchorbase]
\begin{scope} 
	\clip (0,2) circle (1.25);
	\draw[very thick, red] (-.875,-2) to (-.875,1.125) to [out=90,in=180] (0,2) to [out=0,in=90] (.875,1.125) to (.875,-2);
	\draw[very thick] (-.25,-5) to (-.25,5);
	\draw[very thick] (.25,-5) to (.25,5);
	\draw[very thick,fill=white] (-.375,2.25) rectangle (.375,3);
	\node at (0,2.625) {$\overline{\Gamma}$};
	\draw[very thick,fill=white] (-.375,1) rectangle (.375,1.75);
	\node at (0,1.375) {$\tilde{\Gamma}$};
	\node at (0,.875) {\scriptsize$\mydots$};
	\node at (0,1.875) {\scriptsize$\mydots$};
	\node at (0,3.125) {\scriptsize$\mydots$};
 \end{scope}
	 \draw[densely dashed] (0,2) circle (1.25);
\end{tikzpicture}
\xleftrightarrow{\mathrm{gRIII}}
\begin{tikzpicture} [scale=.75,anchorbase]
\begin{scope} 
	\clip (0,2) circle (1.25);
	\draw[very thick, red] (-.875,-3) to (-.875,0.375) to [out=90,in=180] (0,1.25) to [out=0,in=90] (.875,0.375) to (.875,-3);
	\draw[very thick] (-.25,-5) to (-.25,5);
	\draw[very thick] (.25,-5) to (.25,5);
	\draw[very thick,fill=white] (-.375,2.3125) rectangle (.375,3.0625);
	\node at (0,2.6875) {$\overline{\Gamma}$};
	\draw[very thick,fill=white] (-.375,1.375) rectangle (.375,2.125);
	\node at (0,1.75) {$\tilde{\Gamma}$};
	\node at (0,1) {\scriptsize$\mydots$};
	\node at (0,2.225) {\scriptsize$\mydots$};
	\node at (0,3.125) {\scriptsize$\mydots$};
 \end{scope}
	 \draw[densely dashed] (0,2) circle (1.25);
\end{tikzpicture}
=: \Gamma'.
\end{equation}

We now deduce our result by inducting on the number of strands $\ell$ in $\Gamma$, noting that the base case ($\ell=0$) is trivial. 
Let $\Gamma_1$ and $\Gamma_2$ be two graphs with $\ell$ strands that correspond to the same matching (we identify their boundary). 
Let $S_1$ be the distinguished strand in $\Gamma_1$. Since $\Gamma_2$ corresponds to the same matching, this determines a distinguished strand $S_2$ in $\Gamma_2$.
Apply the above procedure to $\Gamma_i$ for $i=1,2$ to obtain graphs $\Gamma'_i$ that are related to $\Gamma_i$ by a sequence of gRIII moves, and such that 
$\Gamma'_i \cap U_i$ contains no crossings.
Now, considering the subgraph in the exterior of the simple closed curves $S_i \cup I$, we have two reduced graphs with $\ell - 1$ strands which represent the same matchings. 
By induction, there is a sequence of gRIII moves relating them, which produces a sequence of gRIII moves relating $\Gamma'_1$ and $\Gamma'_2$.
We thus have
\[
\Gamma_1 \xleftrightarrow{\mathrm{gRIII}} \Gamma'_1  \xleftrightarrow{\mathrm{gRIII}} \Gamma'_2  \xleftrightarrow{\mathrm{gRIII}} \Gamma_2
\]
as desired.
\end{proof}

\begin{proof}[Proof of Lemma \ref{lem:changerex}]
Let $\Gamma$ be a reduced graph, viewed as a morphism in $\Web^\times(\spn)$ and let $\Lambda$ be another reduced graph corresponding to the same matching.
Proposition \ref{prop:Matsumoto} gives a sequence of gRIII moves that take $\Gamma$ to $\Lambda$. Replacing each such move with \eqref{eq:R3} shows that $\Gamma$ and $\Lambda$ 
are equal modulo graphs with strictly fewer crossings.
\end{proof}

\subsection{Pattern avoidance}\label{sec:PA}

In the study of the symmetric group, it is common to study patterns and pattern avoidance. 
For example, $(3 \ 2 \ 1)$ is the longest permutation in $\SG_3$ (written in word notation, not cycle notation). 
A permutation $w \in \SG_\ell$ is said to fit the pattern $(3 \ 2 \ 1)$ if one can find $i < j < k$ such that $w(i) > w(j) > w(k)$.
In other words, $w$ fits the pattern $(3 \ 2 \ 1)$ if ``restricting to three particular strands'' (in the string diagram description of the symmetric group) gives the permutation $(3 \ 2 \ 1)$ in $\SG_3$. 
A permutation is called \emph{$(3 \ 2 \ 1)$-avoiding} if it does not fit the pattern $(3 \ 2 \ 1)$, i.e. if one cannot find such a triple $i < j < k$. 
More generally, we will be interested in permutations that are $(n \ n{-}1 \ \cdots \ 3 \ 2 \ 1)$-avoiding, 
i.e. permutations in $\SG_\ell$ such that we never obtain the longest element of $\SG_n$ upon restricting to a collection of $n$ strands.
We call such permutations \emph{$n$-avoiding}.
The same ideas easily extend to quadrivalent graphs.

\begin{defn}
Given a matching of $2 \ell$ points in $\partial \D$, 
an \emph{$n$-pattern} is a set of $n$ pairs $x_1 = \{a_1,b_1\} , \ldots , x_n = \{a_n,b_n\}$ of distinct points $\{a_i,b_i\}_{i=1}^n \subset \partial \D$
such that any two distinct pairs $\{x_i, x_j\}$ form an inversion. The matching is said to be \emph{$n$-avoiding} if it has no $n$-pattern. 
\end{defn}

Suppose that a matching of $2 \ell$ points in $\partial \D$ has an $n$-pattern.
Ignoring the points that are not involved in the $n$-pattern, we obtain a matching of $2n$ points.
Without loss of generality, we may assume that these points are equally-spaced around $\partial \D$.
If $x_i = \{a_i, b_i\}$ then, by counting inversions, there must be $n-1$ points between $a_i$ and $b_i$ in both directions (clockwise and counterclockwise).
Thus, all matched pairs are antipodal, and we can obtain a reduced graph corresponding to this matching as a small perturbation of the following $2n$-valent graph:
\begin{equation}\label{eq:DK}
\begin{tikzpicture}[scale=.75, smallnodes, anchorbase]
	\draw[densely dashed] (0,0) circle (1);
	\draw[very thick] (0,-1) node[below=-2pt]{$a_1$} to (0,1) node[above=-2pt]{$b_1$};
	\draw[very thick] (.707,-.707) node[below=-2pt,xshift=3pt]{$a_2$} to (-.707,.707) node[above=-2pt,xshift=-3pt]{$b_2$};
	\draw[very thick] (-.707,-.707) node[below=-2pt,xshift=-3pt]{$b_n$} to (.707,.707) node[above=-2pt,xshift=3pt]{$a_n$};
	\node at (.5,-.125){$\vdots$};
	\node at (-.5,-.125){$\vdots$};	
\end{tikzpicture}
\end{equation}
Note that, up to renaming the points and reordering the pairs, we can assume that a counterclockwise reading of boundary points is $a_1 a_2 \cdots a_n b_1 b_2 \cdots b_n$, 
as in \eqref{eq:DK}.
One such reduced graph corresponding to this matching is the half twist $\HT_n$, placed in the disk so that the bottom boundary points in \eqref{eq:HTdef}
when read left-to-right correspond to the points $a_1,\ldots,a_n$ (and the top boundary points when read left-to-right therefore correspond to $b_n,\ldots,b_1$.

We return briefly to the language of trivalent graphs.
Recall that one of the defining relations in $\Web(\spn)$ is that any strand with label $>n$ is zero. 
By Lemma \ref{lem:HT}, $\HT_{n+1}$ is equal to a (non-quadrivalent) diagram that has a strand with label $n+1$, 
modulo quadrivalent graphs with fewer than $\binom{n+1}{2}$ crossings. 
Consequently, in $\Web(\spn)$, $\HT_{n+1}$ is actually in the span of quadrivalent graphs with fewer crossings! 
Similarly, we might hope that any quadrivalent graph with an $(n+1)$-pattern 
should be in the span of graphs with fewer crossings, 
as this would imply that graphs whose corresponding matchings are $(n+1)$-avoiding form a spanning set for $\Hom$-spaces in $\Web(\spn)^{\times}$. 
To prove this, we must express every graph whose corresponding matching contains an $(n+1)$-pattern in terms of a graph containing $\HT_{n+1}$ 
(modulo terms not containing an $(n+1)$-pattern).
We now prove this. To avoid cumbersome indexing, we consider $n$-patterns, rather than $(n+1)$-patterns.

\begin{prop} \label{prop:patterngivesHT} 
Let $x_1, \ldots, x_n$ be an $n$-pattern in a matching of $2 \ell$ points in $\partial \D$, 
then there exists a reduced graph $\Gamma$ corresponding to this matching, and a sub-disk $\D' \subset \D$, 
such that $\D'$ only intersects the strands involved in the $n$-pattern, and moreover $\Gamma \cap \D' = \HT_n$.
\end{prop}

\begin{proof}
Without loss of generality, we may assume that the points involved in the $n$-pattern are equally spaced around $\partial \D$.
Connect each pair of antipodal points via chords as in \eqref{eq:DK} that necessarily pass through the center of $\D$.
Without loss of generality, we may assume that the remaining points in the matching are in sufficiently generic position
(in particular, no pair of them is antipodal).
Draw chords between these remaining points in the matching and note that, since the points were chosen sufficiently generically, 
all of the points of intersection involving these added chords are generic (quadrivalent).
In particular, none of these chords pass through the center of $\D$.
Hence, there is a sufficiently small disk $\D'$ centered at the origin that only meets the chords involved in the $n$-pattern.
Replace the $2n$-valent vertex (i.e. the graph \eqref{eq:DK}) at the center of $\D'$ with $\HT_n$. 
By construction, the resulting graph $\Gamma$ is reduced.
\end{proof}

\begin{rem}
The analogous statement holds in $\SG_k$: any permutation fitting the pattern $(n \ \cdots \ 3 \ 2 \ 1)$ possesses a reduced expression 
wherein all the $\binom{n}{2}$ crossings between the relevant $n$ strands occur consecutively. 
This can be deduced using essentially the same proof as Proposition \ref{prop:patterngivesHT}. 
This fact also could be derived from the structure of packet flips in the higher Bruhat orders of Manin-Schechtmann \cite{MS}. \end{rem}

\begin{thm} \label{thm:avoidingspans} 
Any reduced graph $\Gamma$ in $\Web^\times(\spn)$ whose matching contains an $(n+1)$-pattern is in the span of graphs with fewer crossings.
Thus, $\Hom$-spaces in $\Web^\times(\spn)$ are spanned by reduced graphs, one for each $(n+1)$-avoiding matching. 
\end{thm}

\begin{proof} 
By Lemma \ref{lem:changerex}, $\Gamma$ is equal to any other reduced graph $\Gamma'$ corresponding to the same matching, 
modulo the span of graphs with fewer crossings.
By Proposition \ref{prop:patterngivesHT}, we can assume that $\Gamma'$ has $\HT_{n+1}$ as a subgraph. 
By Lemma \ref{lem:HT}, and the fact that any web containing a strand labeled $n+1$ is zero in $\Web(\spn)$, 
$\HT_{n+1}$ is in the span of quadrivalent graphs with strictly fewer crossings.
It follows that the same is true for $\Gamma'$, which proves the first statement of the theorem.

For the second statement, choose a set $\mathcal{S}$ containing one reduced graph for each matching, 
and let $\mathcal{A}\subset \mathcal{S}$ be the subset consisting of graphs whose matching is $(n+1)$-avoiding.
Let $\mathcal{S}_{\le N}$ and $\mathcal{A}_{\le N}$ denote the span in $\Web^{\times}(\spn)$ of graphs 
in $\mathcal{S}$ and $\mathcal{A}$ with $\le N$ crossings, respectively.
By Theorem \ref{thm:reducedgraphsspan}, it suffices to show that $\mathcal{S}_{\le N} = \mathcal{A}_{\le N}$.
We prove this by induction on $N$.
The base case $N=0$ is trivial. 
Let $\Gamma \in \mathcal{S}_{\le N}$ contain an $(n+1)$ pattern. 
By the first statement in the theorem, and Theorem \ref{thm:reducedgraphsspan}, $\Gamma \in \mathcal{S}_{< N}$. 
Thus, by induction, we have $\Gamma \in \mathcal{A}_{< N}\subset \mathcal{A}_{\le N}$. 
Repeating this for all $\Gamma \in \mathcal{S}$ with $N$ crossings containing an $(n+1)$ pattern, 
we deduce that $\mathcal{S}_{\le N} = \mathcal{A}_{\le N}$. 

\end{proof}

\subsection{Faithfulness of the functor}\label{sec:faithful}

The last major result needed for our proof of Theorem \ref{thm:main} is the following.

\begin{thm}\label{thm:faithful} 
The functor $\Phi \colon \Web(\spn)\rightarrow \FRep(U_q(\spn))$ is faithful.
\end{thm}

\begin{proof} 
By Corollary \ref{cor:reduction}, it suffices to show that the functor
\[
\Phi|_{\SWeb(\spn)} \colon \SWeb(\spn) \rightarrow \SRep(U_q(\spn))
\]
is faithful, i.e. that for any $k_1,k_2 \in \N$ the $\C(q)$-linear map
\begin{equation}\label{eq:linmap}
\Hom_{\SWeb(\spn)}(1^{\otimes k_1},1^{\otimes k_2}) \to \Hom_{U_q(\spn)}(V_1^{\otimes k_1},V_1^{\otimes k_2})
\end{equation}
is injective. By Theorem \ref{thm:Full}, this map is surjective. 
To show \eqref{eq:linmap} is injective, it thus suffices to prove that 
\[
\dim\big( \Hom_{\SWeb(\spn)}(1^{\otimes k_1},1^{\otimes k_2}) \big)
\leq
\dim\big( \Hom_{U_q(\spn)}(V_1^{\otimes k_1},V_1^{\otimes k_2}) \big).
\]

Recall from Corollary \ref{cor:QuadFull} that $\SWeb(\spn) = \Web^\times(\spn)$. 
Since every $1$-manifold has an even number of boundary points, 
it immediately follows that the domain of \eqref{eq:linmap} is zero when $k_1+k_2$ is odd.
Now suppose that $k_1+k_2$ is even.
In this case, Theorem \ref{thm:avoidingspans} shows that $\Hom_{\SWeb(\spn)}(1^{\otimes k_1},1^{\otimes k_2})$ is 
spanned by a set of quadrivalent graphs in bijection with the set of $(n+1)$-avoiding matchings of $k_1+k_2$ points.
Results of Sundaram \cite{SundaramThesis,SundaramTableaux}, which we summarize in Proposition \ref{prop:dimcount} below, 
show that the dimension of $\Hom_{U_q(\spn)}(V_1^{\otimes k_1},V_1^{\otimes k_2})$ exactly equals the cardinality of this latter set.
\end{proof}

We now record the requisite result concerning the dimension of $\Hom_{U_q(\spn)}(V_1^{\otimes k_1},V_1^{\otimes k_2})$. 
Using adjunction, this space is identified with $\Hom_{U_q(\spn)}(V_0, V_1^{\ot (k_1+k_2)})$.

\begin{prop} \label{prop:dimcount} 
The dimension of $\Hom_{U_q(\spn)}(V_0, V_1^{\ot 2\ell})$ is equal to the number of $(n+1)$-avoiding matchings on $2 \ell$ points. 
\end{prop}

For all practical purposes, this is a theorem due to Sundaram.
Indeed, it is cited as such in various places in the literature (see e.g. \cite[first paragraph of \S 8.4]{Kup} and \cite[bottom of p. 2]{RubWes}) 
with citations given either to Sundaram's paper \cite{SundaramTableaux} or her thesis \cite{SundaramThesis}. 
As far as we can tell, the result is not given in this particular form in either work, though it follows (with some interpretation) from the results in Sundaram's thesis. 
Reproducing all of the technology used by Sundaram to give a thorough re-proof of Proposition \ref{prop:dimcount} would add considerable length to this paper, and seems unnecessary. 
In its place, we will give a summary of (and commentary on) the relevant parts of Sundaram's thesis.

To motivate the discussion, first consider the analogous problem in type $A$. 
Let $V$ denote the standard representation of $\gln$. 
When $n \ge m$, the Pieri rule implies that the direct summands of $V^{\ot m}$ correspond to standard Young tableaux with $m$ boxes.
Here, we view a standard tableau as a sequence of partitions of length $m$, each one obtained from the previous by adding one box. 
However, when $n < m$, one must place a further restriction on these tableaux to obtain the appropriate decomposition of $V^{\ot m}$: 
that they have at most $n$ rows.
In either case, if $V_\lambda$ is the irreducible representation of $\gln$ corresponding to the partition $\lambda$, 
we can compute the dimension $\Hom_{\gln}(V_\lambda,V^{\ot m})$ by counting the number of standard Young tableaux of shape $\lambda$.

Now consider type $C$. 
Again, we let $V$ denote the standard representation, now of $\spn$.
The type $C$ Pieri rule implies that the direct summands of $V^{\ot m}$ correspond to \emph{up-down tableaux} of length $m$. 
Here, an up-down tableau of length $m$ is a sequence of $m$ partitions, each of which is obtained from the previous partition by adding or removing one box. 
Again, when $n < m$, one must impose the restriction that the partitions appearing in the up-down tableau never have more than $n$ rows; 
such an up-down tableau is called \emph{$n$-symplectic}. 
As above, the dimension of $\Hom_{U_q(\spn)}(V_0, V_1^{\ot 2\ell})$ is thus equal to the number of $n$-symplectic up-down tableaux
of length $2\ell$ whose final shape is the empty partition.

In \cite[\S 8]{SundaramThesis}, 
Sundaram introduces up-down tableaux, and then proceeds to study many different combinatorial encodings of the same objects. 
For the sake of simplicity, we focus on the special cases of interest:
\begin{itemize}
\item up-down tableaux of length $2 \ell$ whose final shape is the empty partition,
\item matchings of $2 \ell$ points on the planar disk\footnote{Sundaram calls these \emph{arc diagrams}.}, and
\item certain two-row arrays that appear in the Burge correspondence (see \cite[Definition 3.29]{SundaramThesis}).
\end{itemize}
In \cite[Proof of Lemma 8.3]{SundaramThesis} Sundaram provides an explicit bijection, analogous to the Schensted bumping algorithm in type $A$, 
between up-down tableaux and matchings, 
and in \cite[Lemma 8.7]{SundaramThesis} she gives an explicit bijection between up-down tableaux and two-row 
arrays\footnote{In \cite[Lemma 8.7]{SundaramThesis} the bijection also appears to involve a standard tableau $Q_{\mu}$ of shape $\mu$, 
but we have specialized to the case when $\mu$ is the empty partition.}.

In \cite[\S 9]{SundaramThesis} Sundaram introduces $n$-symplectic up-down tableaux, 
and in \cite[Lemma 9.3]{SundaramThesis} she gives the explicit constraint on two-row arrays that corresponds to the partitions appearing in the tableaux having at most $n$ rows.
This is called the \emph{$n$-symplectic condition}.
Sundaram does not provide an explicit constraint on planar matchings corresponding to the $n$-symplectic condition, however. 
Had it been stated that an up-down tableau contains a partition with $k$ rows if and only if the corresponding planar matching has a $(k)$-pattern, 
then the $n$-symplectic condition would correspond to $(n+1)$-avoidance, and Proposition \ref{prop:dimcount} would be an immediate corollary.

In fact, under Sundaram's bijection, it is false that having $k$ rows in an up-down tableau corresponds to the existence of a $(k)$-pattern in the planar matching! 
In \cite[Example 8.4]{SundaramThesis}, one can see an up-down tableau containing a partition with $3$ rows, but the matching has no $3$-pattern. 
Instead, there is a modification of Sundaram's bijection\footnote{We plan to write a short note about this variant of Sundaram's bijection, for posterity.} with the desired property. 
After \cite[Example 8.4]{SundaramThesis}, Sundaram labels the ``right'' boundary points of a planar matching (on $2\ell$ points) with the symbols $\bar{1}$ through $\bar{\ell}$, 
reading from left to right. If instead one labels these points with $\bar{1}$ through $\bar{\ell}$ reading from right to left, 
one can repeat the rest of the construction to obtain the desired up-down tableau. 
This alternate bijection is even easier, as it does not require the use of jeu-de-taquin. 

In \cite[\S 10]{SundaramThesis} and \cite[Theorem 3.11]{SundaramTableaux} 
Sundaram provides an explicit link between her bijections involving $n$-symplectic up-down tableaux and the usual Schensted bumping algorithm.
For the alternate bijection, the link to Robinson-Schensted is even more direct: it is simply the restriction of this bijection to the 
subset of matchings that correspond to permutations.
Schensted's original work \cite{Schensted} on the bumping algorithm shows that the number of rows in the partition corresponding to a permutation $w \in \SG_m$ is equal
to the length of the longest decreasing subsequence in the one-row notation for $w$. 
In particular, the partition corresponding to $w$ has at most $k$ rows if and only if $w$ is $(k+1)$-avoiding.
A direct extension of this shows that the $n$-symplectic condition on up-down tableaux corresponds to $(n+1)$-avoidance 
for this alternate bijection.

\subsection{Proof of our main result}

Finally, we assemble our various results to prove Theorem \ref{thm:main}.

\begin{proof}[Proof of Theorem \ref{thm:main}.]
Theorem \ref{thm:functor} gives a $\C(q)$-linear, pivotal, essentially surjective functor 
\[
\Phi \colon \Web(\spn) \to \FRep(U_q(\spn)).
\] 
Theorems \ref{thm:Full} and \ref{thm:faithful} show that this functor is fully faithful, 
hence an equivalence of $\C(q)$-linear pivotal categories.

Now, recall that $\FRep(U_q(\spn))$ is a ribbon category, 
i.e. a braided pivotal category wherein the twist satisfies a certain equation.
See e.g. \cite[Lemma 4.27]{Sel1}.
We can use the equivalence $\Phi$ to define a braiding on $\Web(\spn)$. 
The braiding 
$
\beta_{1,1} \in \End_{\Web}(1\otimes 1) 
$
and its inverse $\beta_{1,1}^{-1}$
are given by the morphisms in \eqref{eq:braiding}, respectively.
Naturality of the braiding in $\FRep(U_q(\spn))$ shows that if we set
\[
\beta_{k,l} := 
\frac{1}{[k]![l]!}
\begin{tikzpicture}[scale=.5,smallnodes,anchorbase]
	\draw[very thick] (2.5,-1) node[below=-1pt]{$l$} to (2.5,-.5) to [out=150,in=270] (2,0) to [out=90,in=270] (0,3) node[right=-2pt]{$\cdots$} to [out=90,in=210] (.5,3.5);
	\draw[very thick] (2.5,-.5) to [out=30,in=270] (3,0) node[left=-3pt]{$\cdots$} to [out=90,in=270] (1,3) to [out=90,in=330] (.5,3.5) to (.5,4) node[above=-3pt]{$l$};
	\draw[overcross] (0,0) to [out=90,in=270] (2,3);
	\draw[overcross] (1,0) to [out=90,in=270] (3,3);
	\draw[very thick] (.5,-1) node[below=-1pt]{$k$} to (.5,-.5) to [out=150,in=270] (0,0) to [out=90,in=270] (2,3) node[right=-2pt]{$\cdots$} to [out=90,in=210] (2.5,3.5);
	\draw[very thick] (.5,-.5) to [out=30,in=270] (1,0) node[left=-3pt]{$\cdots$} to [out=90,in=270] (3,3) to [out=90,in=330] (2.5,3.5) to (2.5,4) node[above=-3pt]{$k$};
\end{tikzpicture}
\quad \text{and} \quad
\beta_{\vec{k},\vec{l}} :=
\begin{tikzpicture}[scale=.5,smallnodes,anchorbase]
	\draw[very thick] (3,0) node[below=-1pt]{$l_1$} to [out=90,in=270] (0,3) node[above=-2pt]{$l_1$};
	\node at (4,-.375) {$\cdots$}; 
	\node at (4,3.375) {$\cdots$};
	\draw[very thick] (5,0) node[below=-1pt]{$l_s$} to [out=90,in=270] (2,3) node[above=-2pt]{$l_s$};
	\draw[overcross] (0,0) to [out=90,in=270] (3,3);
	\draw[very thick] (0,0) node[below=-1pt]{$k_1$} to [out=90,in=270] (3,3) node[above=-2pt]{$k_1$};
	\node at (1,-.375) {$\cdots$};
	\node at (1,3.375) {$\cdots$};
	\draw[overcross] (2,0) to [out=90,in=270] (5,3);
	\draw[very thick] (2,0) node[below=-1pt]{$k_r$} to [out=90,in=270] (5,3) node[above=-2pt]{$k_r$};
\end{tikzpicture}
\]
for $k,l \geq 1$ and $\vec{k}= (k_1,\ldots,k_r)$ and $\vec{l}=(l_1,\ldots,l_s)$,
then $\Phi$ is an equivalence of braided pivotal (hence ribbon) categories.
\end{proof}

\bibliographystyle{plain}

%

%
\end{document}